\documentclass[a4paper,11pt]{article}
\usepackage[english]{babel}
\usepackage{mathrsfs}
\usepackage{makeidx}
\usepackage{amsfonts}
\usepackage[centertags]{amsmath}
\usepackage{amsthm}
\usepackage{enumitem}
\usepackage{latexsym,amsmath,amssymb}
\usepackage{graphicx,color}
\newcommand{\ds}{\displaystyle}

\usepackage{comment}

\setenumerate[1]{label=(\roman*), ref=(\roman*)}
\usepackage[all]{xy}% Para los diagramas.

\newtheorem{thm}{Theorem}[section]
\newtheorem{prop}[thm]{Proposition}
\newtheorem{definition}[thm]{Definition}
\newtheorem{lemma}[thm]{Lemma}

\newcommand{\Flder}{\rightarrow}
\newcommand{\der}{\partial}
\newtheoremstyle{obs}% name
  {3pt}%      Space above
  {3pt}%      Space below
  {}%         Body font
  {}%         Indent amount (empty = no indent, \parindent = para indent)
  {\bfseries}% Thm head font
  {.}%        Punctuation after thm head
  {.5em}%     Space after thm head: " " = normal interword space;
  {}%         Thm head spec (can be left empty, meaning `normal')
\theoremstyle{obs}
\newtheorem{remark}[thm]{Remark}
\newtheorem{defn}[thm]{Definition}
\newtheorem{ex}[thm]{Example}

\newcommand{\R}{\mathbb{R}}      %Numeros reales
\newcommand{\cua}{^{2}}
\newcommand{\bra}{\langle}
\newcommand{\ket}{\rangle}

\newcommand{\lp}{\left(}
\newcommand{\rp}{\right)}
\newcommand{\lc}{\left\{}
\newcommand{\rc}{\right\}}

\def\qed{\ifvmode\removelastskip\fi
{\unskip\nobreak\hfil\penalty50\hbox{}\nobreak\hfil \hbox{\vrule
height1.2ex width1.2ex}\parfillskip=0pt \finalhyphendemerits=0
\par \smallskip}}
\title{Isotropic submanifolds and the inverse problem for mechanical constrained systems}

\author{\textsc{Mar\'ia Barbero-Li\~n\'an}\thanks{mbarbero@math.uc3m.es}\\
\small
Departamento de Matem\'aticas, Universidad Carlos III de Madrid, \\ \small Avenida de la Universidad 30, 28911 Legan\'es, Madrid, Spain
\\ \small and Instituto de Ciencias Matem\'aticas (CSIC-UAM-UC3M-UCM)\\ \and
\textsc{Marta Farr\'e Puiggal\'i \thanks{marta.farre@icmat.es},}
\textsc{David Mart\'{\i}n de Diego}\thanks{david.martin@icmat.es} \\
\small
Instituto de Ciencias Matem\'aticas (CSIC-UAM-UC3M-UCM) \\ \small  C/Nicol\'as
Cabrera 13-15, 28049 Madrid, Spain
}
\makeindex
\begin{document}

\maketitle

\begin{abstract}
The inverse problem of the calculus of variations consists in determining if the solutions of  a given system
of second order differential equations correspond with the solutions of the Euler-Lagrange equations for some regular Lagrangian.
 This problem in the general version remains unsolved. Here, we contribute to it with a novel description
in terms of Lagrangian submanifolds of a symplectic manifold, also valid under some adaptation for the non-autonomous version.
One of the advantages of this new point of view is that we can easily extend our description to the study of the  inverse problem of the calculus of variations for second order systems along submanifolds. In this case, instead of  Lagrangian submanifolds we will use  isotropic submanifolds,
covering both the nonholonomic and holonomic constraints for autonomous and non-autonomous systems as particular examples.  
Moreover, we use symplectic techniques to extend these isotropic submanifolds to Lagrangian ones, 
allowing us to describe the constrained solutions as solutions of a variational problem now without constraints. 
Mechanical examples
 such as the rolling disk are provided to illustrate the main results.

\vspace{3mm}

\textbf{Keywords:}  Inverse problem, Lagrangian and isotropic submanifolds, constrained variational calculus, nonholonomic systems, hamiltonization

\vspace{3mm}

\textbf{2010 Mathematics Subject Classification:} 49N45; 70F25; 58E30; 53D12
\end{abstract}

\newpage
\tableofcontents
\newpage
\section{Introduction}
The inverse problem of the calculus of variations studies when a given  system of second order ordinary differential equations (SODE)
\begin{equation*}
\ddot{q}^i=\Gamma^i(t, q^{j},\dot{q}^{j}), \quad i,j=1,\ldots,n
\end{equation*}
is related to Euler-Lagrange equations
\begin{equation*}
\dfrac{{\rm d} }{{\rm d} t} \, \dfrac{\partial L}{\partial \dot{q}^i}-\dfrac{\partial L}{\partial q^i}=0,
\end{equation*}
for a regular Lagrangian to be determined.
To prove the equivalence of these two systems is the same as to find a non-singular matrix $(g_{ij})$ such that the following system is satisfied
\begin{equation*}
g_{ij} (\ddot{q}^j-\Gamma^j(t, q,\dot{q}))=\dfrac{{\rm d} }{{\rm d} t} \, \dfrac{\partial L}{\partial \dot{q}^i}-\dfrac{\partial L}{\partial q^i}\, .
\end{equation*} 
When such a matrix exists, the system of second order ordinary differential equations is called \textit{variational}. In 1886
Sonin~\cite{Sonin1886} proved that a single second order ordinary differential equation is always variational. This problem was also studied in 1887 by Helmholtz~\cite{Helmholtz1887} for general systems of second order ordinary differential equations in implicit form. 

History has shown this is an extremely difficult problem because only the full solution for at most
two dimensional systems of second order ordinary differential equations is known~\cite{Douglas}. Douglas' solution consists of an exhaustive classification
in different cases using Riquier-Janet theory. Variational and nonvariational SODE's are included in his work. The techniques used by
Douglas turned out to be very difficult to generalize to higher dimension.
% Henneaux studied the problem for only one independent variable~\cite{82Henneaux}.

Since 1980, the inverse  problem has been considered by many authors ~\cite{81Crampin, 82Henneaux,1990MFLMR,82Sarlet,79Takens}  giving a geometric interpretation of 
Douglas' classification and generalizing some of the results to higher dimensions. % There the property of being locally variational must be connected to a global Lagrangian function.
In particular, a free coordinate characterization of the inverse problem is given in~\cite{94CSMBP}. As a result, it has been proved that
cases I and IIa1 in Douglas~\cite{Douglas} are always variational for arbitrary dimension~\cite{SCM1998} and~\cite{CPST1999}, respectively. Case I was also proven by \cite{AT1992} and \cite{GM1999}  using different approaches. 
Other extensions of the inverse problem include partial differential equations~\cite{Anderson}, field theory~\cite{84Henneaux},  nonholonomic mechanics~\cite{2012Rossi}, driven SODE's~\cite{91IbortMarin}, jet bundles~\cite{2008KP}, etc.

In our paper we will follow a symplectic approach working with Lagrangian submanifolds of symplectic manifolds \cite{Weinstein}
associated to the geometry of the tangent bundle, which is the space where a SODE is geometrically defined.
In terms of the closedness of a suitable 1-form, constructed from the given SODE and a transformation between the tangent 
bundle and its dual, the cotangent bundle, we provide a new characterization of a variational second order differential equation. 
The use of other distinguished submanifolds of symplectic manifolds, isotropic submanifolds, turns out to be suitable 
to characterize the inverse problem for constrained variational calculus. 
Moreover, using a standard construction in symplectic geometry we can extend these isotropic submanifolds to 
Lagrangian ones, allowing us to describe the constrained solutions as solutions of a variational problem now without constraints 
such that the solutions of the new variational problem with initial conditions verifying the constraints are 
precisely real solutions of the original constrained system.   
Our techniques are also related to classical results about the comparison of solutions of nonholonomic systems and 
constrained variational problems (see \cite{93BlochCrouch, 2008BlochFer,2002Cortes} and references therein).

The paper is organized as follows: Section~\ref{Sec:Geometric}
 contains all the basic background on symplectic manifolds and tangent bundle geometry
necessary for this work (see also~\cite{GuiStern,LiMarle,TuH,Tu}). In Section~\ref{Sec:Lagrange}
we introduce some relevant examples of constrained Lagrangian systems: nonholonomic systems and constrained variational systems. 
In Section~\ref{SInverse}
we briefly describe the inverse problem of the calculus of
variations from the
geometric approach given in~\cite{81Crampin}. Then the new geometric characterization of the inverse problem is introduced:
a system of second order differential equations on a manifold $Q$ is variational if it can be associated to a
Lagrangian submanifold of the symplectic manifold $(T^*TQ,\omega_{TQ})$, where $\omega_{T^Q}$ is the natural symplectic structure 
of $T^*TQ$. We relate our results to the so-called Chaplygin hamiltonization for a special type of 
nonholonomic system \cite{2009BlochFerMestdag}. The time-dependent case is also included by using
the notion of Lagrangian submanifold of a Poisson manifold~\cite{Vaisman2}. The problem for constrained variational calculus, 
in particular,
the nonholonomic mechanics, is described in Section~\ref{constraints-section} by linking the notion of being variational to isotropic submanifolds of 
$(T^*TQ,\omega_{TQ})$.
The rolling disk is considered as an example and regular and singular
Lagrangians associated to it are given. Section~\ref{holonomic} focuses on holonomic dynamics where the system evolves on a submanifold
$TN$ of $TQ$. In some cases it is easier to study the problem in the manifold with greater dimension, instead of working on
$TN$ as if there were no constraints. With this geometric approach the typical Lagrangian functions considered when there are constraints on
$Q$ are recovered. The time-dependent case for constrained variational calculus is also characterized by using the notion of isotropic
submanifolds of Poisson manifolds~\cite{Vaisman2} in Section~\ref{Sec:TimeConstrained}. Finally, some future research lines are discussed. 
Appendix~\ref{Appendix} carefully shows in local coordinates the
equivalence between the geometric Helmholtz conditions in~\cite{81Crampin} and the conditions in our paper.

\section{Geometric preliminaries}\label{Sec:Geometric}

In this section we briefly introduce all the definitions and results from differential geometry, in particular symplectic geometry,
that are necessary in the sequel. More details can be found in~\cite{AbMa} and~\cite{LiMarle}. 

In this paper, $TQ$ and $T^*Q$ are the tangent and cotangent bundle of a manifold $Q$, respectively. The set of vector fields on $Q$ is
denoted by $\mathfrak{X}(Q)$ and the set of $k$-forms on $Q$ is denoted by $\Lambda^k(Q)$.

\subsection{Isotropic and Lagrangian submanifolds}\label{Sec:Lagrang}

The two main elements in this work are introduced here: Lagrangian and isotropic submanifolds. The former are the extension to manifolds of the notion of Lagrangian subspaces of
symplectic vector spaces~\cite{Weinstein}.

Let us recall that a symplectic vector space is a pair $(E,\Omega)$ where $E$ is a vector space and $\Omega\colon E\times E \rightarrow \mathbb{R}$ is a skew-symmetric bilinear map of maximal rank. See~\cite{GuiStern,1989LeRo,LiMarle,Weinstein} for more details.

\begin{defn} Let $(E,\Omega)$ be a symplectic vector space and
$F\subset E$ a subspace. The $\Omega$-\textbf{orthogonal complement
of $F$} is the subspace defined by
\begin{equation*}
F^\perp=\{e\in E \; | \; \Omega(e,e')=0 \; \mbox{for all } e'\in
F\}.
\end{equation*}
The subspace $F$ is said to be
\begin{enumerate}
\item \textbf{isotropic} if $F\subseteq F^\perp$, that is,
$\Omega(e,e')=0$ for all $e,e'\in F$.
\item \textbf{Lagrangian} if $F$ is isotropic and has an
isotropic complement, that is, $E=F\oplus F'$, where $F'$ is
isotropic.
\end{enumerate} \label{defn:lagrangianSubspace}
\end{defn}

A well-known characterization of Lagrangian subspaces of finite dimensional symplectic vector spaces is summarized in the following result:

\begin{prop} Let $(E,\Omega)$ be a finite dimensional symplectic vector space and
$F\subset E$ a subspace. The following assertions are
equivalent:
\begin{enumerate}
\item $F$ is Lagrangian,
\item $F=F^\perp$,
\item $F$ is isotropic and ${\rm dim} \, F=\frac{1}{2}{\rm dim}\,
E$.
\end{enumerate}\label{prop:LagrangianSubmanifoldIsotropic}
\end{prop}

As a consequence, we can characterize a Lagrangian subspace $F$ of $(E,\Omega)$ by checking if it has half the dimension of $E$ and 
if the restriction of $\Omega$ to $F$ vanishes, that is, $\Omega_{|F}=0$.

A symplectic manifold
$(M,\omega)$ is defined by a differentiable manifold $M$ and a non-degenerate closed 2-form $\omega$ on $M$. Therefore, for each 
$x\in M$, $(T_xM, \omega_x)$ is a symplectic vector space. A symplectic manifold has even dimension.

The notion of Lagrangian subspace can be transferred to submanifolds by requiring that the tangent space of the
submanifold is a Lagrangian subspace for every point in the submanifold of a symplectic manifold.

\begin{defn} Let $(M,\omega)$ be a symplectic manifold, ${\rm
i}\colon N\rightarrow M$ be an immersion and ${\rm T}_{{\rm i}(x)}\colon T_xN \rightarrow T_{{\rm i}(x)} M$ be the tangent
map of ${\rm i}$. It is said that $N$ is an
\textbf{isotropic immersed submanifold} of $(M,\omega)$ if $({\rm
T}_x{\rm i})({\rm T}_xN) \subset {\rm T}_{{\rm i}(x)}M$ is an isotropic subspace
for each $x\in N$.  A submanifold $N\subset M$ is called
\textbf{Lagrangian} if it is isotropic and there is an isotropic
subbundle $P\subset TM|_{N}$ such that $TM|_{N}=TN\oplus
P$.\label{defn:lagrangianSubm}
\end{defn}

Note that ${\rm i}\colon N\rightarrow M$ is isotropic if and only if
${\rm i}^*\omega=0$, that is, $\omega({\rm T}_x{\rm i} (v_x), {\rm T}_x{\rm i} (u_x))=0$ for every $u_x,v_x\in T_xN$ and for every $x\in N$.

%
% \begin{prop} Let $(P,\omega)$ be a symplectic manifold and $L\subset
% P$ a submanifold. Then $L$ is Lagrangian if and only if $L$ is
% isotropic and ${\rm dim} L=\frac{1}{2}{\rm dim} P$.
% \label{prop:LagrangianSubmanifoldIsotropic}
% \end{prop}

The canonical model of symplectic manifold is the cotangent bundle $T^*Q$ of an arbitrary manifold $Q$ which is the dual bundle of 
$\tau_Q: TQ\rightarrow Q$. Denote by $\pi_Q\colon T^*Q \rightarrow Q$ the canonical projection and define a canonical 1-form $\theta_Q$ on $T^* Q$ by
\begin{equation}\label{eq:ThetaQ}
 \left(\theta_Q\right)_{\alpha_q}(X_{\alpha_q})=\langle \alpha_q, {\rm T}_{\alpha_q} \pi_Q(X_{\alpha_q})\rangle,
\end{equation}
where $X_{\alpha_q}\in T_{\alpha_q}T^*Q$, $\alpha_q\in T^*Q$ and $q\in Q$. If we consider
bundle coordinates $(q^i,p_i)$ on $T^* Q$ such that $\pi_Q(q^i,p_i)=q^i$, then
\begin{equation*}\theta_Q=p_i  {\rm d}q^i\, .
 \end{equation*}
 The 2-form $\omega_Q=-{\rm d}\theta_Q$ is a symplectic form on $T^*Q$ with local expression
\begin{equation*}
\omega_Q={\rm d}q^i \wedge {\rm d}p_i.
\end{equation*}
The Darboux theorem states that this is the local model for an arbitrary symplectic manifold $(M,\omega)$. In other words, there always
exist local
coordinates $(q^i,p_i)$ in a neighbourhood of each point in $M$ such that $\omega={\rm d}q^i \wedge {\rm d}p_i$.

Note that the canonical 1-form $\theta_Q$ verifies that $\gamma^*(\theta_Q)=\gamma$ for an arbitrary 1-form $\gamma$ on $Q$. Hence
$\gamma^*(\omega_Q)=-{\rm d}\gamma$.

A relevant example of a Lagrangian submanifold of the cotangent bundle is the following one.

\begin{prop}[\cite{LiMarle}] Let $\gamma$ be a 1-form on $Q$ and ${\mathcal L}=\hbox{Im } \gamma\subset T^*Q$.  The submanifold ${\mathcal L}$ of $T^*Q$ is Lagrangian if and only if
$\gamma$ is closed. \label{prop:LagrangianClosedForm}
\end{prop}
The result follows because $\dim {\mathcal L}=\dim Q$ and $\gamma^*(\omega_Q)=-{\rm d}\gamma$.

A useful extension of the previous construction is the following one

\begin{prop}[\cite{GuiStern}]\label{Prop:TildeS}
Let $i:N \longrightarrow TQ$ be an immersion. For each Lagrangian submanifold $S\subset T^{*}N$ we can define a Lagrangian submanifold $\tilde{S}\subset T^{*}TQ$ by $\tilde{S}=\left\{ \mu\in T^{*}TQ: i^{*}\mu\in S \right\}$.
\end{prop}

In the above proposition, if $N$ is a submanifold and $S=\mbox{Im}(df)$ for some $f:N\longrightarrow\mathbb{R}$, then we recover the following result:
\begin{thm}[\cite{TuH},\cite{Tu}]\label{tulchi}
  \label{thm:tulczyjew}
  Let $Q$ be a smooth manifold, $\tau_Q:TQ\rightarrow  Q$ its tangent bundle projection, $ N\subset Q $ a submanifold, and $ f
  \colon N \rightarrow \mathbb{R} $.  Then
  \begin{multline*}
    \Sigma _f = \bigl\{ p \in T ^\ast Q \mid \pi _Q (p) \in N \text{
        and } \left\langle p, v \right\rangle = \left\langle
        \mathrm{d} f , v \right\rangle \\
      \text{ for all } v \in T N \subset T Q \text{ such that } \tau
      _Q(v) = \pi _Q (p) \bigr\}
  \end{multline*}
  is a Lagrangian submanifold of $ T ^\ast Q $.
\end{thm}

Given  a symplectic manifold $(M, \omega)$, $\dim M=2n$, it is well-known that its tangent bundle $TM$ is equipped with a
symplectic structure denoted by $\mathrm{d}_T\omega$, where ${\rm d}_T \omega$ denotes the tangent lift of $\omega$ to $TM$. If we
take Darboux coordinates $(q^i,p_i)$ on $M$, that is, $\omega=\mathrm{d}q^i\wedge \mathrm{d}p_i$, then
$\mathrm{d}_T\omega=\mathrm{d}\dot{q}^i\wedge \mathrm{d}p_i+\mathrm{d}q^i\wedge \mathrm{d}\dot{p}_i$, where
$(q^i, p_i, \dot{q}^i, \dot{p}_i)$ are the induced coordinates on $TM$. We will denote the bundle coordinates on $T^*M$ by
$(q^i, p_i, a_i, b^i)$, then $\omega_{M}=\mathrm{d}q^i\wedge \mathrm{d}a_i+\mathrm{d}p_i\wedge \mathrm{d}b^i$.
If we denote by $\flat_{\omega}: TM\to T^*M$ the isomorphism defined by $\omega$, that is, $\flat_{\omega}(v)={\rm i}_{v}\,\omega$, then
we have $\flat_{\omega}(q^i, p_i, \dot{q}^i, \dot{p}_i)=(q^i, p_i, -\dot{p}_i, \dot{q}^i)$. 
This isomorphim plays an important role on the description of the dynamics of Lagrangian and Hamiltonian systems as summarized in 
Section~\ref{Sec:Tulczy}
(more details  can be found in~\cite{TuH}).

Given a function $H: M\to \R$, and its associated Hamiltonian vector field $X_H$, that is, ${\rm i}_{X_H}\omega=\mathrm{d}H$,
then the image of  $X_H$, $\hbox{Im}(X_H)$, is a Lagrangian submanifold of  $(TM, \mathrm{d}_T\omega)$.

The following construction can be found in \cite{Vaisman} and will be useful in Section~\ref{SInverse}. 
Assume we have a submanifold $N$ of a symplectic manifold  $(M, \omega)$ such
    that for a neighborhood $U_{p}$ of a point $p$ in $M$ we can write
    $$
    U_{p}\cap N=\left\{x\in M \; |\; \phi_{1}(x)=0,\ldots,\phi_{k}(x)=0\right\}.
    $$
    If we have an isotropic submanifold $N_{0}\subset N$ with  $p\in N_{0}$, $\dim(N_{0})=\frac{\dim(N)-k}{2}$ and
    the Hamiltonian vector fields $X_{\phi_{1}},\ldots,X_{\phi_{k}}$ of
    $\phi_{1},\ldots,\phi_{k}$ satisfy that
    \begin{itemize}
    \item $\exists\ \epsilon >0$ such that the flows of $X_{\phi_{i}}$ are
    defined for all $|t|<\epsilon$,
    \item $X_{\phi_{i}}(p)\not\in T_{p}N_{0}$, for all $i=1,\ldots,k$
    and $p\in N_{0}$,
		\item $X_{\phi_{i}}(p)$ are linearly independent for all $p\in N_{0}$,
    \end{itemize}
    then we can extend it to a Lagrangian submanifold transporting $N_{0}$
    along the flows of the Hamiltonian vector fields
    $X_{\phi_{1}},\ldots,X_{\phi_{k}}$.

We will illustrate the construction  for the case $k=1$ and rename $\phi_1$ by $\phi$.
Since $X_{\phi}$ is transverse to $N_0$, there exists an open interval $I$ about $0$ in $\R$ such that $\hbox{exp }(tX_{\phi}(\tilde{p}))$ is defined
for all $t\in I$ and $\tilde{p}\in N_0\cap U_p$. Therefore the map
\[
\begin{array}{rrcl}
j:& N_0\times I&\longrightarrow&M\\
  &(\tilde{p}, t)&\longmapsto&\hbox{exp }(tX_{\phi}(\tilde{p}))
  \end{array}
  \]
allows us to realize locally $N_0\times I$ as a submanifold $Z$ of $M$ whose tangent space is
\[
T_{\hbox{exp }(tX_{\phi}(\tilde{p}))}Z= (\hbox{exp }(tX_{\phi}))_*(T_{\tilde{p}}N_0)\oplus \hbox{span}\left\{ X_{\phi}(\hbox{exp }(tX_{\phi}(\tilde{p})))\right\},
\]
where $(\hbox{exp }(tX_{\phi}))_*$ is the pushforward of $\hbox{exp }(tX_{\phi})$. Obviously $\dim Z=\dim N_0+1$ and $Z$ is also 
isotropic because, first, for any two vectors in $(\hbox{exp }(tX_{\phi}))_*(T_{\tilde{p}}N_0)$ we have that
\[
\omega((\hbox{exp }(tX_{\phi}))_*v_1, (\hbox{exp }(tX_{\phi}))_*v_2)=((\hbox{exp }(tX_{\phi}))^*\omega)(v_1, v_2)=\omega(v_1, v_2)=0
\]
since $(\hbox{exp }(tX_{\phi}))_*$ is a symplectomorphism and $v_1, v_2\in T_{\tilde{p}}N_0$.

Second, it must be checked that the 2-form $\omega$ also vanishes for a vector in  $(\hbox{exp }(tX_{\phi}))_*(T_{\tilde{p}}N_0)$ and one in $X_{\phi}(\hbox{exp }(tX_{\phi}(\tilde{p})))$. Note that 
\[
\omega((\hbox{exp }(tX_{\phi}))_*v, X_{\phi}(\hbox{exp }(tX_{\phi}(\tilde{p}))))=d\phi(\tilde{p})(v)=0,
\]
because $\phi$ vanishes on $N_0$ and $v\in T_{\tilde{p}}N_0$.

\subsection{Second order differential equations}
Consider the tangent bundle $\tau_Q: TQ\rightarrow Q$ where $(q^i, \dot{q}^i)$ are canonical coordinates on $TQ$ and $(q^i)$ on $Q$. In $TQ$ we can define the following geometric objects: the \emph{Liouville} or \emph{dilation vector field} $\Delta\in {\mathfrak X}(TQ)$ and a type $(1,1)$ tensor field $S$ called the \emph{vertical endomorphism}. In canonical coordinates
\[
\Delta=\dot{q}^i\frac{\partial}{\partial q^i}  \quad \hbox{and} \quad S=dq^i\otimes \frac{\partial}{\partial \dot q^i}.
\]
A \emph{SODE} (second order differential equation) $\Gamma$ is a vector field on $TQ$ satisfying $S(\Gamma)=\Delta$. In coordinates,
\[
\Gamma=\dot{q}^i\frac{\partial}{\partial q^i}+\Gamma^i(q, \dot{q})\frac{\partial}{\partial \dot{q}^i}.
\]
The solutions of the SODE $\Gamma$ are precisely the solutions of the system of second order differential equations
\[
\frac{d^2 q^i}{dt^2}=\Gamma^i(q, \dot{q}).
\]
As shown in the following section, SODE's are key elements to describe intrinsically Lagrangian mechanics.

\section{Lagrangian mechanics}\label{Sec:Lagrange}

The calculus of variations can be defined geometrically by means of SODE's.  
Consider  a curve $c:[a, b]\to Q$ of class $C^2$ connecting two fixed points  in the configuration space $Q$. The set of all these curves is denoted by 
\begin{eqnarray*}
{\mathcal C}^2(q_0, q_1, [a, b])=\left\{ c: [a, b]\subseteq {\mathbb R}\longrightarrow Q\; \Big|\; c\in C^2,\; c(a)=q_0,\; c(b)=q_1\right\}.
\end{eqnarray*}
This set is a smooth infinite dimensional manifold. Its tangent space at $c$ is given by
\begin{eqnarray*}
T_c{\mathcal C}^2(q_0, q_1, [a, b])=\left\{ X: [a, b]\longrightarrow TQ\; \Big|\ X\in C^1, X(t)\in T_{c(t)}Q \; \forall t\in[a,b] \; \hbox{ and } X(a)=X(b)=0, \right\}.
\end{eqnarray*}
Now, let $L: TQ\rightarrow \R$ be a Lagrangian function and consider the action functional 
\[
\begin{array}{rrcl}
{\mathcal J}:&{\mathcal C}^2(q_0, q_1, [a, b])&\longrightarrow&\R\\
&c&\longmapsto &\int_a^b L(\dot{c}(t))\; dt.
\end{array}
\]

\begin{definition}{\bf [Hamilton's principle]}
A curve $c\in{\mathcal C}^2(q_0, q_1, [a, b])$ is a solution of the Lagrangian system given by $L: TQ\rightarrow \R$ if and only if $c$ is a critical point of ${\mathcal J}$, that is, 
\begin{equation}\label{var}
d{\mathcal J}(c)=0\, .
\end{equation}
\end{definition}
Using standard arguments from variational calculus, it is easy to show that the solutions of the Lagrangian system given by (\ref{var})  are the solutions of the Euler-Lagrange equations for the Lagrangian $L: TQ\rightarrow \R$: 
\[
\frac{d}{dt}\left(\frac{\partial L}{\partial \dot{q}^i}\right)-\frac{\partial L}{\partial q^i}=0, \quad 1\leq i\leq n=\dim Q
\]
where $(q^i, \dot{q}^i)$ are local coordinates for $TQ$. 

Now we will derive intrinsically the Euler-Lagrange equations using the geometry of the tangent bundle.
 Given $L: TQ\rightarrow \R$  we define the \emph{Poincar\'e-Cartan 1-form}  $\Theta_L=S^*(dL)$, the associated  \emph{Poincar\'e-Cartan 2-form}  $\Omega_L=-d\Theta_L$
and the \emph{energy function} $E_L: TQ\rightarrow \R$ by $E_L=\Delta(L)- L$. Locally,
\[ \Theta_L=\frac{\partial L}{\partial \dot{q}^i}\, dq^i, \quad
E_L=\dot{q}^i\frac{\partial L}{\partial \dot{q}^i}-L\; .
\]
When the Lagrangian $L$ is regular, that is,  $\Omega_L$ is a symplectic 2-form, or locally when the $n\times n$-Hessian matrix 
$(\partial^2 L/\partial \dot{q}^i\partial \dot{q}^j)$ is regular, then there exists a unique SODE $\Gamma_L$ solution of the equation
\begin{equation}\label{aqw}
i_{\Gamma_L}\Omega_L=dE_L,
\end{equation}
or alternatively 
\begin{equation}\label{eq:Alternative}
{\mathcal{L}}_{\Gamma_L}\Theta_L=dL,
\end{equation}
where ${\mathcal{L}}_{\Gamma_L}\Theta_L$ is the Lie derivative of $\Theta_L$ along $\Gamma_L$.

The integral curves of $\Gamma_L$ are precisely the solutions to the Euler-Lagrange equations for $L$.

\subsection{Constrained Lagrangian mechanics: nonholonomic systems}\label{Sec:Nonholo}

In this section, we will see one of the main examples where second order differential equations along submanifolds arise: the case of nonholonomic Lagrangian systems. To do so, we introduce constraints to a given Lagrangian system $L: TQ\rightarrow \R$. 

\begin{definition}
A nonholonomic Lagrangian system on a manifold $Q$ consists of a pair $(L, M)$ where $L: TQ\rightarrow {\mathbb R}$ is a Lagrangian function and $M$ is a submanifold of $TQ$
\end{definition}
Let $\tau_Q: TQ\rightarrow Q$ be the canonical projection. 
In the sequel we will assume that $\tau_Q(M)=Q$ avoiding the existence of holonomic constraints (see Section~\ref{holonomic} for more details). In mechanical and real examples, $M$ is typically a vector subbundle $D$ of $\tau_Q$, that is, the constraints
 are linear on velocities; in other examples, $M$ is an affine subbundle modelled on a vector bundle $D$.
From now on,  we  assume that $M=D$ is a vector subbundle or an affine subbundle modelled on $D$.

The existence  of the constraints prescribed by $M$ induces the introduction of reaction forces which restrict the motion to $M$.  This forces are determined by the Lagrange-d'Alembert principle.

Define the set of admissible curves by
\begin{eqnarray*}
{\mathcal C}_M^2(q_0, q_1, [a, b])=\left\{ c: [a, b]\subseteq {\mathbb R}\longrightarrow Q\; \Big|\; c\in {\mathcal C}^2(q_0, q_1, [a, b]), \  \dot{c}(t)\in M_{c(t)} \; \forall t\in[a,b]\right\},
\end{eqnarray*}
and the set of possible virtual variations along $c$ by
\begin{eqnarray*}
{\mathcal V}_c=\left\{ X: [a, b]\longrightarrow TQ\; \Big|\ X\in C^1, X(t)\in D_{c(t)} \; \forall t\in[a,b]\; \hbox{ and } X(a)=X(b)=0\right\}.
\end{eqnarray*}

\begin{definition}{\bf [Lagrange-d'Alembert's principle]} Let $c\in {\mathcal C}_M^2(q_0, q_1, [a, b])$, then $c$ is a solution of the nonholonomic Lagrangian system $(L, M)$ if 
\begin{equation*}\label{usc}
\langle d{\mathcal J}(c), X\rangle=0, \hbox{ for all } X\in  {\mathcal V}_c.
\end{equation*}
 \end{definition}
Locally, if the submanifold $M$ is determined by the vanishing of constraints $\phi^{\alpha}(q^i, \dot{q}^i)=0$ (either linear or affine constraints), then the equations of motion of a nonholonomic Lagrangian system are: 
\begin{eqnarray}
\frac{d}{dt}\left(\frac{\partial L}{\partial \dot{q}^i}\right)-\frac{\partial L}{\partial q^i}&=&\lambda_{\alpha}\frac{\partial \phi^{\alpha}}{\partial \dot{q}^i}\; ,\label{aaa} \\
\phi^{\alpha}(q^i, \dot{q}^i)&=&0\; .\nonumber
\end{eqnarray}
If the constraints are written as $\phi^{\alpha}(q^i, \dot{q}^i)=\mu^{\alpha}_i(q)\dot{q}^i+\mu^{\alpha}_0(q)$, then the previous equations reduce to:
\begin{eqnarray*}
\frac{d}{dt}\left(\frac{\partial L}{\partial \dot{q}^i}\right)-\frac{\partial L}{\partial q^i}&=&\lambda_{\alpha}\mu^{\alpha}_i(q)\; ,\\
\mu^{\alpha}_i(q)\dot{q}^i+\mu^{\alpha}_0(q)&=&0\; .
\end{eqnarray*}
If the Hessian matrix $W$ of $L$ with respect to the velocities is definite, then the matrix
\[
{\mathcal C}=(C_{\alpha\beta}) \hbox{ with } C_{\alpha\beta}=\mu^{\alpha}_i W^{ij}\mu^{\beta}_j
\]
is regular, where $(W^{ij})$ is the inverse of the Hessian matrix
\[
W_{ij}=\left( \frac{\partial^2 L}{\partial \dot{q}^i\partial \dot{q}^j}
\right)\; .
\]
Observe that the definiteness  condition is automatically satisfied in  mechanics when  $L=T-V$, being $T$  the kinetic energy associated to a Riemannian metric on $Q$ and being $V$ the potential energy. 
It is easy to show that under this condition, we can write the equations of motion of a nonholonomic system as a system of explicit second order differential equations on the constraint submanifold $M$.  In fact, the Lagrange multipliers are determined univocally as
\begin{eqnarray*}
{\lambda}_{\alpha}(q, \dot{q})&=& -C_{\alpha\beta}\left(\frac{\partial \mu^{\beta}_i}{\partial q^j}\dot{q}^i\dot{q}^j+\frac{\partial \mu^{\beta}_0}{\partial q^i}\dot{q}^i+\mu^{\beta}_iW^{ij}\left[\frac{\partial L}{\partial q^j}-\frac{\partial^2 L}{\partial \dot{q}^j\partial q^k}\dot{q}^k\right]\right),
  \end{eqnarray*}
and given an initial condition on $M$, $\dot{c}(0)\in M_{c(0)}$, the unique solution of the second order differential equation 
\[
\ddot{q}^i=W^{ij}\left[\lambda_{\alpha}(q, \dot{q})\mu^{\alpha}_j(q)+\frac{\partial L}{\partial q^j}-\frac{\partial^2 L}{\partial \dot{q}^j\partial q^k}\dot{q}^k\right]
\]
evolves on the constraint submanifold $M$, that is, $\dot{c}(t)\in M_{c(t)}$.

\subsubsection{Nonholonomic Chaplygin systems}\label{chaplygin}

Let us consider a nonholonomic Lagrangian system with symmetry, that is, a nonholonomic Lagrangian system $(L, D)$ where $D$ is a vector subbundle of $TQ$ (for simplicity)  and a Lie group action $\Psi: G\times Q\to Q$, such that both $L$ and $D$ are $G$-invariant with respect to the induced action on $TG$. Consider the subclass 
of nonholonomic systems with symmetry corresponding to
\[
D_q\oplus T_q \hbox{Orb}(q)=T_qQ,
\]
known as the purely kinematical case, where the symmetry directions complement the constraints given by $D$.  Here $\hbox{Orb}(q)=\{\bar{q}\in Q\; \Big|\; \bar{q}=\Psi(g, q), \hbox{with } g\in G\}$ is the orbit of $q\in Q$. 

In the particular  case of Chaplygin systems, these  data are given by a principal $G$-bundle $\pi: Q\to Q/G$, associated with a free and proper action $\Psi: G\times Q\to Q$ such that $L$ is $G$-invariant and $D$ is determined by the horizontal distribution of a principal connection ${\mathcal A}: TQ\to {\mathfrak g}$. Remember that ${\mathcal A}(\xi_Q(q))=\xi$, where 
\[
\xi_Q(q)=\frac{d}{dt}\Big|_{t=0}\Psi(\hbox{exp }(t\xi), q), \hbox{ with } \xi\in {\mathfrak g}
\]
and ${\mathcal A}(T\Psi_g(X))=\hbox{Ad}_g({\mathcal A}(X))$, for all $X\in TQ$ where $\Psi_g(q)=\Psi(g, q)$. Observe that in this case
\[
D_q=\{ v_q\in T_qQ\; \Big|\; {\mathcal A}(v_q)=0\},
\]
that is, $D_q$ is the horizontal subspace at $q$ determined by the connection ${\mathcal A}$. 

Therefore, for any $v_q\in T_q Q$ we have a unique decomposition
\[
v_{q}=\hbox{hor}_q v_q +\hbox{ver}_q v_q\; ,
\]
where
\[
\hbox{ver}_q v_q=(\mathcal{A}(v_q))_Q(q)\; .
\]
Then
\[
\hbox{hor}_q v_q =v_q- (\mathcal{A}(v_q))_Q(q)\in D_q.
\]
The projection map $\pi: Q\to Q/G$ induces an isomorphism from $D_q$
to $T_{\pi(q)}(Q/G)$, and the inverse map is called the horizontal lift. 
Thus for any vector field $X\in {\mathfrak X}(Q/G)$ on the base space, we have a unique vector field $X^h$ (the horizontal lift of $X$) that is horizontal and $\pi$-related to $X$.

 We define the curvature ${\mathcal B}$ of ${\mathcal A}$ as the Lie algebra valued 2-form
 defined by
 \[
 {\mathcal B}(X,Y)=-{\mathcal A}[X^h, Y^h],  \qquad X, Y\in \mathfrak{X}(Q/G)\; .
 \]
  Consider a local trivialization $U\times G$ of $\pi$ where now the action of $G$ is given by left translation on the second factor and $U$ is a neighborhood of $Q/G$. Take coordinates $r^a$ on $U$ and a basis $\{e_{\alpha}\}$ of ${\mathfrak g}$. Then, any element $\xi\in {\mathfrak g}$ is written as $\xi=\xi^{\alpha}e_{\alpha}$. In this local trivialization we can write the connection ${\mathcal A}$ as follows
\[
{\mathcal A}(r, g, \dot{r}, \dot{g})=\hbox{Ad}_g (g^{-1}\dot{g}+A^{\alpha}_a \dot{r}^a e_{\alpha})\; .
\]
Similarly, the coefficients of the curvature are
\[
{\mathcal B}^{\alpha}_{ab}=\frac{\partial A^{\alpha}_a}{\partial r^b}
-\frac{\partial A^{\alpha}_b}{\partial r^a}-C^{\alpha}_{\beta\gamma}A^{\beta}_bA^{\gamma}_a\; ,
\]
where
\[
{\mathcal B}\left(\frac{\partial} {\partial r^a}, \frac{\partial} {\partial r^b}\right)={\mathcal B}^{\alpha}_{ab}e_{\alpha}.
\]

In this case, the Lagrangian $L: TQ\rightarrow \R$ induces a Lagrangian $L^*: T(Q/G)\rightarrow \R$
by 
\[
L^*(X(\bar{q}))=L(X^h(q)).
\]
Locally,
\[
L^*(r^a, \dot{r}^a)=l(r^a, \dot{r}^a, -A^{\alpha}_a \dot{r}^a e_{\alpha}),
\]
where $l: TU\times {\mathfrak g}\to \R$ represents the reduction of $L: T(U\times G)\rightarrow \R$  to $TQ/G$. 

After some computations,  we can see that the reduced dynamics are given by the following system of equations on $T(Q/G)$:
\[
\frac{d}{dt}\left(\frac{\partial L^*}{\partial \dot{r}^a}\right)-\frac{\partial L^*}{\partial r^a}=
\Lambda_a,
\]
where
\[
\Lambda_a=-\left(\frac{\partial l}{\partial \xi^{\alpha}}\right)_{\bf c} {\mathcal B}^{\alpha}_{ab}\dot{r}^b
\]
and the subindex ``${\bf c}$'' on the right-hand side indicates that, after computing the derivative of $l$ with respect to $\xi^a$, one evaluates this partial derivative on $(r^a, \dot{r}^a, -A^{\alpha}_a \dot{r}^a e_{\alpha})$. 

Moreover, if $L$ is regular, we have that $L^*$ is also regular and we obtain the following system of second-order differential equations now defined on the full space $T(Q/G)$: 
\begin{equation}\label{asr}
\frac{d^2 r^a}{dt^2}=\widehat{W}^{ab}\left(\frac{\partial L^*}{\partial r^b}-\Lambda_b\right)\; ,
\end{equation}
where $(\widehat{W}^{ab})$ is the inverse of the Hessian matrix
$\widehat{W}_{ab}=\left(\frac{\partial^2 L^*}{\partial \dot{r}^a\partial \dot{r}^b}\; 
\right)$.

\subsection{Variational constrained equations}\label{VC}

Now we study a dynamical system given by the same pair $(L, M)$  but using purely variational techniques.
As above, let us consider a regular Lagrangian $L: TQ\Flder {\mathbb R}$, and a set of  constraints $\phi^{\alpha}(q^{i},\dot q^{i})=0$, $1\leq\alpha\leq m$ that determine a $2n-m$ dimensional submanifold $M\subset TQ$. Take the extended Lagrangian $\mathcal{L}={L}+\lambda_{\alpha}\phi^{\alpha}$ which includes the Lagrange multipliers $\lambda_{\alpha}$ as new extra variables. The equations of motion for the constrained variational problem are the Euler-Lagrange equations for $\mathcal{L}$, that is:
\begin{eqnarray}\nonumber
&&\frac{\mathrm{d}}{\mathrm{d}t}\lp\frac{\der {L}}{\der\dot q^{i}}\rp-\frac{\der {L}}{\der q^{i}}=-\dot\lambda_{\alpha}\frac{\der\phi^{\alpha}}{\der\dot q^{i}}-\lambda_{\alpha}\left[\frac{\mathrm{d}}{\mathrm{d}t}\lp\frac{\der\phi^{\alpha}}{\der\dot q^{i}}\rp-\frac{\der\phi^{\alpha}}{\der q^{i}}\right],\\\label{Vako}\\\nonumber
&&\phi^{\alpha}(q^{i},\dot q^{i})=0,\,\,\,\,1\leq\alpha\leq m.
\end{eqnarray}

Observe that the equations of a variational constrained system are different from the equations of a nonholonomic system given in (\ref{aaa}).

\subsection{Lagrangian mechanics using  the Tulczyjew's triple}\label{Sec:Tulczy}

The theory of Lagrangian submanifolds gives an intrinsic geometric description of Lagrangian and Hamiltonian dynamics~\cite{TuH},~\cite{Tu}. Moreover, it allows us to relate Lagrangian and Hamiltonian formalisms using as a main tool the so-called Tulczyjew's triple
\begin{equation*}
\xymatrix{ T^*TQ && TT^*Q \ar[rr]^{\hbox{\small{$\beta_Q$}}} \ar[ll]_{\hbox{\small{$\alpha_Q$}}} && T^*T^*Q\, .}
\end{equation*}

 The Tulczyjew map $\alpha_{Q}$ is an isomorphism between $TT^*Q$ and $T^*TQ$. Besides, it is also a symplectomorphism
between these  vector bundles considered as  symplectic manifolds, i.e. $(TT^*Q\,,\,\mathrm{d}_T\,\omega_Q)$, where
$\mathrm{d}_T\,\omega_Q$ is the tangent lift of $\omega_Q$, and $(T^{*}TQ, \omega_{TQ})$. For completeness, we recall
the construction of the symplectomorphism $\alpha_Q$. To do this, it is necessary to introduce the canonical involution $\kappa_Q$ on $TTQ$
\begin{equation*}
\xymatrix{
TTQ\ar[d]_{\tau_{TQ}}\ar[r]^{\kappa_{Q}}&TTQ\ar[d]^{T\tau_{Q}}\\
TQ\ar[r]_{\mbox{Id}}&TQ,
}
\end{equation*}
defined by
$$
\kappa_{Q}\lp\frac{{\rm d}}{{\rm d}s}\Big |_{s=0}\frac{{\rm d}}{{\rm d}t}\Big |_{t=0}\hspace{1mm}\chi\lp s,t\rp\rp=
\frac{{\rm d}}{{\rm d}s}\Big |_{s=0}\frac{{\rm d}}{{\rm d}t}\Big |_{t=0}\hspace{1mm}\tilde\chi\lp s,t\rp,
$$
where $\chi:\mathbb{R}\cua\rightarrow Q$ and $\tilde\chi:\mathbb{R}\cua\rightarrow Q$ are related by $\tilde\chi\lp s,t\rp=\chi\lp t,s\rp$. If $\lp q^{i}\rp$ are the local coordinates for $Q$, $\lp q^{i},v^{i}\rp$ for $TQ$ and $\lp q^{i},v^{i},\dot q^{i},\dot v^{i} \rp$ for $TTQ$, then the canonical involution is locally given by $\kappa_{Q}\lp q^{i},v^{i},\dot q^{i},\dot v^{i}\rp=\lp q^{i},\dot q^{i},v^{i},\dot v^{i}\rp$.

In order to describe $\alpha_{Q}$ it is also necessary to define a tangent pairing. Given two manifolds $M$ and $N$, and a pairing $\bra\cdot,\cdot\ket:M\times N\rightarrow\mathbb{R}$ between them, the tangent pairing $\bra\cdot,\cdot\ket^{T}:TM\times TN\rightarrow\mathbb{R}$ is determined by
$$
\left\langle\frac{{\rm d}}{{\rm d}t}\Big |_{t=0}\hspace{1mm}\gamma\lp t\rp,\frac{{\rm d}}{{\rm d}t}\Big |_{t=0}\hspace{1mm}\delta\lp t\rp\right\rangle^{T}=\frac{{\rm d}}{{\rm d}t}\Big |_{t=0}\hspace{1mm}\langle\gamma\lp t\rp,\delta\lp t\rp\rangle
$$
where $\gamma:\mathbb{R}\rightarrow M$ and $\delta:\mathbb{R}\rightarrow N$.

Finally, we can define  $\alpha_{Q}$ as $\langle\alpha_{Q}\lp z\rp,w\rangle=\langle z,\kappa_{Q}\lp w\rp\rangle^{T}$, where $z\in TT^{*}Q$ and $w\in TTQ$. In local coordinates $\lp q^{i},p_{i}\rp$ for $T^{*}Q$ and $\lp q^{i},p_{i},\dot q^{i},\dot p_{i}\rp$ for $TT^{*}Q$, we have
$$
\alpha_{Q}\lp q^{i},p_{i},\dot q^{i},\dot p_{i}\rp=\lp q^{i},\dot q^{i},\dot p_{i},p_{i}\rp.
$$

 The  isomorphism $\beta_{Q}:TT^{*}Q\rightarrow T^{*}T^{*}Q$ is just given by $\beta_Q=\flat_{\omega_Q}$, where $\flat_{\omega_Q}$ is the
isomorphism defined by $\omega_Q$, that is, $\flat_{\omega_Q}(v)=i_v \omega_Q$.

 The Lagrangian dynamics is described by the Lagrangian submanifold ${\rm d}{ L}(TQ)$ of $T^*TQ$ where ${L}\colon TQ \rightarrow \mathbb{R}$ is the Lagrangian function, while the Hamiltonian formalism is described by the Lagrangian submanifold ${\rm d}H(T^*Q)$ of $T^*T^*Q$ where $H\colon T^*Q \rightarrow \mathbb{R}$ is the corresponding Hamiltonian energy. The solutions of the dynamics are curves $\gamma: I\subset \R\to T^*Q$ such that $\frac{\mathrm{d}\gamma}{\mathrm{d}t}:I\subset \R\to TT^*Q$ verifies that $\frac{\mathrm{d}\gamma}{\mathrm{d}t}(I)\subset \alpha_Q^{-1}\lp {d}{\rm L}(TQ)\rp$ in the Lagrangian description and $\frac{\mathrm{d}\gamma}{\mathrm{d}t}(I)\subset \beta_Q^{-1}\lp{\rm d}{ H}(T^*Q)\rp$ in the Hamiltonian case.

  Variationally constrained problems described in Section \ref{VC} are determined by a pair $(M, l)$ where $M$ is a submanifold of $TQ$, with inclusion $i_M: M\hookrightarrow TQ$, and $l: M\to \R$ is a Lagrangian function restricted to $M$. The submanifold  $\Sigma_l$ is a Lagrangian submanifold of $(T^*TQ, \omega_{Q})$ (see Theorem 2.5). Now using the Tulczyjew's symplectomorphism $\alpha_Q$, we induce a new Lagrangian submanifold
$\alpha_Q^{-1}\lp\Sigma_l\rp$ of $(TT^*Q, \mathrm{d}_T\omega_Q)$, which completely determines the constrained variational dynamics.
Now we will see that this procedure gives the correct equations for the constrained variational dynamics.
Take an arbitrary extension $L:TQ\to\R$ of $l: M\to \R$, that is, $ L\circ i_M=l$. Locally, 
\begin{eqnarray*}\label{SigmaLCont}
\Sigma_l&=&\{
(q^i, \dot{q}^i, \mu_{i}, \tilde{\mu}_i)\in T^*TQ\; |\;  \mu_i=\frac{\partial {L} }{\partial q^i} +\lambda_{\alpha}\frac{\partial \phi^{\alpha} }{\partial q^i},\\\nonumber
&&\tilde{\mu}_i=\frac{\partial {L} }{\partial \dot{q}^i} +\lambda_{\alpha}\frac{\partial \phi^{\alpha} }{\partial \dot{q}^i},\quad \phi^{\alpha}(q,\dot q)=0,\,\,\, 1\leq\alpha\leq m\}\; .
\end{eqnarray*}
Therefore,
\begin{eqnarray*}\label{TulcSigmaLCont}
\alpha_Q^{-1}\lp\Sigma_l\rp&=&\{(q^i, p_i, \dot{q}^i, \dot{p}_i)\in TT^*Q\;|\; p_i=\frac{\partial {L} }{\partial \dot{q}^i} +\lambda_{\alpha}\frac{\partial \phi^{\alpha} }{\partial \dot{q}^i},\\\nonumber
&&\dot{p}_i=\frac{\partial {L} }{\partial q^i} +\lambda_{\alpha}\frac{\partial \phi^{\alpha} }{\partial q^i},\; \phi^{\alpha}(q,\dot q)=0,\,\,\, 1\leq\alpha\leq m\}\; .
\end{eqnarray*}
The solutions for the dynamics given by $\alpha_Q^{-1}\lp\Sigma_l\rp\subset TT^*Q$ are curves $\gamma: I\subset \R\to T^*Q$ such that $\frac{\mathrm{d}\gamma}{\mathrm{d}t}:I\subset \R\to TT^*Q$ verifies that $\frac{\mathrm{d}\gamma}{\mathrm{d}t}(I)\subset \alpha_Q^{-1}\lp\Sigma_l\rp$. Locally, if $\gamma(t)=(q^i(t), p_i(t))$ then it must verify the following set of differential equations:
\begin{eqnarray*}
\frac{\mathrm{d}}{\mathrm{d}t}\left(\frac{\partial {L} }{\partial \dot{q}^i} +\lambda_{\alpha}\frac{\partial \phi^{\alpha} }{\partial \dot{q}^i}\right)-\frac{\partial {L} }{\partial q^i} -\lambda_{\alpha}\frac{\partial \phi^{\alpha} }{\partial q^i}&=&0,\\
\phi^{\alpha}(q^{i},\dot q^{i})&=&0,
\end{eqnarray*}
which coincide with equations (\ref{Vako}).

\section{The inverse problem of the calculus of variations}\label{SInverse}

In the previous section it is shown that given a regular Lagrangian function $L: TQ\rightarrow \R$ we can always associate a unique SODE $\Gamma_L$, see equation (\ref{aqw}).
The inverse problem of the calculus of variations studies when a prescribed SODE $\Gamma$ is   equivalent to the  Euler-Lagrange equations for a regular Lagrangian $L: TQ\rightarrow \R$, in the sense of searching a non-singular multiplier matrix $( g_{ij}(q, \dot{q}))$ such that
\begin{equation}\label{eq:condInvP}
g_{ij}\left(\ddot{q}^j-\Gamma^j(q, \dot{q})\right)=\frac{d}{dt}\left(\frac{\partial L}{\partial \dot{q}^i}\right)- \frac{\partial L}{\partial q^i}, \quad i,j=1,\ldots,n=\dim Q 
\end{equation}
has a regular solution $L$. Note that in the affirmative case we have that $g_{ij}=\frac{\partial^{2}L}{\partial\dot{q}^{i}\partial\dot{q}^{j}}$ and the solutions to $\Gamma$ are exactly the same as the solutions to the Euler-Lagrange equations for $L$.

Geometrically, condition~\eqref{eq:condInvP} can be captured into the requirement of the existence of a function $L: TQ\rightarrow \R$ such that
${\mathcal L}_{\Gamma}\Theta_L=dL$, see~\eqref{eq:Alternative}.
When the condition is satisified, the SODE $\Gamma$ is called \emph{variational}.

The existence of a regular Lagrangian for $\Gamma$ is equivalent to the existence of multipliers $(g_{ij}(q,\dot{q}))$ satisfying the  \emph{Helmholtz conditions} (see \cite{Helmholtz1887} for a more general version):
\begin{eqnarray}
&&\det(g_{ij})\not= 0, \qquad g_{ji}=g_{ij}, \qquad
\frac{\partial g_{ij}}{\partial \dot{q}^k}=\frac{\partial g_{ik}}{\partial \dot{q}^j} \label{eq:Helmholtz1}\\
&&\Gamma(g_{ij})-\nabla^k_jg_{ik}-\nabla^k_i g_{kj}=0,  \label{eq:HelmholtzNabla}\\ &&
g_{ik}\Phi^k_j=g_{jk}\Phi^k_i.  \label{eq:HelmholtzPhi}
\end{eqnarray}
where $\Gamma=\dot{q}^i\frac{\partial}{\partial q^i}+
\Gamma^i(q, \dot{q})\frac{\partial}{\partial \dot{q}^i}$,
$\nabla^i_j=-\frac{1}{2}\frac{\partial \Gamma^i}{\partial \dot q^j}$ and
$\Phi^k_j=\Gamma\left(\frac{\partial \Gamma^k}{\partial \dot q^j}\right)-2\frac{\partial \Gamma^k}{\partial  q^j}-\frac{1}{2}\frac{\partial \Gamma^i}{\partial \dot q^j}\frac{\partial \Gamma^k}{\partial \dot q^i}.
$

The problem is specially difficult since Helmholtz conditions are a  mixed set of algebraic equations and partial differential equations (PDE) for the multipliers $g_{ij}$.
There are many characterizations of the inverse problem of the calculus of variations in the literature, but not much is known about the complete solution. For instance, $n=1$ is always variational \cite{Sonin1886} and $n=2$ was solved by Douglas in \cite{Douglas}, but for $n>2$ no complete classification exists. Some partial results exist, more precisely, some cases in Douglas' classification have been generalized to arbitrary $n$. See for instance \cite{CPST1999}, \cite{94CSMBP}, \cite{SCM1998}.

The following characterization of being variational will be very useful in the sequel.
\begin{thm}\cite{81Crampin}\label{theorem-crampin}
    A SODE $\Gamma$ on $TQ$ is variational if and only if there exists a $2$-form $\Omega$ on
    $TQ$ of maximal rank such that
    \begin{enumerate}
		\item $d\Omega=0$,
		\item $\Omega(v_{1},v_{2})=0\ \forall v_{1}, v_{2}\in V(TQ)$, where $V(TQ)$ denotes the set of all vertical vector fields for $\tau_{Q}\colon TQ \to Q$, that is, $V(TQ)={\rm Ker} T\tau_Q$,
    \item $\mathcal{L}_{\Gamma}\Omega=0$.
    \end{enumerate}
    \end{thm}

\subsection{A new geometric characterization for the inverse of the calculus of variations}\label{alter}

For a given SODE $\Gamma :TQ \longrightarrow TTQ$ and a local diffeomorphism $F:TQ\longrightarrow
    T^{*}Q$ of fibre bundles over $Q$ (that is, $\pi_Q\circ F=\tau_Q$),
    we define a submanifold $\Sigma_{\Gamma,F}:=$Im$(\mu_{\Gamma, F})\subset T^{*}TQ$,
    where $\mu_{\Gamma, F}=\alpha_Q \circ
    TF\circ\Gamma$ is a $1$-form on $TQ$:
    $$
    \xymatrix{
    TTQ \ar[r]^{TF} & TT^{*}Q \ar[r]^{\alpha_Q} & T^{*}TQ \\
    &&&\\
    TQ \ar[uu]^{\Gamma} \ar[r]^{F} \ar[uurr]_{\mu_{\Gamma,F}} \ar[uur] & T^{*}Q & \\
    }
    $$

   Let $(q^{i},\dot{q}^i)$ denote
    fibered coordinates on $TQ$ and we write $F$ and $\Gamma$ in
    these coordinates as
   $$
   F\left(q^{i},\dot{q}^{i}\right)= \left(q^{i},F_{i}(q,\dot{q})\right), \quad
   \Gamma\left(q^{i},\dot{q}^{i}\right)= \left(q^{i},\dot{q}^{i},\dot{q}^{i},\Gamma^{i}(q,\dot{q})\right).
   $$

Then the above diagram in coordinates becomes
   $$
\xymatrix{
   \left(q^{i},\dot{q}^{i},\dot{q}^{i},\Gamma^{i}(q,\dot{q})\right) \ar[r]^/-15pt/{TF} &
    \left(q^{i},F_{i},\dot{q}^{i},\frac{\partial F_{i}}{\partial q^{j}}\dot{q}^{j}+\frac{\partial F_{i}}{\partial
    \dot{q}^{j}}\Gamma^{j}\right) \ar[r]^{\alpha_Q}
    & \left(q^{i},\dot{q}^{i},\frac{\partial F_{i}}{\partial q^{j}}\dot{q}^{j}+
    \frac{\partial F_{i}}{\partial \dot{q}^{j}}\Gamma^{j},F_{i}\right)& \\
    &&\\
    \left(q^{i},\dot{q}^{i}\right) \ar[r]^{F} \ar[uu]^{\Gamma} \ar[uurr]_{\mu_{\Gamma, F}} &  \left(q^{i},F_{i}\right).&& \\
    }
    $$
   Note that $\mu_{\Gamma,F}$ is a $1$-form on $TQ$ locally given by $\left(\frac{\partial F_{i}}{\partial
    q^{j}}\dot{q}^{j}+\frac{\partial F_{i}}{\partial
    \dot{q}^{j}}\Gamma^{j}\right)dq^{i}+F_{i}d\dot{q}^{i}$.
From this last expression it is easy to deduce that
\begin{equation}\label{useful-eq1}
\mu_{\Gamma, F}={\mathcal L}_{\Gamma}F^*\theta_Q.
\end{equation}

In this section we will show that the inverse problem of the calculus of variations  for a SODE  $\Gamma$  is equivalent to see whether or not it
    is possible to find a local diffeomorphism
     $F:TQ\longrightarrow T^{*}Q$ of fibre bundles over  $Q$
     such that $\Sigma_{\Gamma,F}=\hbox{Im}(\mu_{\Gamma, F})$ is  a Lagrangian submanifold of $(T^*TQ, \omega_{TQ})$. This characterization will be useful  for our approach to the inverse problem for constrained systems. % (see Section \ref{Sec:Tulczy}). 

    Observe that  since
    $\Sigma_{\Gamma,F}$ is the image of the 1-form $\mu_{\Gamma, F}$ on $TQ$, $\Sigma_{\Gamma,F}$ is a Lagrangian submanifold of $(T^*TQ, \omega_{TQ})$ if and only if $\mu_{\Gamma, F}$ is closed, i.e. $d\mu_{\Gamma, F}=0$. Therefore, using Poincar\'e lemma  we deduce the local existence of a  function $L$ on $TQ$
     such that $\mu_{\Gamma, F}=dL$.

\begin{thm}\label{poi}
A SODE $\Gamma$ on $TQ$ is variational if and only if there exists a  local diffeomorphism  $F:TQ\longrightarrow T^{*}Q$ of fibre bundles over $Q$ such that  $\hbox{Im}(\mu_{\Gamma, F})$   is a Lagrangian submanifold of $(T^*TQ, \omega_{TQ})$.
\end{thm}
\begin{proof}
We use the characterization in Theorem \ref{theorem-crampin} to prove this result.

   $\Leftarrow$
   Define $\Omega=-d(F^{*}\theta_{Q})$, where $\theta_{Q}$ denotes
    the Liouville $1$-form on $T^{*}Q$ in~\eqref{eq:ThetaQ}. Note that if
    $F(q^i, \dot{q}^i)=(q^i, F_i(q, \dot{q}))$,then
    \begin{eqnarray*}
    \mathcal{L}_{\Gamma}F^{*}\theta_{Q}&=&\mathcal{L}_{\Gamma}(F_{i}dq^{i})=\Gamma(F_{i})dq^{i}
    +F_{i}d \dot{q}^{i}    \\
    &=&\left(\frac{\partial F_{i}}{\partial
    q^{j}}\dot{q}^{j}+\frac{\partial F_{i}}{\partial
    \dot{q}^{j}}\Gamma^{j}\right)dq^{i}+F_{i}d\dot{q}^{i}=\mu_{\Gamma, F}.
    \end{eqnarray*}
Then $\Omega$ trivially satisfies all the conditions in Theorem \ref{theorem-crampin}.

 $\Rightarrow$
From Theorem  \ref{theorem-crampin} we have that $\Gamma$ is variational if and only if there exists a non-degenerate 2-form $\Omega$ on $TQ$ satisfying   $\mathcal{L}_{\Gamma}\Omega=0$,
    $\Omega (v, w)=0$ for all $v, w\in V(TQ)$ and
     $d\Omega=0$.
From the last condition we deduce that locally $\Omega=d\Theta$ on a neighborhood $U\subseteq TQ$, where $\Theta$ is a $1$-form on $U$.
The restriction of $d\Theta$ to vertical subspaces is zero. Thus the restriction of $\Theta$ to each fiber is exact, then there is a function $f: U\rightarrow \R$ such that
$\Theta(v)=\langle df, v\rangle$ for any $v\in V(TQ)$.
Therefore,
$\widetilde\Theta=\Theta-df$ verifies $\widetilde{\Theta}(v)=0$ for all $v\in V(TQ)$ and $d\widetilde{\Theta}=\Omega$.
Using $\tilde{\Theta}$ we construct the map $F: U\subseteq TQ\rightarrow T^*Q$ as follows:
$$
\langle F(v_{q}), w_{q} \rangle = \langle \widetilde{\Theta}(v_{q}) , W_{q} \rangle, 
$$
where $v_{q}\in TQ$, $w_{q}\in TQ$ and $W_{q}\in TTQ$ satisfies $T\tau_{Q}(W_{q})=w_{q}$. This definition does not depend on the choice of $W_q$ since $\widetilde{\Theta}$ vanishes on vertical vector fields.
%\[
%F(u)=\tau_{T^*Q}(\alpha^{-1}_Q(({\mathcal L}_{\Gamma}\widetilde{\Theta})(u))), \quad u\in U:
%\]
%$$
%\xymatrix{
%T^{*}TQ \ar[rr]_{\alpha_{Q}^{-1}} & & TT^{*}Q \ar[dd]^{\tau_{T^{*}Q}}  \\
%\\
%TQ \ar@{-->}[rr]^{F} \ar[uu]^{\mathcal{L}_{\Gamma}\tilde{\Theta}} & & T^{*}Q
%}
%$$
Then, it is easy to show that $\widetilde{\Theta}=F^*\theta_Q$ and from equation  (\ref{useful-eq1}), $\mu_{\Gamma, F}={\mathcal L}_{\Gamma}\widetilde{\Theta}$ verifies
\[
d\mu_{\Gamma, F}=d{\mathcal L}_{\Gamma}\widetilde{\Theta}={\mathcal L}_{\Gamma}d\widetilde{\Theta}={\mathcal L}_{\Gamma}\Omega=0.
\]
Hence $\hbox{Im}(\mu_{\Gamma, F})$ is a Lagrangian submanifold of $(T^*TQ, \omega_{TQ})$. Note that the non-degeneracy of $\Omega$ implies that $\det\left(\frac{\partial F_{i}}{\partial \dot{q}^{j}}\right)\not=0$ which is precisely the condition for $F$ to be a local diffeomorphism.
\end{proof}

    Observe that the  submanifold $\Sigma_{\Gamma, F}\subset T^{*}TQ$, given in local  coordinates by
    $\displaystyle{\left(q^{i},\dot{q}^{i},
    \frac{\partial F_{i}}{\partial q^{j}}\dot{q}^{j}+\frac{\partial F_{i}}{\partial
    \dot{q}^{j}}\Gamma^{j},F_{i}\right)}$ is a Lagrangian submanifold  of $(T^*TQ, \omega_{TQ})$  if and only if
    there exists a locally defined function $L\colon TQ \rightarrow \mathbb{R}$ such that
    $$\frac{\partial L}{\partial \dot{q}^{i}}=F_{i}\   \hbox{   and   } \ \frac{\partial L}{\partial q^{i}}=
    \frac{\partial F_{i}}{\partial q^{j}}\dot{q}^{j}+\frac{\partial F_{i}}{\partial
    \dot{q}^{j}}\Gamma^{j}.$$

%Recalling that operating with $\Gamma$ is equivalent to taking the total derivative, 
We have that
    \begin{eqnarray*}
    \frac{d}{dt}\left(\frac{\partial L}{\partial
    \dot{q}^{i}}\right)-\frac{\partial L}{\partial q^{i}}&=&
    \frac{d F_i}{dt}-\frac{\partial F_{i}}{\partial q^{j}}\dot{q}^{j}-\frac{\partial F_{i}}{\partial
    \dot{q}^{j}}\Gamma^{j}
   \\
    &=&
    \frac{\partial F_{i}}{\partial q^{j}}\dot{q}^{j}+\frac{\partial F_{i}}{\partial
    \dot{q}^{j}}\ddot{q}^{j}-\frac{\partial F_{i}}{\partial q^{j}}\dot{q}^{j}-\frac{\partial F_{i}}{\partial
    \dot{q}^{j}}\Gamma^{j}\\
    &=&
    \frac{\partial F_{i}}{\partial
    \dot{q}^{j}}(\ddot{q}^{j}-\Gamma^{j}).
    \end{eqnarray*}
   Thus the
    solutions to Euler-Lagrange equations for $L$ coincide with
    the solutions to the SODE $\Gamma$,
    since  $F$ is a local diffeomorphism, that is, locally the matrix
    $\left(\frac{\partial F_{i}}{\partial \dot{q}^{j}}\right)$ is non-degenerate. Then the multipliers for the Helmholtz conditions are $g_{ij}=\frac{\partial F_{i}}{\partial \dot{q}^{j}}$. 

\begin{remark}
Since $\alpha_Q: TT^*Q\rightarrow T^*TQ$ is a symplectomorphism (see Section \ref{Sec:Tulczy}) then we can alternatively characterize the inverse problem of the calculus of variations  for a SODE  $\Gamma$  seeing  whether the submanifold  $S_{\Gamma, F}$ defined by 
\[
S_{\Gamma, F}=TF(\Gamma(Q))=\alpha_{Q}^{-1}(\mu_{\Gamma, F}(Q))
\]
 is a Lagrangian submanifold of the symplectic manifold $(TT^*Q, d_T\omega_Q)$. 
\end{remark}

\begin{remark}
The submanifold $\Sigma_{F,\Gamma}$ will be Lagrangian if
    \begin{eqnarray*}
    d\left(\left(\frac{\partial F_{i}}{\partial q^{j}}\dot{q}^{j}
    +\frac{\partial F_{i}}{\partial \dot{q}^{j}}\Gamma^{j}\right)dq^{i}+F_{i}d\dot{q}^{i}\right)&=&0.\\
    \end{eqnarray*}
Equivalently, we get the following conditions:
    \begin{eqnarray}
    \frac{\partial F_{i}}{\partial \dot{q}^{k}}&=&\frac{\partial F_{k}}{\partial
    \dot{q}^{i}}, \label{L11} \\
    \frac{\partial^{2}F_{i}}{\partial q^{k}\partial q^{j}}\dot{q}^{j}
    +\frac{\partial^{2}F_{i}}{\partial q^{k}\partial \dot{q}^{j}}\Gamma^{j}
    +\frac{\partial F_{i}}{\partial \dot{q}^{j}}\frac{\partial \Gamma^{j}}{\partial q^{k}}
    &=&\frac{\partial^{2} F_{k}}{\partial q^{i}\partial q^{j}}\dot{q}^{j}
    +\frac{\partial^{2} F_{k}}{\partial q^{i}\partial \dot{q}^{j}}\Gamma^{j}
    +\frac{\partial F_{k}}{\partial \dot{q}^{j}}\frac{\partial \Gamma^{j}}{\partial
    q^{i}}, \label{L22} \\
    \frac{\partial F_{k}}{\partial q^{i}}&=&
    \frac{\partial^{2} F_{i}}{\partial \dot{q}^{k}\partial q^{j}}\dot{q}^{j}
    +\frac{\partial F_{i}}{\partial q^{k}}
    +\frac{\partial^{2}F_{i}}{\partial \dot{q}^{k}\partial \dot{q}^{j}}\Gamma^{j}
    +\frac{\partial F_{i}}{\partial \dot{q}^{j}}\frac{\partial \Gamma^{j}}{\partial \dot{q}^{k}}. \label{L33}
    \end{eqnarray} 
In Appendix~\ref{Appendix} detailed computations show the equivalence between the equations (\ref{L11}), (\ref{L22}) and (\ref{L33}) and the Helmholtz conditions~\eqref{eq:Helmholtz1},~\eqref{eq:HelmholtzNabla},~\eqref{eq:HelmholtzPhi} for $g_{ij}=\frac{\partial F_{i}}{\partial \dot{q}^{j}}$.
\end{remark}
		
\begin{remark}
If we admit that the matrix $(g_{ij})$ is degenerate, then we get conditions for the existence of a singular Lagrangian $L$ such that 
$$
g_{ij}(\ddot{q}^{j}-\Gamma^{j})=\frac{d}{dt}\lp\frac{\partial L}{\partial \dot{q}^{i}}\rp-\frac{\partial L}{\partial q^{i}},
$$  
which implies that the solutions of the SODE are also solutions to the Euler-Lagrange equations for $L$.

\end{remark}

%\begin{remark}
%Since $\alpha_Q: TT^*Q\rightarrow T^*TQ$ is a symplectomorphism (see Section \ref{Sec:Tulczy}) then we can alternatively characterize the inverse problem of the calculus of variations  for a SODE  $\Gamma$  seeing  whether the submanifold  $S_{\Gamma, F}$ defined by 
%\[
%S_{\Gamma, F}=TF(\Gamma(Q))=\alpha_{Q}^{-1}(\mu_{\Gamma, F}(Q))
%\]
 %is a Lagrangian submanifold of the symplectic manifold $(TT^*Q, d_T\omega_Q)$. 
%\end{remark}

\begin{ex}
Let $Q=\mathbb{R}^2$, $\Gamma$ be given by $\ddot{x}=f(x,y),\ \ddot{y}=f(x,y)$, that is, $\Gamma^{1}=\Gamma^{2}=f(x,y)$. Then $L=\frac{1}{2}(\dot{x}-\dot{y})^{2}$ is a singular Lagrangian that gives the dynamics $\ddot{x}=\ddot{y}$, which includes the solutions to $\Gamma$, and satisfies
$$
g_{ij}(\ddot{q}^{j}-\Gamma^{j})=\frac{d}{dt}\lp\frac{\partial L}{\partial \dot{q}^{i}}\rp-\frac{\partial L}{\partial q^{i}}
$$
with $g_{11}=g_{22}=1$ and $g_{12}=g_{21}=-1$. For some choices of $f(x,y)$, the SODE will fall into one of the cases in \cite{Douglas} which do not admit a regular Lagrangian. For instance if we take $f(x,y)=xy$, then, in the notation of \cite{Douglas} (except for the coordinates which we denote as $(t,x,y,\dot{x},\dot{y})$), we get
$$
\begin{array}{ccc}
\begin{array}{rcl}
A&=&-2x,\\
A_{1}&=&-2\dot{x},\\
A_{2}&=&-2xy,\\
\end{array}
&
\begin{array}{rcl}
B&=&(y-x),\\
B_{1}&=&2(\dot{y}-\dot{x}),\\
B_{2}&=&0,
\end{array}
&
\begin{array}{rcl}
C&=&2y;\\
C_{1}&=&2\dot{y};\\
C_{2}&=&2xy.\\
\end{array}\\
\end{array}
$$

Then the determinant of  $\left(\begin{smallmatrix} A & B & C \\ A_{1} & B_{1} & C_{1} \\ A_{2} & B_{2} & C_{2} \end{smallmatrix}\right)$ is nonzero and the example falls into the nonvariational Case IV of Douglas~\cite{Douglas}.
\end{ex}

\subsection{Chaplygin hamiltonization}

As we have seen in Section \ref{chaplygin}, the equations of motion of a  noholonomic Chaplygin system can be reduced to a second-order differential equation on $Q/G$. Then, we can apply the inverse problem of the calculus of variations  in an attempt to find a Lagrangian $L: T(Q/G)\rightarrow \R$ such that equations (\ref{asr}) are  equivalent to the Euler-Lagrange equations for the Lagrangian  $L$.

Denote by $\Gamma$ the SODE on $T(Q/G)$  in equations (\ref{asr}). By Theorem~\ref{poi}, $\Gamma$ is equivalent to the Euler-Lagrange equations of a Lagrangian if there exists a fiber diffeomorphism $F: T(Q/G)\to T^*(Q/G)$ such that $\hbox{Im}(\mu_{\Gamma, F})$ is a Lagrangian submanifold of $(T^*T(Q/G), \omega_{T(Q/G)})$.  

 Equivalently,  in the case of Chaplygin systems we can use the reduced Lagrangian $L^*: T(Q/G)\rightarrow \R$ defined in Section \ref{chaplygin} and  its  associated Legendre transformation 
\[
Leg_{L^*}: T(Q/G)\rightarrow T^*(Q/G)\; .
\]
Then we can define the vector field $\widetilde{\Gamma}=(Leg_{L^*})_*\Gamma$ on $T^*(Q/G)$ representing the nonholonomic dynamics, now on the Hamiltonian side. 
But if there exists a solution $F: T(Q/G)\to T^*(Q/G)$ of the inverse problem of calculus of variations  then 
the vector field $F_*\Gamma$ is locally Hamiltonian. That is, locally there exists a function $\widehat{H}: T^*(Q/G)\rightarrow \R$ such that 
\[
i_{F_*\Gamma}\omega_{Q/G}=d\hat{H}.
\]
Therefore, if we consider the diffeomorphism
$G: T^*(Q/G)\to T^*(Q/G)$ given by $G=F\circ (Leg_{L^*})^{-1}$ then 
it is clear by construction that $G_*\widetilde{\Gamma}=F_*\Gamma$ and 
\begin{equation}\label{hamiltonization}
i_{\widetilde{\Gamma}}\Omega=d\hat{H},
\end{equation}
where $\Omega=G^*(\omega_{Q/G})$ and $\hat{H}=H\circ G$. 
Equation (\ref{hamiltonization}) corresponds to the standard notion of hamiltonization of a Chaplygin system~\cite{2010BalNar,2009BlochFerMestdag}.

$$
\xymatrix{
TT^{*}(Q/G) \ar[rrrr]^{TG} & & & & TT^{*}(Q/G) \\
&&&& \\
T^{*}(Q/G) \ar[uu]^{\widetilde{\Gamma}=(Leg_{L^*})_*\Gamma} \ar@/_2pc/[rrrr]^{G} & & T(Q/G) \ar[rr]^{F} \ar[ll]_{Leg_{L^*}} & & T^{*}(Q/G) \ar[uu]_{X_{\hat{H}}=F_*\Gamma} \\
}
$$
\subsection{A new geometric characterization for the time-dependent inverse problem}\label{Sec:NewTime}
Now we consider a non-autonomous second order differential system of the form
\begin{eqnarray}
\ddot{q}^{j}=\Gamma^{j}(t,q^{i},\dot{q}^{i}). \label{eq-time}
\end{eqnarray}
We want to characterize when a regular time-dependent Lagrangian $L(t,q,\dot{q})$ exists such that the solutions of the corresponding Euler-Lagrange equations coincide with the solutions of the system \eqref{eq-time}. Finding a regular Lagrangian is equivalent to finding a multiplier matrix $(g_{ij}(t,q,\dot{q}))$ satisfying the Helmholtz conditions for time-dependent SODE's, which can be written as in the previous section~\eqref{eq:Helmholtz1},~\eqref{eq:HelmholtzNabla},~\eqref{eq:HelmholtzPhi}, but now $\Gamma=\frac{\partial}{\partial t}+\dot{q}^{i}\frac{\partial}{\partial q^{i}}+\Gamma^{i}\frac{\partial}{\partial \dot{q}^{i}}$.
In \cite{Douglas}, Douglas solved this problem for the two dimensional case. He thoroughly analyzed the Helmholtz conditions using Riquier theory to give a classification of variational and nonvariational SODE's in terms of conditions that depend only on $\Gamma^{j}$ and some of its partial derivatives.

\begin{defn} \label{time-SODE}
A vector field $\Gamma$ on $\mathbb{R}\times TQ$ is a SODE if $\langle \Gamma, \theta^{i} \rangle=0$ and $\langle \Gamma , dt \rangle=1$, where $\theta^{i}=dq^{i}-\dot{q}^{i}dt$ are the usual contact $1$-forms. In local coordinates $(t,q,\dot{q})$ for $\mathbb{R}\times TQ$,
$$
\Gamma=\frac{\partial}{\partial t}+\dot{q}^{i}\frac{\partial}{\partial q^{i}}+\Gamma^{i}\frac{\partial}{\partial \dot{q}^{i}}.
$$
The integral curves of $\Gamma$ are the ones satisfying the system of explicit second order differential equations $\ddot{q}^{i}=\Gamma^{i}(t,q,\dot{q})$.
\end{defn}

\begin{remark} \label{Rem:ThetaL}
An example of SODE on $\mathbb{R}\times TQ$ is the Euler-Lagrange vector field associated to a regular Lagrangian function $L:\mathbb{R}\times TQ\longrightarrow \mathbb{R}$, which is defined as the unique vector field $\Gamma$ satisfying $i_{\Gamma}\Omega_{L}=0$ and $\langle \Gamma , dt \rangle=1$, where $\Omega_{L}=-d\theta_{L}$ is the Cartan $2$-form, $\theta_{L}=Ldt+dL\circ S$ is the Cartan $1$-form and $S=\frac{\partial}{\partial\dot{q}^{i}}\otimes\theta^{i}$. Note that $(\Omega_{L},dt)$ provides $\mathbb{R}\times TQ$ with a cosymplectic structure if $L$ is regular \cite{1992Cantrijn}.
\end{remark}

In \cite{84CPT} an alternative characterization analogous to the one in \cite{81Crampin} is given for the time-dependent case:

\begin{thm} \cite{84CPT} \label{time-crampin}
A SODE $\Gamma$ on $\mathbb{R}\times TQ$ is variational if and only if there exists a $2$-form $\Omega$ on 
$\mathbb{R}\times TQ$ of maximal rank such that
\begin{enumerate}
\item $d\Omega=0$,
\item $\Omega(v_{1},v_{2})=0\; ,\ \forall v_{1},v_{2}\in V(\mathbb{R}\times TQ)$,
\item $i_{\Gamma}\Omega=0$.
\end{enumerate}
\end{thm}

Consider now the following diagram, where $F:\mathbb{R}\times TQ \longrightarrow \mathbb{R}\times T^{*}Q$ is a local diffeomorphism over $\mathbb{R}\times Q$:

$$
\xymatrix{
		T(\mathbb{R}\times TQ) \ar[r]^/-20pt/{TF} & T(\mathbb{R}\times T^{*}Q)\cong T\mathbb{R}\times TT^{*}Q \\	
	  &&&\\
		\mathbb{R}\times TQ \ar[r]^{F}  \ar[uu]^{\Gamma} \ar[uur]_{\quad \gamma_{\Gamma,F}:=TF \circ\Gamma} & \mathbb{R}\times T^{*}Q. &\\
    }
 $$

In local coordinates, if we write $\Gamma(t,q^{i},\dot{q}^{i})= (t,q^{i},\dot{q}^{i},1,\dot{q}^{i},\Gamma^{i}(t,q^{j},\dot{q}^{j}))$ and $F(t,q^{i},\dot{q}^{i})=(t,q^{i},F_{i}(t,q,\dot{q}))$, then we get
$$
\gamma_{\Gamma,F}(t,q^{i},\dot{q}^{i})=\left(t,q^{i},F_{i}(t,q,\dot{q}),1,\dot{q}^{i},\frac{\partial F_{i}}{\partial t}+\dot{q}^{j}\frac{\partial F_{i}}{\partial q^{j}}+\Gamma^{j}\frac{\partial F_{i}}{\partial \dot{q}^{j}}\right).
$$

To characterize the variational property of the time-dependent SODE, we need the following definitions. Let $f$ be a function on $P$, $f^c$ and $f^v$ denote respectively the complete and vertical lift of the function $f$ to $TP$, see~\cite{Yano}. These lifts are defined as follows:
\begin{eqnarray*}
f^c&=&\Delta( f\circ \tau_P), \quad \Delta \mbox{ is the Liouville vector field},\\
f^v&=&f\circ \tau_P, 
\end{eqnarray*}
where $ \tau_P\colon TP \rightarrow P$ is the canonical projection. 
\begin{defn}[\cite{1990Courant},\cite{Grabowski}]
Let $(P,\left\{,\right\})$ be a Poisson manifold. The tangent Poisson bracket is given by
\begin{eqnarray*}
\left\{ f^{c}, g^{c} \right\}^{T}&=&\left\{f,g\right\}^{c},\\
\left\{ f^{c}, g^{v} \right\}^{T}&=&\left\{f,g\right\}^{v},\\
\left\{ f^{v}, g^{v} \right\}^{T}&=&0.
\end{eqnarray*}
If $(x^{i})$ denote local coordinates in $P$ and the Poisson bivector is given by
$$
\Lambda=\frac{1}{2}\Lambda^{ij}(x)\frac{\partial}{\partial x^{i}}\wedge \frac{\partial}{\partial x^{j}},
$$
then
$$
d_{T}\Lambda=\Lambda^T:=\Lambda^{ij}(x)\frac{\partial}{\partial x^{i}}\wedge \frac{\partial}{\partial\dot{x}^{j}}+\frac{1}{2}\frac{\partial\Lambda^{ij}(x)}{\partial x^{k}}\dot{x}^{k}\frac{\partial}{\partial\dot{x}^{i}}\wedge \frac{\partial}{\partial\dot{x}^{j}}
$$
is the Poisson bivector corresponding to the bracket $\{ \cdot,\cdot\}^T$.
\end{defn}

\begin{defn}[\cite{Vaisman2}]
Let $(P,\left\{ ,\right\})$ be a Poisson manifold and $N$ be a submanifold of $P$. Denote by $\Lambda$ the Poisson bivector and by $\sharp:T^{*}P\longrightarrow TP$ the induced morphism of vector bundles. The submanifold $N$ is called Lagrangian if 
$$
\sharp(TN^{\circ})=TN\cap\mathcal{C},
$$
where $TN^{\circ}$ is the annihilator of $TN$ and $\mathcal{C}:=\hbox{Im}(\sharp)$ is the characteristic distribution.
\end{defn}

Now we consider the projection $\tilde{\pi}: T^*(\R\times Q)\equiv T^*\R\times T^*Q\rightarrow \R\times T^*Q$ given by $\tilde{\pi}=(\pi_{\R}, id_{T^*Q})$ that is
$\tilde{\pi}(\alpha_t, \beta_q)=(t, \beta_q)$ where $\alpha_t\in T_t^*\R$ and $\beta_q\in T^*_qQ$. We induce a Poisson  bracket on $\R\times T^*Q$ such that $\tilde{\pi}$ is a Poisson morphism where we are considering in $T^*(\R\times Q)$ the standard Poisson bracket induced by the symplectic 2-form $\omega_{\R\times Q}$. 
Locally, in coordinates $(t, q^i, p_i)$ for $\R\times T^*Q$ we have that the induced bracket $\{\; ,\; \}$ is defined by
 \[
 \{t, q^i\}=\{t, p_i\}=\{q^i, q^j\}=\{p_i, p_j\}=0 \ \hbox{and }   \left\{q^{i},p_{i}\right\}=1\; .
 \]
Then we take its tangent lift to $T(\mathbb{R}\times T^{*}Q)$, which is defined on the induced coordinate functions  $(t, q, p, v_t, \dot{q}, \dot{p})$ by
$$
\left\{ q^{i}, \dot{p}_{i} \right\}^{T}=1, \ \left\{ \dot{q}^{i}, p_{i}\right\}^{T}=1
$$
and the remaining Poisson brackets vanish. The variational property of $\Gamma$ will be characterized in terms of Lagrangian submanifolds for this Poisson structure. To be more precise, $\hbox{Im}(\gamma_{\Gamma,F})$ must  be Lagrangian for some $F$. 

%First recall the following definition:
%\begin{defn}[\cite{Vaisman2}]
%Let $(P,\left\{ ,\right\})$ be a Poisson manifold and $N\subset P$ a submanifold. Denote by $\Lambda$ the Poisson bivector and by $\sharp:T^{*}P\longrightarrow TP$ the induced morphism of vector bundles. $N$ is called Lagrangian if 
%$$
%\sharp(TN^{\circ})=TN\cap\mathcal{C},
%$$
%where $TN^{\circ}$ is the annihilator of $TN$ and $\mathcal{C}:=\hbox{Im}\sharp$ is the characteristic distribution.
%\end{defn}

Now we will write the conditions that arise when forcing $\hbox{Im}(\mu_{\Gamma,F})$ to be Lagrangian. In local coordinates $(t,q,p,v_{t},\dot{q},\dot{p})$ for $T(\mathbb{R}\times T^{*}Q)$ we have
\begin{eqnarray*}
T(\hbox{Im}(\gamma_{\Gamma,F}))&=&\hbox{span}\left\{ \frac{\partial}{\partial t}+\frac{\partial F_{j}}{\partial t}\frac{\partial}{\partial p_{j}}+\frac{\partial\Gamma(F_{j})}{\partial t}\frac{\partial}{\partial \dot{p}_{j}}, \right.\\ 
&&\frac{\partial}{\partial q^{i}}+\frac{\partial F_{j}}{\partial q^{i}}\frac{\partial}{\partial p_{j}}+\frac{\partial\Gamma(F_{j})}{\partial q^{i}}\frac{\partial}{\partial \dot{p}_{j}}, \left.\frac{\partial}{\partial \dot{q}^{i}}+\frac{\partial F_{j}}{\partial \dot{q}^{i}}\frac{\partial}{\partial p_{j}}+\frac{\partial\Gamma(F_{j})}{\partial\dot{q}^{i}}\frac{\partial}{\partial \dot{p}_{j}} \right\},\\
\end{eqnarray*}
$$
T(\hbox{Im}(\gamma_{\Gamma,F}))^{\circ}=\hbox{span}\left\{ dv_{t}, \frac{\partial F_{i}}{\partial q^{j}}dq^{j}-dp_{i}+\frac{\partial F_{i}}{\partial t}dt+\frac{\partial F_{i}}{\partial \dot{q}^{j}}d\dot{q}^{j}, \frac{\partial \Gamma(F_{i})}{\partial q^{j}}dq^{j}-d\dot{p}_{i}+\frac{\partial\Gamma(F_{i})}{\partial t}dt+\frac{\partial\Gamma(F_{i})}{\partial\dot{q}^{j}}d\dot{q}^{j} \right\},
$$
and
$$
\sharp(T(\hbox{Im}(\gamma_{\Gamma,F}))^{\circ})=\hbox{span}\left\{\frac{\partial F_{i}}{\partial q^{j}}\frac{\partial}{\partial\dot{p}_{j}}+\frac{\partial}{\partial \dot{q}^{i}}+\frac{\partial F_{i}}{\partial\dot{q}^{j}}\frac{\partial}{\partial p_{j}}, \frac{\partial \Gamma(F_{i})}{\partial q^{j}}\frac{\partial}{\partial \dot{p}_{j}}+\frac{\partial}{\partial q^{i}}+\frac{\partial \Gamma(F_{i})}{\partial \dot{q}^{j}}\frac{\partial}{\partial p_{j}} \right\}.
$$
As ${\mathcal C}={\rm span} \{dq^i,dp_i,d\dot{q}^i,d\dot{p}_i\}$, the equality $\sharp(T(\hbox{Im}(\gamma_{\Gamma,F}))^{\circ})=T(\hbox{Im}(\gamma_{\Gamma,F}))\cap\mathcal{C}$ holds if the following conditions are satisfied
\begin{equation}\label{time-eq}
\frac{\partial F_{j}}{\partial \dot{q}^{i}}=\frac{\partial F_{i}}{\partial\dot{q}^{j}}, \ \ \ 
\frac{\partial F_{i}}{\partial q^{j}}=\frac{\partial \Gamma(F_{j})}{\partial\dot{q}^{i}}, \ \ \ 
\frac{\partial\Gamma(F_{j})}{\partial q^{i}}=\frac{\partial\Gamma(F_{i})}{\partial q^{j}},
\end{equation}
which in more detail read
\begin{eqnarray}
\frac{\partial F_{j}}{\partial \dot{q}^{i}}&=&\frac{\partial F_{i}}{\partial \dot{q}^{j}}, \label{T1} \\
\frac{\partial^{2} F_{j}}{\partial \dot{q}^{i}\partial t}+\frac{\partial F_{j}}{\partial q^{i}}+\dot{q}^{k}\frac{\partial^{2}F_{j}}{\partial \dot{q}^{i}\partial q^{k}}+\frac{\partial \Gamma^{k}}{\partial \dot{q}^{i}}\frac{\partial F_{j}}{\partial \dot{q}^{k}}+\Gamma^{k}\frac{\partial^{2}F_{j}}{\partial \dot{q}^{i}\partial\dot{q}^{k}}&=&\frac{\partial F_{i}}{\partial q^{j}}, \label{T2} \\
\frac{\partial^{2} F_{j}}{\partial q^{i} \partial t}+\dot{q}^{k}\frac{\partial^{2}F_{j}}{\partial q^{i}\partial q^{k}}+\frac{\partial \Gamma^{k}}{\partial q^{i}}\frac{\partial F_{j}}{\partial \dot{q}^{k}}+\Gamma^{k}\frac{\partial^{2}F_{j}}{\partial q^{i}\partial \dot{q}^{k}}&=&\frac{\partial^{2}F_{i}}{\partial q^{j}\partial t}+\dot{q}^{k}\frac{\partial^{2}F_{i}}{\partial q^{j}\partial q^{k}}+\frac{\partial \Gamma^{k}}{\partial q^{j}}\frac{\partial F_{i}}{\partial \dot{q}^{k}}\nonumber \\ &&+\Gamma^{k}\frac{\partial^{2}F_{i}}{\partial q^{j}\partial \dot{q}^{k}}. \label{T3}
\end{eqnarray}

\begin{remark}
Note that the above conditions are the same that arise if we require that the natural projection of $\hbox{Im}(\gamma_{\Gamma,F})\subset T(\R\times T^*Q)$ onto $TT^{*}Q$ be a Lagrangian submanifold for each time coordinate with the symplectic structure $d_T\omega_{Q}$. 
\end{remark}

Now we give a characterization of the variational character of a time-dependent SODE in terms of Lagrangian submanifolds  of the Poisson manifold $(T(\R\times T^*Q), \{\; , \; \}^T)$. 
\begin{thm}\label{Teop}
A SODE $\Gamma$ on $\mathbb{R}\times TQ$ is variational if and only if there is a local diffeomorphism $F:\mathbb{R}\times TQ \longrightarrow \mathbb{R}\times T^{*}Q$ over $\mathbb{R}\times Q$ such that $\hbox{Im}(\gamma_{\Gamma,F})$ is a Lagrangian submanifold of $(T(\mathbb{R}\times T^{*}Q),\left\{, \right\}^T)$.
\end{thm}

\begin{proof}

$\Rightarrow$ If $\Gamma$ is variational  then there is a local regular Lagrangian $L:\mathbb{R}\times TQ\longrightarrow \mathbb{R}$ such that
$$
\frac{\partial^{2}L}{\partial\dot{q}^{i}\partial \dot{q}^{j}}(\ddot{q}^{j}-\Gamma^{j})=\frac{d}{dt}\left(\frac{\partial L}{\partial \dot{q}^{i}}\right)-\frac{\partial L}{\partial q^{i}},
$$
that is,
$$
\frac{\partial^{2}L}{\partial\dot{q}^{i}\partial\dot{q}^{j}}\Gamma^{j}=\frac{\partial L}{\partial q^{i}}-\frac{\partial^{2}L}{\partial t\partial\dot{q}^{i}}-\frac{\partial^{2}L}{\partial q^{j}\partial\dot{q}^{i}}\dot{q}^{j}.
$$
We can define $F_{i}=\frac{\partial L}{\partial \dot{q}^{i}}$, then
$$
\gamma_{\Gamma,F}(t,q^{i},\dot{q}^{i})=\left(t,q^{i},\frac{\partial L}{\partial\dot{q}^{i}},1,\dot{q}^{i},\frac{\partial L}{\partial q^{i}}\right),
$$
whose image is a Lagrangian submanifold of $(T(\mathbb{R}\times T^{*}Q),\left\{, \right\}^T)$.

$\Leftarrow$ Given a local diffeomorphism
$$
\begin{array}{lccc}
F: & \R\times TQ & \longrightarrow & \R\times T^{*}Q\\
& (t,q^{i},\dot{q}^{i}) & \longmapsto & (t,q^{i},F_{i})
\end{array}
$$
satisfying (\ref{T1}), (\ref{T2}) and (\ref{T3}), we define
$$
\Omega=-dF^{*}\theta_{Q}-i_{\Gamma}dF^{*}\theta_{Q}\wedge dt=-dF^{*}\theta_{Q}+(di_{\Gamma}F^{*}\theta_{Q}-\mathcal{L}_{\Gamma}F^{*}\theta_{Q})\wedge dt,
$$
which clearly satisfies $\Omega(v_{1},v_{2})=0$ for all $v_{1}, v_{2} \in V(\R\times TQ)$.
%$$
%=-dF^{*}\theta_{Q}+d(i_{\Gamma}F^{*}\theta_{Q})\wedge dt-\Gamma(F_{i})dq^{i}\wedge dt-F_{i}d\dot{q}^{i}\wedge dt
%$$
In local coordinates
$$
\Omega=-\frac{\partial F_{i}}{\partial q^{j}}dq^{j}\wedge dq^{i}-\frac{\partial F_{i}}{\partial \dot{q}^{j}}d\dot{q}^{j}\wedge dq^{i}+
\left(\frac{\partial F_{i}}{\partial q^{j}}\dot{q}^{i}-\frac{\partial F_{j}}{\partial q^{i}}\dot{q}^{i}-\frac{\partial F_{j}}{\partial \dot{q}^{i}}\Gamma^{i}\right)dq^{j}\wedge dt +\frac{\partial F_{i}}{\partial \dot{q}^{j}}\dot{q}^{i}d\dot{q}^{j}\wedge dt.
$$
Computing the exterior derivative of $\Omega$ we get
$$
d\Omega=-\frac{\partial\Gamma(F_{i})}{\partial q^{j}}dq^{j}\wedge dq^{i}\wedge dt-\frac{\partial\Gamma(F_{i})}{\partial \dot{q}^{j}}d\dot{q}^{j}\wedge dq^{i}\wedge dt-\frac{\partial F_{i}}{\partial q^{j}}dq^{j}\wedge d\dot{q}^{i}\wedge dt-\frac{\partial F_{i}}{\partial \dot{q}^{j}}d\dot{q}^{j}\wedge d\dot{q}^{i}\wedge dt.
$$
Conditions~\eqref{time-eq} on $F$ yield $d\Omega=0$. It is also readily checked that $i_{\Gamma}\Omega=0$. Since $F$ is a local diffeomorphism, that is, rank$\left(\frac{\partial F_{i}}{\partial \dot{q}^{j}}\right)=n$, the term $\frac{\partial F_{i}}{\partial \dot{q}^{j}}d\dot{q}^{j}\wedge dq^{i}$ makes $\Omega$ have maximal rank. Thus $\Omega$ satisfies all the conditions in Theorem \ref{time-crampin} and $\Gamma$ is variational.
\end{proof}

\begin{remark}
Note that the Cartan $2$-form $\Omega_L$  (see Remark~\ref{Rem:ThetaL}) can be alternatively rewritten as  
 $\Omega_{L}=\omega+dE_{L}\wedge dt$, with $\omega=-d\left(dL\circ S-(i_{\frac{\partial}{\partial t}}dL\circ S)dt\right)$ and $E_{L}=\Delta(L)-L$.
 If we consider the Legendre transformation
 $Leg_L: \R\times TQ\rightarrow \R\times T^*Q$ locally given by $$Leg_L(t, q, \dot{q})=(t, q, \frac{\partial L}{\partial \dot{q}^i})$$ then 
 $\omega=-d (Leg_L)^{*}\theta_{Q}$, $\Delta(L)=i_{\Gamma_L}(Leg_L)^{*}\theta_{Q}$ and $dL\wedge dt=\mathcal{L}_{\Gamma_L}(Leg_L)^{*}\theta_{Q}\wedge dt$. These substitutions motivate the definition of $\Omega$ in the proof of Theorem \ref{Teop} above for an $F$ and $\Gamma$ arbitrary instead of $Leg_L$ and $\Gamma_L$. For more on the formulation of time-dependent Lagrangian mechanics see \cite{84CPT,ranada} .
\end{remark}

\begin{remark}
If we replace the trivial bundle $\mathbb{R}\times Q\longrightarrow \mathbb{R}$ by an arbitrary fiber bundle $\pi:E\longrightarrow \R$, then the first jet manifold, denoted by $J^{1}E$ is the generalization of $\mathbb{R}\times TQ$. The generalization of $\mathbb{R}\times T^{*}Q$ is $V^{*}\pi$, the dual bundle of the vertical bundle to $\pi$. $V^{*}\pi$ is also equipped with a Poisson structure that can be lifted to $TV^{*}\pi$. The picture in this case is as follows:
$$
\xymatrix{
TJ^{1}\pi \ar[rr]^{TF} & & TV^{*}\pi \\
\\
J^{1}\pi \ar[rr]^{F} \ar[uu]^{\Gamma} \ar[rruu]_{\gamma_{\Gamma,F}=TF\circ\Gamma} & & V^{*}\pi
}
$$
Now the variationality of $\Gamma$ could also be studied in terms of $\hbox{Im}(\gamma_{\Gamma,F})$ being Lagrangian in $(TV^{*}\pi,\left\{,\right\}^{T})$ \cite{GrKaGra2004}.

\end{remark}

\section{The inverse problem for constrained systems} \label{constraints-section}

In this section, we will study the extension of the inverse problem of the calculus of variations to the case of constrained systems. Consider a  submanifold $M$ of $TQ$ and a vector field $\Gamma$ on $M$ verifying the SODE condition, that is,
\[
S_x(\Gamma(x))=\Delta(x), \quad \forall x\in M.
\]
Nonholonomic mechanics is an example of this situation, as we will see later.

From now on, we assume that $M$ projects over the whole configuration manifold $Q$. 
Inspired by Theorem \ref{poi} we give the following definition:

\begin{defn}\label{Dfn:VarConstraint}
 A SODE $\Gamma$ on the submanifold $M$ of $TQ$ is variational if there exists an immersion $F: M\rightarrow T^*Q$ over $Q$ such that $\Sigma_{\Gamma,F}:=\hbox{Im}(\mu_{\Gamma, F})$ is an isotropic submanifold of $(T^*TQ, \omega_{TQ})$, where
$\mu_{\Gamma,F}=\alpha_Q \circ TF\circ\Gamma$.
\end{defn}
    $$
    \xymatrix{
    TM \ar[r]^{TF} & TT^{*}Q \ar[r]^{\alpha_Q} & T^{*}TQ \\
    M \ar[u]^{\Gamma} \ar[urr]_{\mu_{\Gamma,F}} \ar[r]^{F} & T^{*}Q
    }
    $$

Assume that  $M$ is determined by the constraints
$$
\dot{q}^{\alpha}=\psi^{\alpha}(q^{i},\dot{q}^{a}), \quad 1\leq \alpha\leq m,
$$
so $(q^i, \dot{q}^a)$ are local coordinates on $M$, $1\leq a\leq n-m$, $n=\dim Q$. Then the solutions of the SODE
 $\Gamma$ are now represented by the following system of differential equations
\begin{eqnarray*}
    \ddot{q}^{a}&=&\Gamma^{a}(q^{i},\dot{q}^{a}),\\
      \dot{q}^{\alpha}&=&\psi^{\alpha}(q^{i},\dot{q}^{a}).
    \end{eqnarray*}

    For each map
    $$
    \begin{array}{lccc}
    F: & M & \longrightarrow & T^{*}Q \\
    & (q^{i},\dot{q}^{a}) & \longmapsto & (q^{i},F_{j}(q^{i},\dot{q}^{a}))
    \end{array}
    $$
    satisfying that rank $\left(\frac{\partial F_{i}}{\partial \dot{q}^{a}}\right)=n-m$, the submanifold
    $\hbox{Im}(\alpha_Q\circ TF\circ\Gamma)=\hbox{Im}(\mu_{\Gamma,F})$ is given in coordinates by
    $$
    \left(q^{i},\dot{q}^{a},\psi^{\alpha},\frac{\partial F_{i}}{\partial q^{a}}\dot{q}^{a}
    +\frac{\partial F_{i}}{\partial q^{\alpha}}\psi^{\alpha}
    +\frac{\partial F_{i}}{\partial\dot{q}^{a}}\Gamma^{a},F_{i}\right).
    $$
    We look for an immersion $F: M\rightarrow T^*Q$ such that $\hbox{Im}(\mu_{\Gamma, F})$ is isotropic in $(T^*TQ, \omega_{TQ})$, that is, such that the following conditions are satisfied:
       \begin{eqnarray}
        0&=&\frac{\partial F_{a}}{\partial \dot{q}^{b}}+\frac{\partial \psi^{\alpha}}{\partial \dot{q}^{a}}
\frac{\partial F_{\alpha}}{\partial \dot{q}^{b}}-\frac{\partial F_{b}}{\partial\dot{q}^{a}}    -\frac{\partial \psi^{\alpha}}{\partial\dot{q}^{b}}\frac{\partial
  F_{\alpha}}{\partial\dot{q}^{a}}, \label{CH1}
   \\
0&=&\frac{\partial^{2}F_{i}}{\partial q^{j} \partial
    q^{b}}\dot{q}^{b}+\frac{\partial^{2}F_{i}}{\partial q^{j}\partial
    q^{\beta}}\psi^{\beta}+\frac{\partial F_{i}}{\partial
    q^{\beta}}\frac{\partial  \psi^{\beta}}{\partial
    q^{j}}+\frac{\partial^{2}F_{i}}{\partial
    q^{j}\partial\dot{q}^{b}}\Gamma^{b}+\frac{\partial F_{i}}{\partial
    \dot{q}^{b}}\frac{\partial\Gamma^{b}}{\partial q^{j}}+\frac{\partial\psi^{\alpha}}{\partial
    q^{i}}\frac{\partial F_{\alpha}}{\partial q^{j}} \nonumber
    \\
    &&-\frac{\partial^{2}F_{j}}{\partial q^{i} \partial
    q^{b}}\dot{q}^{b}-\frac{\partial^{2}F_{j}}{\partial q^{i}\partial
    q^{\beta}}\psi^{\beta}-\frac{\partial F_{j}}{\partial
    q^{\beta}}\frac{\partial  \psi^{\beta}}{\partial
    q^{i}}-\frac{\partial^{2}F_{j}}{\partial
    q^{i}\partial\dot{q}^{b}}\Gamma^{b}-\frac{\partial F_{j}}{\partial
   \dot{q}^{b}}\frac{\partial\Gamma^{b}}{\partial q^{i}}-\frac{\partial\psi^{\alpha}}{\partial
    q^{j}}\frac{\partial F_{\alpha}}{\partial q^{i}}, \label{CH2}
    \\
    0&=&\frac{\partial^{2}F_{i}}{\partial \dot{q}^{a} \partial
    q^{b}}\dot{q}^{b}+\frac{\partial F_{i}}{\partial
    q^{a}}+\frac{\partial^{2}F_{i}}{\partial \dot{q}^{a}\partial
    q^{\beta}}\psi^{\beta}+\frac{\partial F_{i}}{\partial q^{\beta}}
    \frac{\partial \psi^{\beta}}{\partial\dot{q}^{a}}+
    \frac{\partial^{2}F_{i}}{\partial\dot{q}^{a}\partial\dot{q}^{b}}\Gamma^{b}+
    \frac{\partial F_{i}}{\partial\dot{q}^{b}}\frac{\partial\Gamma^{b}}{\partial\dot{q}^{a}} \nonumber
    \\
    &&-\frac{\partial F_{a}}{\partial q^{i}}+
    \frac{\partial\psi^{\alpha}}{\partial q^{i}}\frac{\partial F_{\alpha}}{\partial \dot{q}^{a}}
    -\frac{\partial\psi^{\alpha}}{\partial\dot{q}^{a}}\frac{\partial F_{\alpha}}{\partial
    q^{i}}. \label{CH3}
    \end{eqnarray}
We will refer to them as \textit{constrained Helmholtz conditions}.

Now, we will see the relationship between
    $\hbox{Im}(\mu_{\Gamma, F})$ and the dynamics given by the SODE $\Gamma$ on $M$.
    Take the submanifold
     $\alpha_Q^{-1}(\hbox{Im}(\mu_{\Gamma,F}))=TF(\Gamma(M))$  of $TT^*Q$.
     Since $TT^*Q$ is a tangent bundle, we have dynamics related to any submanifold.
    In our case $TF(\Gamma(M))$ is given by
    $$
    \left(q^{i},F_{i}(q^j, \dot{q}^b),\dot{q}^{a},\psi^{\alpha}(q^j, \dot{q}^b),\frac{\partial F_{i}}{\partial q^{a}}\dot{q}^{a}
    +\frac{\partial F_{i}}{\partial q^{\alpha}}\psi^{\alpha}
    +\frac{\partial F_{i}}{\partial\dot{q}^{a}}\Gamma^{a}\right)
    $$
    in the typical coordinates in $TT^*Q$.
    Tangent curves to this submanifold satisfy the following equations
    $$
		\dot{q}^{\alpha}=\psi^{\alpha}(q^j, \dot{q}^b)\ \mbox{ and  }\ 
    \frac{d}{dt}F_{i}=\frac{\partial F_{i}}{\partial q^{a}}\dot{q}^{a}
    +\frac{\partial F_{i}}{\partial q^{\alpha}}\psi^{\alpha}
    +\frac{\partial F_{i}}{\partial\dot{q}^{a}}\Gamma^{a}.
    $$
    Then
        $$
    \frac{\partial
    F_{i}}{\partial\dot{q}^{a}}\left(\ddot{q}^{a}-\Gamma^{a}(q^j, \dot{q}^b)\right)=0.
    $$
    Since $\left(\frac{\partial
    F_{i}}{\partial\dot{q}^{a}}\right)$ is assumed to have maximal rank, we get
    $\ddot{q}^{a}=\Gamma^{a}(q^j, \dot{q}^b)$ and $\dot{q}^{\alpha}=\psi^{\alpha}(q^j, \dot{q}^b)$.
    In this case we have  seen that the isotropic submanifold $TF(\Gamma(M))=\alpha^{-1}_{Q}(\Sigma_{\Gamma,F})$ on $TT^*Q$ carries the original dynamics defined by the SODE $\Gamma$ on $M$.

Now we generalize the characterization of Theorem \ref{theorem-crampin} to constrained systems:

\begin{thm} \label{omega-constraints}
A SODE $\Gamma$ on $M$ is variational if and only if there exists a $2$-form $\Omega$ on $M$ satisfying
\begin{enumerate}
\item d$\Omega=0$,
\item $\Omega(v_{1},v_{2})=0$ for all $v_{1},v_{2}\in V(M)$,
\item $\mathcal{L}_{\Gamma}\Omega=0$,
\item $\flat_{\Omega}|_{V(M)}$ is injective.
\end{enumerate}
\end{thm}

\begin{proof}
$\Rightarrow$ Assume that $\Gamma$ is variational, that is, there exists an immersion $F:M\longrightarrow T^{*}Q$ such that $\Sigma_{\Gamma, F}=\hbox{Im}(\mu_{\Gamma,F})$ is isotropic in $(T^{*}TQ,\omega_{TQ})$. Then we define $\Omega=dF^{*}\theta_{Q}\in\bigwedge^{2}(M)$. We first prove that $\Sigma_{\Gamma,F}$ is isotropic if and only if Im$(\mathcal{L}_{\Gamma}F^{*}\theta_{Q})$ is Lagrangian in $(T^*M, \omega_M)$, that is, $d(\mathcal{L}_{\Gamma}F^{*}\theta_{Q})=\mathcal{L}_{\Gamma}\Omega=0$. In local coordinates $(q^{i},\dot{q}^{a})$ on M,

$$
\mathcal{L}_{\Gamma}F^{*}\theta_{Q}=\left(\Gamma(F_{i})+\frac{\partial\psi^{\alpha}}{\partial q^{i}}F_{\alpha}\right)dq^{i}+\left(F_{a}+\frac{\partial\psi^{\alpha}}{\partial\dot{q}^{a}}F_{\alpha}\right)d\dot{q}^{a}.
$$
On the other hand, $\Sigma_{\Gamma,F}$ is given by the following set of points of $T^*TQ$:
$$
\lp q^{i}, \dot{q}^{a}, \psi^{\alpha}, \frac{\partial F_{i}}{\partial q^{a}}\dot{q}^{a}+\frac{\partial F_{i}}{\partial q^{\alpha}}\psi^{\alpha}+\frac{\partial F_{i}}{\partial \dot{q}^{a}}\Gamma^{a}, F_{i} \rp.
$$
If we denote by $i_{\Sigma_{\Gamma,F}}:\Sigma_{\Gamma,F}\longrightarrow T^{*}TQ$ the inclusion, then
$$
i^{*}_{\Sigma_{\Gamma,F}}\theta_{TQ}=\Gamma(F_{i})dq^{i}+F_{a}d\dot{q}^{a}+F_{\alpha}d\psi^{\alpha}=\lp \Gamma(F_{i})+\frac{\partial \psi^{\alpha}}{\partial q^{i}} F_{\alpha}\rp dq^{i} + \lp F_{a}+\frac{\partial\psi^{\alpha}}{\partial\dot{q}^{a}}F_{\alpha}\rp d\dot{q}^{a}.
$$
Now it is clear that the condition of isotropy, $i^{*}_{\Sigma_{\Gamma,F}}\omega_{TQ}=0$, is equivalent to the condition of Im$(\mathcal{L}_{\Gamma}F^{*}\theta_{Q})$ being Lagrangian in $(T^*M, \omega_M)$, in other words,  $d\mathcal{L}_{\Gamma}F^{*}\theta_{Q}=\mathcal{L}_{\Gamma}\Omega=0$.

The first two properties in the statement of the Theorem follow directly from the definition of $\Omega$ and the last one from $F$ being an immersion. Indeed, $\Omega=\lp\frac{\partial F_{i}}{\partial q^{j}}\rp dq^{j}\wedge dq^{i}+\lp\frac{\partial F_{i}}{\partial \dot{q}^{a}}\rp d\dot{q}^{a}\wedge dq^{i}$, and for any $v_{1},v_{2}$ in $V(M)$, $i_{v_{1}}\Omega-i_{v_{2}}\Omega=\lp v_{1}^{a}-v_{2}^{a}\rp \lp \frac{\partial F_{i}}{\partial \dot{q}^{a}}\rp dq^{i}=0\mbox{ for all } i=1,\ldots,n$. As $\lp \frac{\partial F_{i}}{\partial \dot{q}^{a}}\rp$ has maximal rank, $v_{1}=v_{2}$.

$\Leftarrow$ Now, given a $2$-form on $M$
satisfying the conditions in the statement, we construct an
immersion that provides an isotropic submanifold $\Sigma_{\Gamma,F}$ of $(T^{*}TQ,\omega_{TQ})$.
Since $d\Omega=0$, locally we can write $\Omega=d\Theta$. Then using
the second condition we get that there exists a locally defined
function $f$ on $M$ such that $\Theta(v)=df(v)$ for each vertical
vector $v\in V(M)$. We can define $\tilde{\Theta}=\Theta-df$ which is a
semi-basic $1$-form on $M$, that is, it vanishes on vertical vectors
and can be written in coordinates as $\tilde{\Theta}=\mu_{i}dq^{i}$,
$\mu_{i}$ being functions on $M$. Moreover $d\tilde{\Theta}=\Omega$. Then
we define $F\colon M\to T^*Q$ by
$$
\langle F(m), v_{q} \rangle = \langle \tilde{\Theta}(m), w_{m}\rangle,
$$
where $m\in M$ and $w_{m}$ is any vector in $T_{m}M$ satisfying $T_{m}\tau_{Q}|_{M}(w_{m})=v_{q}$. This definition does not depend on the choice of $w_{m}$ since $\tilde{\Theta}$ vanishes on vertical vectors and it gives $\tilde{\Theta}=F^{*}\theta_{Q}$.

Since the 1-form $\mathcal{L}_{\Gamma}F^{*}\theta_{Q}\in\bigwedge^{1}(M)$ is closed, then $\hbox{Im}(\mathcal{L}_{\Gamma}F^{*}\theta_{Q})$ is a Lagrangian submanifold of $(T^*M, \omega_M)$. Having Proposition~\ref{Prop:TildeS} in mind, we obtain from it a Lagrangian submanifold of $(T^*TQ, \omega_{TQ})$
$$\widetilde{\hbox{Im}(\mathcal{L}_{\Gamma}F^{*}\theta_{Q})}=
\{\mu\in T^*TQ\; |\; i_M^*\mu\in \hbox{Im}(\mathcal{L}_{\Gamma}F^{*}\theta_{Q})\},
$$
where $i_M: M\hookrightarrow TQ$ is the canonical inclusion. 
In coordinates, $\widetilde{\hbox{Im}(\mathcal{L}_{\Gamma}F^{*}\theta_{Q})}$ is expressed as
$$
\lp q^{i}, \dot{q}^{a}, \psi^{\alpha}, \Gamma(F_{i})+\frac{\partial \psi^{\alpha}}{\partial q^{i}}F_{\alpha}-\frac{\partial \psi^{\alpha}}{\partial q^{i}}\tilde{p}_{\alpha}, F_{a}+\frac{\partial \psi^{\alpha}}{\partial\dot{q}^{a}}F_{\alpha}-\frac{\partial \psi^{\alpha}}{\partial \dot{q}^{a}}\tilde{p}_{\alpha}, \tilde{p}_{\alpha}  \rp.
$$

In particular for $\tilde{p}_{\alpha}=F_{\alpha}$ we have
$$
\hbox{Im}(\mu_{\Gamma,F})\subset \widetilde{\hbox{Im}(\mathcal{L}_{\Gamma}F^{*}\theta_{Q})}.
$$

As $\mathcal{L}_{\Gamma}\Omega=0$, we get that both $\hbox{Im}(\mathcal{L}_{\Gamma}F^{*}\theta_{Q})$ and $\widetilde{\hbox{Im}(\mathcal{L}_{\Gamma}F^{*}\theta_{Q})}$ need to be Lagrangian and therefore $\hbox{Im}(\mu_{\Gamma,F})$ is isotropic in $(T^*TQ, \omega_{TQ})$.

Finally since $\flat_{\Omega}|_{V(TM)}$ is injective and $d\tilde{\Theta}=\Omega$, $\lp \frac{\partial F_{i}}{\partial \dot{q}^{a}} \rp$ has maximal rank and $F$ is an immersion. Now we conclude that $\Gamma$ is variational according to Definition~\ref{Dfn:VarConstraint}.
\end{proof}

\begin{remark}
Note that in the proof above we have described a way to assign to each isotropic submanifold $\Sigma_{\Gamma,F}$ a Lagrangian submanifold that contains it and projects over the constraint submanifold, see Proposition~\ref{Prop:TildeS}. From $\mathcal{L}_{\Gamma}\Omega=0$ we obtain a locally defined function $l:M\longrightarrow \R$ such that $\mathcal{L}_{\Gamma}F^{*}\theta_{Q}=dl$. Since $\widetilde{\hbox{Im}(\mathcal{L}_{\Gamma}F^{*}\theta_{Q})}$ coincides with $\Sigma_{l}= \lc \mu\in T^{*}TQ: i^{*}\mu=dl \rc\subset T^{*}TQ$, the construction from Theorem \ref{thm:tulczyjew}, it gives the constrained variational  dynamics associated to $l$ (see Section \ref{VC}). Summing up, given a variational SODE $\Gamma$ on $M$, we can always find a local Lagrangian $l$ on $M$ such that the solutions of $\Gamma$ are constrained variational trajectories for $l$.

Note that in this case we were not adressing the question of finding a Lagrangian $L:TQ\longrightarrow \R$ such that the solutions of the nonholonomic equations for $L$ coincide with the solutions of $\Gamma$, but asking when the nonholonomic dynamics can be seen as constrained variational dynamics, see Sections~\ref{Sec:Nonholo} and~\ref{VC}.

%This discussion is specially relevant in the case of nonholonomic dynamics. In some particular examples, the nonholonomic trajectories of a system can be included in the %constrained variational  trajectories for the same Lagrangian. In such a case, if $l$ denotes the restriction of the original  Lagrangian $L: TQ\rightarrow \R$, the nonholonomic %dynamics will correspond to an isotropic submanifold of the Lagrangian submanifold $\Sigma_{l}\subset T^{*}TQ$ which projects projecting over $M$. See, for instance, %\cite{2008BlochFer,93BlochCrouch,2002Cortes}.
%By lemma \ref{lemma-isotropa}, the isotropic submanifold can be extended locally to the image of a closed $1$-form on $TQ$, that is, the original trajectories are restricted free %trajectories for a local Lagrangian on $TQ$.

\end{remark}
%\begin{lemma}
%If $\Sigma_{\Gamma, F}$ is isotropic we can extend it to a
%Lagrangian submanifold $\mathcal{L}$ that satisfies $\mathcal{L}=dL$
%for some local Lagrangian function $L$ on $TQ$ (not necessarily regular).
%\end{lemma}

In the next, we will study the problem of how to derive a description of the constrained dynamics in terms of a variational problem without constraints (see \cite{2009BlochFerMestdag}, \cite{2010Crampin}). We will need the following Lemma.

\begin{lemma} \label{lemma-isotropa}
Let $P$ be a smooth manifold, $C\subset P$ a submanifold and $\gamma$ a section of $\left.T^{*}P\right|_{C}\longrightarrow C$, where $\left.T^{*}P\right|_{C}=\left\{\mu\in T^{*}P: \pi_{P}(\mu)\in C \right\}$ and $\pi_{P}:T^{*}P\longrightarrow P$ denotes the projection over $P$. If $\gamma(C)$ is isotropic in $(T^{*}P,\omega_{P})$, then there is a $1$-form $\tilde{\gamma}$ defined in a neighborhood of $C$ such that
\begin{itemize}
\item $\left.\tilde{\gamma}\right|_{C}=\gamma$,
\item $d\tilde{\gamma}=0$.
\end{itemize}		
\end{lemma}

\begin{proof}
Take adapted coordinates $(x^{i},y^{a})$, $i=1,\ldots,n-m,\ a=1,\ldots,m$, on $P$ such that $C$ is given by $y^{a}=0$ and denote the corresponding momenta coordinates by $p_{i}$ and $\tilde{p}_{a}$. Then $\gamma(C)$ is given by
$$
\left( x^{i}, 0, \gamma_{i}(x), \tilde{\gamma}_{a}(x) \right),
$$
and it projects over $C$. The isotropy condition gives $\frac{\partial \gamma_{i}}{\partial x^{j}}=\frac{\partial \gamma_{j}}{\partial x^{i}}$. 
We want to see $\gamma(C)$ inside some submanifold $N$ of $T^{*}P$ of dimension $2n-m$ and then apply the construction at the end of Section \ref{Sec:Lagrang} to extend it to a Lagrangian submanifold via the Hamiltonian vector fields corresponding to the constraints defining $N$. For that we have many options, for instance we can choose among the constraints
$$
y^{a}=0, \ \ p_{i}-\gamma_{i}=0, \ \ \tilde{p}_{a}-\tilde{\gamma}_{a}=0
$$
and linear combinations of them. If we consider $\phi_{a}=\tilde{p}_{a}-\tilde{\gamma}_{a}$ the Hamiltonian vector field is given by
$$
X_{\phi_{a}}=\frac{\partial}{\partial y^{a}}+\frac{\partial \tilde{\gamma}_{a}}{\partial x^{j}}\frac{\partial}{\partial p_{j}},
$$
which satisfies $X_{\phi_{a}}(y_{a})=1$, so it is not tangent to  $\gamma(C)$. Extending $\gamma(C)$ along the flows of $X_{\phi_{a}}$ we obtain
$$
\lp x^{i}, y^{a}, \gamma_{i}(x)+\frac{\partial \tilde{\gamma}_{a}}{\partial x^{i}}y^{a}, \tilde{\gamma}_{a}(x) \rp,
$$
which is the image of $\tilde{\gamma}=dL$ with $L\colon P \to \mathbb{R}$, $L(x,y)=\tilde{\gamma}_{a}(x)y^{a}+f(x)$, not necessarily regular, and $\frac{\partial f}{\partial x^{i}}=\gamma_{i}$. The existence of such a function $f$ on $C$ is guaranteed by the isotropy condition. 
\end{proof}

\begin{remark}
Note that there are many possible ways to choose the constraints and construct Lagrangians. For instance taking $\phi_{a}=y^{a}+\tilde{p}_{a}-\gamma_{a}$ we obtain $L=\gamma_{a}y^{a}+\gamma_{i}x^{i}-\frac{\partial \gamma_{a}}{\partial x^{j}}(y^{a})^{2}$. On the other hand, if we take $\phi_{a}=y^{a}$ then we obtain a Lagrangian submanifold projecting over $M$ which corresponds to the constrained variational description.
\end{remark}

As a consequence of Lemma \ref{lemma-isotropa}, taking $\gamma(C)=\Sigma_{\Gamma,F}$ we obtain the following important result.

\begin{thm} \label{extension-thm}
If a SODE $\Gamma$ on $M$ is variational, then there exists a Lagrangian
$L:TQ\longrightarrow\mathbb{R}$ such that
the integral curves of $\Gamma$ are the restriction of the solutions of the
Euler-Lagrange equations of $L$ to $M$.
\end{thm}

\begin{ex}
{\rm 
Let $Q=\mathbb{R}^{2}$ with coordinates $(x,y)$ and denote fibered coordinates on $TQ$ and $T^{*}TQ$ by $(x,y,\dot{x},\dot{y})$ and $(x,y,\dot{x},\dot{y},\mu_{x},\mu_{y},\tilde{\mu}_{x},\tilde{\mu}_{y})$ respectively. Let $N=\lc (x,y,\dot{x},f(x,y,\dot{x})) \rc \subset TQ$ be the constraint submanifold and the SODE $\Gamma$ on $N$ be given by $\ddot{x}=0$. That is, we have the dynamics given by
\[
\ddot{x}=0, \; \quad \dot{y}=f(x,y,\dot{x})\; .
\]
 We define $F:N\longrightarrow T^{*}Q$ by $F(x,y,\dot{x})=(x,y,\dot{x}+y,x)$, which is an immersion. Then $\Sigma_{\Gamma,F}\subset T^{*}TQ$ is locally described by $(x,y,\dot{x},f,f,\dot{x},\dot{x}+y,x)$ and is an isotropic submanifold of dimension 3, for $dx\wedge df+dy\wedge d\dot{x}+d\dot{x}\wedge d(\dot{x}+y)+df\wedge dx=0$. Note that
$$
\mathcal{L}_{\Gamma}F^{*}\theta_{Q}=\lp f+x\frac{\partial f}{\partial x}\rp dx + \lp \dot{x}+x\frac{\partial f}{\partial y} \rp dy + \lp \dot{x}+y+x\frac{\partial f}{\partial \dot{x}} \rp d\dot{x}\,.
$$
Therefore $\widetilde{\hbox{Im} \mathcal{L}_{\Gamma}F^{*}\theta_{Q}}\subset T^{*}TQ$  is locally described by
$$
\lp x,y,\dot{x},f,f+x\frac{\partial f}{\partial x}-\frac{\partial f}{\partial x}\tilde{\mu}_{y}, \dot{x}+x\frac{\partial f}{\partial y}-\frac{\partial f}{\partial y}\tilde{\mu}_{y}, \dot{x}+y+x\frac{\partial f}{\partial \dot{x}}-\frac{\partial f}{\partial \dot{x}}\tilde{\mu}_{y}, \tilde{\mu}_{y} \rp.
$$
When $\tilde{\mu}_{y}=x$, $\Sigma_{\Gamma,F}$  is recovered. Since $d\mathcal{L}_{\Gamma}F^{*}\theta_{Q}=0$, we have a local Lagrangian $l:N\longrightarrow\mathbb{R}$,  $l=\frac{\dot{x}^{2}}{2}+\dot{x}y+xf(x,y,\dot{x})$,  satisfying
$$
\frac{\partial l}{\partial x}=f+x\frac{\partial f}{\partial x},\ 
\frac{\partial l}{\partial y}=\dot{x}+x\frac{\partial f}{\partial y},\ 
\frac{\partial l}{\partial \dot{x}}=\dot{x}+y+x\frac{\partial f}{\partial \dot{x}}.
$$
Note that $l$ is the restiction of the singular Lagrangian $L_{1}=\frac{\dot{x}^{2}}{2}+\dot{x}y+x\dot{y}$ to $\dot{y}=f$.

Consider the constraint $\phi=\dot{y}-f+\tilde{\mu}_{y}-x$ and the
corresponding Hamiltonian vector field for the symplectic structure $\omega_{TQ}$:

$$X_{\phi}=-\frac{\partial}{\partial\tilde{\mu}_{y}}+\frac{\partial
f}{\partial \dot{x}}\frac{\partial}{\partial
\tilde{\mu}_{x}}+\frac{\partial f}{\partial
x}\frac{\partial}{\partial \mu_{x}}+\frac{\partial f}{\partial
y}\frac{\partial}{\partial \mu_{y}}+\frac{\partial}{\partial
\dot{y}}+\frac{\partial}{\partial \mu_{x}}.$$ 

If we extend the
isotropic submanifold $\Sigma_{\Gamma, F}$ along its flow we obtain the
Lagrangian submanifold
$$
\lp x,y, \dot{x}, \dot{y}, \dot{y}+\frac{\partial f}{\partial
x}\dot{y}-\frac{\partial f}{\partial x}f, \dot{x}+\frac{\partial
f}{\partial y}\dot{y}-\frac{\partial f}{\partial y}f,
\dot{x}+y+\frac{\partial f}{\partial \dot{x}}\dot{y}-\frac{\partial
f}{\partial \dot{x}}f, x-\dot{y}+f \rp,
$$
which is the image of $dL_{2}$ with
$L_{2}=x\dot{y}-\frac{\dot{y}^{2}}{2}+f\dot{y}+\frac{\dot{x}^{2}}{2}+\dot{x}y-\frac{f^{2}}{2}$, another extension of $l$.
However, this is a regular Lagrangian since $\det\lp
\frac{\partial^{2}L_{2}}{\partial \dot{q}^{i} \partial \dot{q}^{j}}
\rp=-1-\dot{y}\frac{\partial^{2}f}{\partial
\dot{x}^{2}}+f\frac{\partial^{2}f}{\partial\dot{x}^{2}}$, which does
not vanish in a neighborhood of $\Sigma_{\Gamma, F}$. It is possible to recover $\Gamma$
by computing the corresponding Euler-Lagrange equations and
restricting them to $M$. 
%Note that $\hbox{Im}(dL_{2})$ projects over the whole $T^{*}Q$ while $\hbox{Im}(dL_{1})$ does not.
}
\end{ex}

\begin{ex}[Vertical rolling disk]
{\rm 
Consider the configuration space $Q=S^{1}\times S^{1}\times \mathbb{R}^{2}$ with  coordinates $(\theta,\varphi,x,y)$, where $\theta$ denotes the angle of rotation, $\varphi$ the angle between the direction in which the disk moves and the $x$-axis and $(x,y)$ are the coordinates of the contact point. We consider the Lagrangian $L=\frac{1}{2}(\dot{\theta}^{2}+\dot{\varphi}^{2}+\dot{x}^{2}+\dot{y}^{2})$ and the constraints given by the condition of rolling without sliding are $\dot{x}=\cos(\varphi)\dot{\theta}$ and $\dot{y}=\sin(\varphi)\dot{\theta}$.

 We know that for the rolling disk the nonholonomic equations are
    $$
    \ddot{\theta}=0,\ \ddot{\varphi}=0,\
    \dot{x}=\cos(\varphi)\dot{\theta},\
    \dot{y}=\sin(\varphi)\dot{\theta},
    $$
    and the variational constrained ones are
    $$
    2\ddot{\theta}=\dot{\varphi}(-A\sin(\varphi)+B\cos(\varphi)), \
    \ddot{\varphi}=\dot{\theta}(A\sin(\varphi)-B\cos(\varphi)),
    $$
    $$
    \dot{x}=\cos(\varphi)\dot{\theta},\
    \dot{y}=\sin(\varphi)\dot{\theta},
    $$
    where $A$ and $B$ are constants (see \cite{93BlochCrouch}). Taking $A=B=0$ we see that the set of nonholonomic solutions is contained in the set of variational constrained ones. Now consider the constrained Lagrangian $l(\theta,\varphi,x,y,\dot{\theta}, \dot{\varphi})=\dot{\theta}^{2}+\frac{\dot{\varphi}^{2}}{2}$ and define $F$ as the Legendre transformation associated to the following extension $L(\theta,\varphi,x,y,\dot{\theta}, \dot{\varphi}, \dot{x}, \dot{y})=\dot{\theta}^{2}+\frac{\dot{\varphi}^{2}}{2}$, that is,
		
		$$
		\begin{array}{lccc}
		F\equiv Leg_L: & M & \longrightarrow & T^{*}Q \\
		   & (\theta,\varphi,x,y,\dot{\theta},\dot{\varphi}) & \longmapsto & (\theta,\varphi,x,y,2\dot{\theta},\dot{\varphi},0,0).
		\end{array}
		$$
As $\Gamma^{1}=\Gamma^{2}=0$, the submanifold $\Sigma_{\Gamma, F}\subset T^{*}TQ$ can be locally described by
$$
\lp \theta,\varphi,x,y,\dot{\theta},\dot{\varphi},\cos(\varphi)\dot{\theta},\sin(\varphi)\dot{\theta},0,0,0,0,2\dot{\theta},\dot{\varphi},0,0 \rp.
$$
It is isotropic and has dimension $6$, so we want to choose $2$ constraint functions on $T^{*}TQ$ satisfied by $\Sigma_{\Gamma, F}$ and extend it in the corresponding directions. First we take the constraints
$$
\phi_{1}=\dot{x}-\cos(\varphi)\dot{\theta}+\tilde{\mu}_{x}, \ \ \phi_{2}=\dot{y}-\sin(\varphi)\dot{\theta}+\tilde{\mu}_{y},
$$
with corresponding Hamiltonian vector fields
$$
X_{\phi_1}=-\frac{\partial}{\partial \tilde{\mu}_{x}}+\cos(\varphi)\frac{\partial}{\partial \tilde{\mu}_{\theta}}-\sin(\varphi)\dot{\theta}\frac{\partial}{\partial \mu_{\varphi}}+\frac{\partial}{\partial \dot{x}},
$$
$$
X_{\phi_2}=-\frac{\partial}{\partial \tilde{\mu}_{y}}+\sin(\varphi)\frac{\partial}{\partial \tilde{\mu}_{\theta}}+\cos(\varphi)\dot{\theta}\frac{\partial}{\partial \mu_{\varphi}}+\frac{\partial}{\partial \dot{y}}.
$$

Extending  $\Sigma_{\Gamma, F}$ along the flows of $X_{\phi_1}$ and $X_{\phi_2}$ we obtain the Lagrangian submanifold with local expression
$$
\lp \theta,\varphi,x,y,\dot{\theta},\dot{\varphi},\dot{x},\dot{y},0,\dot{\theta}(\cos(\varphi)\dot{y}-\sin(\varphi)\dot{x}), 0,0,\dot{\theta}+\cos(\varphi)\dot{x}+\sin(\varphi)\dot{y},\dot{\varphi},-\dot{x}+\cos(\varphi)\dot{\theta},-\dot{y}+\sin(\varphi)\dot{\theta}  \rp
$$
which is the image of $d\bar{L}$ with $\bar{L}=\frac{1}{2} (\dot{\theta}^{2}+\dot{\varphi}^{2}-\dot{x}^{2}-\dot{y}^{2}) +\dot{\theta}(\cos(\varphi)\dot{x}+\sin(\varphi)\dot{y})$. So we have obtained a regular Lagrangian whose unconstrained trajectories include the nonholonomic trajectories of the first Lagrangian. This is the same Lagrangian as the one obtained in \cite{2008BlochFer}.

If we take $\phi_{1}=\tilde{\mu}_{x}, \ \ \phi_{2}=\tilde{\mu}_{y}$ then we obtain the Lagrangian submanifold 
$$(\theta,\varphi,x,y,\dot{\theta},\dot{\varphi},\dot{x},\dot{y},0,0,0,0,2\dot{\theta},\dot{\varphi},0,0)$$ and recover the singular Lagrangian function $L=\dot{\theta}^{2}+\frac{\dot{\varphi}^{2}}{2}$.

For $\phi_{1}=\dot{x}-\cos(\varphi)\dot{\theta}, \ \ \phi_{2}=\dot{y}-\sin(\varphi)\dot{\theta}$ we get the Lagrangian submanifold 
$$\lp\theta,\varphi,x,y,\dot{\theta},\dot{\varphi},\cos(\varphi)\dot{\theta},\sin(\varphi)\dot{\theta},0,\dot{\theta}(\tilde{\mu}_{x}\sin(\varphi)-\tilde{\mu}_{y}\cos(\varphi)),0,0, \right.
$$
$$
\left.
2\dot{\theta}-\tilde{\mu}_{x}\cos(\varphi)-\tilde{\mu}_{y}\sin(\varphi),\dot{\varphi},\tilde{\mu}_{x},\tilde{\mu}_{y}\rp,$$
which coincides with $\widetilde{\hbox{Im}(\mathcal{L}_{\Gamma}F^{*}\theta_{Q})}$, for $\frac{\partial \psi^{1}}{\partial \dot{\theta}}=\cos(\varphi)$ and $\frac{\partial \psi^{2}}{\partial \dot{\theta}}=\sin(\varphi)$, where $\psi^{1}=\cos(\varphi)\dot{\theta},\ \psi^{2}=\sin(\varphi)\dot{\theta}$. Therefore, we obtain the variational constrained equations for the constrained Lagrangian $l: M\rightarrow \R$.

Now we find another immersion $F:M\longrightarrow T^{*}Q$ that makes $\Sigma_{\Gamma,F}$ isotropic. After extending it we get new Lagrangian functions defined on $TQ$.

We make the following assumptions on the dependence of coordinates of $F$ 
$$
F_{\theta}(\dot{\theta},\dot{\varphi})=F_{x}(\dot{\theta},\dot{\varphi})=F_{y}(\dot{\theta},\dot{\varphi}), \ \ F_{\varphi}(\varphi,\dot{\theta},\dot{\varphi}).
$$
Then the only constrained Helmholtz equations  (\ref{CH1}), (\ref{CH2}) and (\ref{CH3}) that do not vanish identically are
\begin{eqnarray}
\frac{\partial F_{\varphi}}{\partial\dot{\theta}}&=&\left( 1+\cos(\varphi)+\sin(\varphi) \right)\frac{\partial F_{\theta}}{\partial\dot{\varphi}},\\
0&=&\dot{\varphi}\frac{\partial^{2} F_{\varphi}}{\partial\dot{\theta}\partial\varphi}+\dot{\theta}(\cos(\varphi)-\sin(\varphi))\frac{\partial F_{x}}{\partial \dot{\theta}},\\
\frac{\partial F_{\varphi}}{\partial \varphi}&=&\frac{\partial}{\partial \dot{\varphi}}\left( \frac{\partial F_{\varphi}}{\partial\varphi}\dot{\varphi} \right)+\dot{\theta}(\cos(\varphi)-\sin(\varphi))\frac{\partial F_{x}}{\partial \dot{\varphi}},
\end{eqnarray}
and $F_{\theta}=F_{x}=F_{y}=\frac{\dot{\theta}}{\dot{\varphi}}$, $F_{\varphi}=\rho(\dot{\varphi})-\frac{\dot{\theta}^{2}}{2\dot{\varphi}^{2}}(1+\cos(\varphi)+\sin(\varphi))$ is a solution, where $\rho: \R\rightarrow \R$ is arbitrary.

Setting $\rho(\dot{\varphi})=\dot{\varphi}$, define 
$$
\begin{array}{lccc}
F: & M & \longrightarrow & T^{*}Q \\
& (\theta, \varphi, x, y, \dot{\theta},\dot{\varphi}) & \longmapsto & \left(\theta,\varphi,x,y,\frac{\dot{\theta}}{\dot{\varphi}}, \dot{\varphi}-\frac{\dot{\theta}^{2}}{2\dot{\varphi}^{2}}\left( 1+\cos(\varphi)+\sin(\varphi) \right),\frac{\dot{\theta}}{\dot{\varphi}},\frac{\dot{\theta}}{\dot{\varphi}}\right)
\end{array}
$$
to get $\Sigma_{\Gamma, F}$ given by
$$
\left( \theta, \varphi, x, y, \cos(\varphi)\dot{\theta}, \sin(\varphi)\dot{\theta}, \dot{\theta}, \dot{\varphi}, 0, \frac{1}{2}\frac{\dot{\theta}^{2}}{\dot{\varphi}}(\sin(\varphi)-\cos(\varphi)),0,0, \frac{\dot{\theta}}{\dot{\varphi}}, \dot{\varphi}-\frac{\dot{\theta}^{2}}{2\dot{\varphi}^{2}}(1+\cos(\varphi)+\sin(\varphi)), \frac{\dot{\theta}}{\dot{\varphi}}, \frac{\dot{\theta}}{\dot{\varphi}}\right),
$$
which is isotropic of dimension 6 on $(T^*TQ, \omega_{TQ})$ .

If we take $\phi_{1}=\tilde{\mu}_{x}-\frac{\dot{\theta}}{\dot{\varphi}}$ and $\phi_{2}=\tilde{\mu}_{y}-\frac{\dot{\theta}}{\dot{\varphi}}$, the corresponding Hamiltonian vector fields are
\begin{eqnarray*}
X_{\phi_1}&=& \frac{\partial}{\partial\dot{x}}+\frac{1}{\dot{\varphi}}\frac{\partial}{\partial \tilde{\mu}_{\theta}}-\frac{\dot{\theta}}{\dot{\varphi}^{2}}\frac{\partial}{\partial\tilde{\mu}_{\varphi}}, \\
X_{\phi_2}&=& \frac{\partial}{\partial\dot{y}}+\frac{1}{\dot{\varphi}}\frac{\partial}{\partial \tilde{\mu}_{\theta}}-\frac{\dot{\theta}}{\dot{\varphi}^{2}}\frac{\partial}{\partial\tilde{\mu}_{\varphi}}.
\end{eqnarray*}
Extending $\Sigma_{\Gamma,F}$ along the flows of $X_{\phi_1}$ and $X_{\phi_2}$ we get
$$
\left(
\theta, \varphi, x, y, \dot{\theta}, \dot{\varphi}, \dot{x}, \dot{y}, 0, \frac{1}{2}\frac{\dot{\theta}^{2}}{\dot{\varphi}}(\sin(\varphi)-\cos(\varphi)), 0, 0, \frac{\dot{\theta}}{\dot{\varphi}}(1-\sin(\varphi)-\cos(\varphi))+\frac{\dot{x}+\dot{y}}{\dot{\varphi}},
\right.
$$
$$
\left.
\dot{\varphi}-\frac{\dot{\theta}^{2}}{2\dot{\varphi}}(1-\sin(\varphi)-\cos(\varphi))-\frac{\dot{\theta}}{\dot{\varphi}^{2}}(\dot{x}+\dot{y}), \frac{\dot{\theta}}{\dot{\varphi}}, \frac{\dot{\theta}}{\dot{\varphi}}
\right),
$$
which is the image of $d\bar{L}$ for the singular Lagrangian
$$
\bar{L}=\frac{\dot{\varphi}^{2}}{2}+\frac{\dot{\theta}^{2}}{\dot{\varphi}}\left( \frac{1}{2}-\cos(\varphi)-\sin(\varphi) \right)+\frac{\dot{\theta}}{\dot{\varphi}}(\dot{x}+\dot{y}).
$$

Now we choose constraints $\phi_1=\dot{x}-\cos(\varphi)\dot{\theta}+\tilde{\mu}_{x}-\frac{\dot{\theta}}{\dot{\varphi}}$, $\phi_{2}=\dot{y}-\sin(\varphi)\dot{\theta}+\tilde{\mu}_{y}-\frac{\dot{\theta}}{\dot{\varphi}}$ with Hamiltonian vector fields
\begin{eqnarray*}
X_{\phi_1}&=&-\frac{\partial}{\partial \tilde{\mu}_{x}}+\left( \cos(\varphi)+\frac{1}{\dot{\varphi}} \right)\frac{\partial}{\partial\tilde{\mu}_{\theta}}-\frac{\dot{\theta}}{\dot{\varphi}^{2}}\frac{\partial}{\partial\tilde{\mu}_{\varphi}}-\dot{\theta}\sin(\varphi)\frac{\partial}{\partial \mu_{\varphi}}+\frac{\partial}{\partial\dot{x}}, \\
X_{\phi_2}&=&-\frac{\partial}{\partial \tilde{\mu}_{y}}+\left( \sin(\varphi)+\frac{1}{\dot{\varphi}} \right)\frac{\partial}{\partial\tilde{\mu}_{\theta}}-\frac{\dot{\theta}}{\dot{\varphi}^{2}}\frac{\partial}{\partial\tilde{\mu}_{\varphi}}+\dot{\theta}\cos(\varphi)\frac{\partial}{\partial \mu_{\varphi}}+\frac{\partial}{\partial\dot{y}}.
\end{eqnarray*}
Extending $\Sigma_{\Gamma, F}$ along their flows we obtain
$$
\left( \theta, \varphi, x, y, \dot{\theta}, \dot{\varphi}, \dot{x}, \dot{y}, 0, \frac{1}{2}\frac{\dot{\theta}^{2}}{\dot{\varphi}}(\sin(\varphi)-\cos(\varphi))-\dot{x}\dot{\theta}\sin(\varphi)+\dot{\theta}\dot{y}\cos(\varphi), 0, 0, \right.
$$
$$
\frac{\dot{\theta}}{\dot{\varphi}}(1-\cos(\varphi)-\sin(\varphi))+\dot{x}\cos(\varphi)-\dot{\theta}+\frac{\dot{x}+\dot{y}}{\dot{\varphi}}+\dot{y}\sin(\varphi), \dot{\varphi}-\frac{\dot{\theta}^{2}}{2\dot{\varphi}^{2}}(1-\cos(\varphi)-\sin(\varphi))-\frac{\dot{\theta}}{\dot{\varphi}^{2}}(\dot{x}+\dot{y}),
$$
$$
\left.
\frac{\dot{\theta}}{\dot{\varphi}}-\dot{x}+\cos(\varphi)\dot{\theta}, \frac{\dot{\theta}}{\dot{\varphi}}-\dot{y}+\sin(\varphi)\dot{\theta}
\right),
$$
which is the image of $d\bar{L}$ for 
$$
\bar{L}=\frac{1}{2}\left(\dot{\varphi}^{2}-\dot{\theta}^{2}-\dot{x}^{2}-\dot{y}^{2}\right)+\frac{\dot{\theta}^{2}}{2\dot{\varphi}}\left( 1-\cos(\varphi)-\sin(\varphi) \right)+\dot{\theta}\dot{x}\left( \cos(\varphi)+\frac{1}{\dot{\varphi}} \right)+\dot{\theta}\dot{y}\left( \sin(\varphi)+\frac{1}{\dot{\varphi}} \right).
$$ 
As  $\det\left(\frac{\partial^{2}L}{\partial\dot{q}^{i}\partial\dot{q}^{j}}\right)=\frac{1}{\dot{\varphi}^{5}}\left(-2\dot{\theta}^{2}(1-\sin(\varphi)-\cos(\varphi))-\dot{\theta}^{2}\dot{\varphi}+2\dot{\varphi}^{3}+\dot{\varphi}^{4}(1+\sin(\varphi)+\cos(\varphi)) \right)$, observe that this Lagrangian is regular except at a hypersurface of singular points.  

}

\end{ex}

\begin{ex}[Nonholonomic particle]
{\rm
Consider the system defined by $Q=\mathbb{R}^{3}$, $L=\frac{1}{2}(\dot{x}^{2}+\dot{y}^{2}+\dot{z}^{2})$ and constraint $\dot{z}=-x\dot{y}$. The nonholonomic SODE is given by $\Gamma^{1}=0, \Gamma^{2}=-\frac{x\dot{x}\dot{y}}{1+x^{2}}$. This SODE is  variational as a constrained system as we will see. Indeed, in  \cite{2009BlochFerMestdag} the authors show that this system can be represented as the restriction of the Euler-Lagrange vector field associated to a Lagrangian defined on the full space $TQ$. In our framework, we define the map
$$
\begin{array}{lccc}
F: & M & \longrightarrow & T^{*}Q \\
& (x,y,z,\dot{x},\dot{y}) & \longmapsto & \left( x,y,z,\dot{x}-\frac{\dot{y}^{2}}{2\dot{x}^{2}}\sqrt{1+x^{2}}\left(1+x\right), \frac{\sqrt{1+x^{2}}\dot{y}}{\dot{x}}, -\frac{\sqrt{1+x^{2}}\dot{y}}{\dot{x}} \right),
\end{array}
$$
then $\Sigma_{\Gamma,F}$ is given by
$$
\left(
x,y,z,\dot{x},\dot{y},-x\dot{y},-\frac{\dot{y}^{2}}{2\dot{x}}\frac{(1-x)}{\sqrt{1+x^{2}}},0,0,
\dot{x}-\frac{\dot{y}^{2}}{2\dot{x}^{2}}\sqrt{1+x^{2}}\left(1+x\right), \frac{\sqrt{1+x^{2}}\dot{y}}{\dot{x}}, -\frac{\sqrt{1+x^{2}}\dot{y}}{\dot{x}}
\right)
$$
and is isotropic in $(T^*TQ, \omega_{TQ})$. Also $\mathcal{L}_{\Gamma}F^{*}\theta_{Q}=dl$ for
$$
l=\frac{\dot{x}^{2}}{2}+\frac{\dot{y}^{2}}{2}\frac{\sqrt{1+x^{2}}}{\dot{x}}(1+x).
$$
Note that $l\not=\left.L\right|_{M}=\frac{1}{2}(\dot{x}^{2}+\dot{y}^{2}(1+x^{2}))$. Since $\Sigma_{\Gamma,F}\subset\Sigma_{l}$, the solutions of $\Gamma$ can be seen as constrained variational for $l$, although not for $\left.L\right|_{M}$ (see \cite{2008BlochFer}).

Now we look for a Lagrangian on $TQ$. Taking $\phi=\tilde{\mu}_{z}+\frac{\sqrt{1+x^{2}}\dot{y}}{\dot{x}}$ as constraint and extending $\Sigma_{\Gamma,F}$ along the flow of
$$
X_{\phi}=\frac{\partial}{\partial\dot{z}}-\frac{\sqrt{1+x^{2}}}{\dot{x}}\frac{\partial}{\partial\tilde{\mu}_{y}}+\frac{\sqrt{1+x^{2}}\dot{y}}{\dot{x}^{2}}\frac{\partial}{\partial \tilde{\mu}_{x}}-\frac{\dot{y}}{\dot{x}}\frac{x}{\sqrt{1+x^{2}}}\frac{\partial}{\partial \mu_{x}}
$$
we get
$$
\left( x,y,z,\dot{x},\dot{y},\dot{z}, -\frac{\dot{y}^{2}}{2\dot{x}}\frac{(1-x)}{\sqrt{1+x^{2}}}-\frac{\dot{y}}{\dot{x}}\frac{x}{\sqrt{1+x^{2}}}(\dot{z}+x\dot{y}),0,0, \dot{x}+\frac{\sqrt{1+x^{2}}}{2\dot{x}^{2}}\dot{y}^{2}(x-1),
\right.
$$
$$
\left.
\frac{\sqrt{1+x^{2}}}{\dot{x}}\dot{y}(1-x)-\frac{\sqrt{1+x^{2}}}{\dot{x}}\dot{z},-\frac{\sqrt{1+x^{2}}}{\dot{x}}\dot{y}   \right),
$$
generated by the regular Lagrangian 
$$\bar{L}=\frac{\dot{x}^{2}}{2}+\frac{(1-x)\sqrt{1+x^{2}}}{2\dot{x}}\dot{y}^{2}-\frac{\sqrt{1+x^{2}}}{\dot{x}}\dot{z}\dot{y}.$$

}

\end{ex}

\section{The inverse problem for holonomic constraints} \label{holonomic}

A particular case of constrained systems is given by a submanifold $M$ of $TQ$ which is precisely a tangent bundle of a submanifold $N$ of $Q$, this is the case of {\bf  holonomic constraints}. In other words, $M=TN$.
In many cases of interest it is useful \emph{to work extrinsically}, that is, on the manifold $Q$ instead of \emph{intrinsically}, that is, on $N$. As a result, the system on $N$ is described in terms of a system on $Q$.
Assume that $TN$ is locally described by the vanishing of the constraints
\begin{equation*}
\psi^\alpha(q^a,q^\beta)=0 \mbox{  and  } \dfrac{\partial \psi^\alpha}{\partial q^a} \dot{q}^a+ \dfrac{\partial \psi^\alpha}{\partial q^\beta} \dot{q}^\beta=0, \quad 1\leq \alpha \leq m.
\end{equation*}
For simplicity and without loss of generality, we consider the local coordinates on $Q$ adapted to $N$ and the corresponding local coordinates on $TQ$ adapted to $TN$, so that
\begin{equation*}
N=\{(q^a,q^\alpha) \in Q \; | \; q^\alpha=0\}\, , \hbox{ and } 
TN=\{(q^a,q^\alpha,\dot{q}^a,\dot{q}^\alpha) \in TQ \; | \; q^\alpha=0, \; \dot{q}^\alpha=0\}\, ,
\end{equation*}
where $a=1,\dots, n-m$. The SODE $\Gamma$ on $TN$ is locally described by $$\Gamma(q^a,\dot{q}^a)=(q^a,\dot{q}^a,\dot{q}^a,\Gamma^a(q^b,\dot{q}^b)).$$ 

The difference between holonomic dynamics and the nonholonomic one considered in Section~\ref{constraints-section} is that $M=TN$ does not project over the entire $Q$. Thus, the notion of variational SODE for constrained systems in Definition 4.1 must be adapted, because if $M$ does not project over the entire $Q$, $F\colon M \rightarrow T^*Q$ might not be an immersion.

\begin{defn}
Let $\Gamma$ be a SODE along $M$ and assume that $N=\tau_Q(M)$ is a submanifold so that we have the canonical inclusion ${\rm i}_{TN} \colon TN \rightarrow TQ$. The SODE $\Gamma$ is variational if there exists a function $F\colon M \rightarrow T^*Q$ such that the map $({\rm i}_{TN}^* \circ F)_{|M \cap TN}\colon M \cap TN \rightarrow T^*N$ is an immersion and $\Sigma_{\Gamma,F}=\hbox{Im}(\alpha_{Q}\circ TF\circ\Gamma)$ is an isotropic submanifold of $(T^*TQ,\omega_{TQ})$, where ${\rm i}_{TN}^*$ is the transpose  map of ${\rm i}_{TN}$ as defined below in~\eqref{eq:iTN*}.
\end{defn}

%---------------------
%Let $N=\tau_Q(M)$ and assume that $N$ is a submanifold???? so that we have the canonical inclusion ${\rm i}_{TN} \colon TN \rightarrow TQ$,  the SODE along M is variational if there exists $F\colon M \rightarrow T^*Q$ such that the map $({\rm i}_{TN}^* \circ F)_{|M \cap TN}\colon M \cap TN \rightarrow T^*N$ is an immersion and $\Sigma_{\Gamma,F}$ is an isotropic submanifold of $(T^*TQ,\omega_{TQ})$, where ${\rm i}_{TN}^*$ is the transpose  map of ${\rm i}_{TN}$ as defined below. With this adapted notion of a variational SODE Theorem 4.2 is still true.
%--------------------

With this adapted notion of a variational SODE for holonomic constraints Theorem \ref{omega-constraints} can be also proved similarly as the proof in Section~\ref{constraints-section} when $M$ projects onto the entire $Q$.

Our interest now is to establish a relationship between the inverse problem without constraints when we work intrinsically on $TN$ and the inverse problem with the holonomic constraints, when we work extrinsically on $TQ$.

$$
    \xymatrix{
   T^*TN &  TT^*N \ar[l]_{\alpha_N} & TTN \ar[l]_{{\rm T}f}\ar[r]^{{\rm T}F}& TT^*Q \ar[r]^{\alpha_Q} & T^{*}TQ \\ & T^*N & TN\ar[l]^{f}  \ar[u]^{\Gamma} \ar[ull]_{\mu_{\Gamma,f}}\ar[urr]^{\mu_{\Gamma,F}}    \ar[r]_{F} & T^*Q&
    }
    $$
\begin{thm} \label{prop-holo} A SODE $\Gamma$ on $TN$ is variational for the inverse problem of the calculus of variations without constraints  if and only if it is variational along the submanifold $TN$ of $TQ$ in the inverse problem for constrained systems.
\proof 
$\Rightarrow$ If $\Gamma$ is variational for the unconstrained system on $TN$, then there exists a regular Lagrangian $l\colon TN \rightarrow \mathbb{R}$ whose solutions of the Euler-Lagrange equations are also integral curves of the SODE $\Gamma$ and vice-versa. The function $f\colon TN \rightarrow T^*N$ in the above diagram is the Legendre transformation of $l$, that is, $f(q,\dot{q})=\mbox{Leg}_{l} (q,\dot{q})=(q,\partial l/ \partial \dot{q})$. Moreover, $\hbox{Im}(\mu_{\Gamma, f})$ is a Lagrangian submanifold of $(T^*TN, \omega_N)$.

Let ${\rm i}_{TN}\colon TN \rightarrow TQ$ be the inclusion and consider an arbitrary fiber function $F: TN\rightarrow T^*Q$ such that the following diagram is commutative:
$$\xymatrix{
TN \ar[d]_{\mbox{Leg}_{l}} \ar[r]^{F}&T^*_NQ \ar[ld]^{{\rm i}_{TN}^*} \\ T^*N & 
}$$
where ${\rm i}_{TN}^*$ is the transpose map of ${\rm i}_{TN}$ defined by
\begin{equation}\label{eq:iTN*} \langle  {\rm i}_{TN}^*(p_{\tilde{q}}),v_q \rangle = \langle p_{\tilde{q}}, {\rm i}_{TN}(v_q)\rangle, \quad\end{equation}
where $ p_{\tilde{q}}\in T_N^*Q,$  $v_q\in TN$,   $(\tau_Q\circ {\rm i}_{TN})(v_q)=\pi_Q(p_{\tilde{q}})$, $q\in N,$ $\tilde{q}\in Q$, $\pi_N(q)=\tilde{q}$.

Since $l: TN\rightarrow {\mathbb R}$ is regular (that is, $\mbox{Leg}_{l}: TN\rightarrow T^*N$ is a local diffeomorphism), it is easy to deduce that  $F: TN \rightarrow T^*Q$ is an immersion. 
 In local coordinates, the function $F$ looks like
 \begin{equation*}
\begin{array}{rcl} F\colon TN & \longrightarrow & T^*Q\\
(q^a,\dot{q}^a) &\longmapsto & \left( q^a, 0, \frac{\partial l}{\partial \dot{q}^a},  F_{\alpha}(q^b, \dot{q}^b)\right)
\end{array}
\end{equation*}
 where $F_\alpha$ are  arbitrary functions on $TN$.

The local expression  in adapted coordinates of the submanifold $\hbox{Im}(\mu_{\Gamma,F})$ of $T^*TQ$ is
\begin{equation*}
 \left(q^a,0,\dot{q}^a,0; \dfrac{\partial^2 l}{\partial \dot{q}^a\partial q^b}\dot{q}^b+ \dfrac{\partial^2 l}{\partial \dot{q}^a\partial \dot{q}^b}\Gamma^b,
\dfrac{\partial F_\alpha}{\partial q^b}\dot{q}^b+ \dfrac{\partial F_\alpha}{\partial \dot{q}^b}\Gamma^b, \frac{\partial l}{\partial \dot{q}^a},F_\alpha\right)\, .
\end{equation*}
This submanifold is isotropic if $(\mu_{\Gamma,{F}})^*(\omega_{TQ})$ is equal to zero, equivalently,
\begin{equation*}
{\rm d} \left(  \dfrac{\partial^2 l}{\partial \dot{q}^a\partial q^b}\dot{q}^b+ \dfrac{\partial^2 l}{\partial \dot{q}^a\partial \dot{q}^b}\Gamma^b\right) \wedge {\rm d} q^a+{\rm d}\left(\dfrac{\partial l}{\partial \dot{q}^a}\right) \wedge {\rm d} \dot{q}^a=d^2l=0
\end{equation*}
because $\Gamma$ is the Euler-Lagrange vector field for $l: TN\rightarrow {\mathbb R}$, that is, locally 
\[
\dfrac{\rm d}{{\rm d} t} \left( \dfrac{\partial l}{\partial \dot{q}^a}\right)=\dfrac{\partial^2 l}{\partial \dot{q}^a\partial q^b}\dot{q}^b+ \dfrac{\partial^2 l}{\partial \dot{q}^a\partial \dot{q}^b}\Gamma^b=\frac{\partial l}{\partial q^a}\, .
\]
%Note that this condition is exactly the same one that guarantees that $\mu_{\Gamma,{\mathcal F}l}}=dl$ is a Lagrangian submanifold of $(T^*TN, \omega_N)$.

$\Leftarrow$ Assuming now that $\Gamma$ is variational for the inverse problem with constraints, then there exists $F\colon TN \rightarrow T^*Q$ such that the map $({\rm i}_{TN}^* \circ F)\colon TN \rightarrow T^*N$ is an immersion and $\hbox{Im}(\mu_{\Gamma, F})$ is isotropic in $(T^*TQ, \omega_{TQ})$. Now we find a solution of the inverse problem of the calculus of variations (without constraints) by taking $f=i_{TN}^*\circ F: TN\rightarrow T^*N$. In coordinates,  $f(q^a,\dot{q}^a)=(q^a,F_a(q^b,\dot{q}^b))$. Obviously, $\hbox{Im}(\mu_{\Gamma, f})$ is Lagrangian in $(T^*TN, \omega_{TN})$ and $f$ is a local diffeomorphism. 

This result can  be also proved intrinsically because $f$ and $F$ must make the following diagram commutative:
$$\xymatrix{
TN \ar[d]_{f} \ar[r]^{F}&T^*_NQ \ar[ld]^{{\rm i}_{TN}^*} \\ T^*N & 
}$$
Note that the diagram is commutative if $F_a=f_a$, but the remaining $F_\alpha$ are arbitrary. It can be easily proved that $f^*\theta_N=F^*\theta_Q$. Then the 2-form characterizing the inverse problem for the calculus of variations, Theorem 3.1, and the one characterizing the inverse problem for the constrained systems, Theorem 4.2, coincide. This concludes the proof.
\qed
\end{thm}

Let $\Gamma$ be a SODE on $TN$ which is the Euler-Lagrange vector field corresponding to a regular Lagrangian $l: TN\rightarrow {\mathbb R}$. Applying  Theorem \ref{prop-holo} we obtain an isotropic submanifold of $(T^*TQ, \omega_{TQ})$ by simply taking $\hbox{Im}(\mu_{\Gamma, F})$ for any  map $F: M\rightarrow T^*Q$ verifying
\[
{\rm i}_{TN}^*\circ F=\mbox{Leg}_{l},
\]
where $\mbox{Leg}_{l}: TN\rightarrow T^*N$ is the Legendre transformation associated to $l: TN\rightarrow {\mathbb R}$. 

%\begin{equation*}
%\begin{array}{rcl} F\colon TQ & \longrightarrow & T^*Q\\
%(q,v) &\longmapsto & \left( q, f_a(q^b,v^b), F_\alpha(q,v)\right)= \left( q, \dfrac{\partial l}{\partial v^a}(q^b,v^b), F_\alpha(q,v)\right)\, ,
%\end{array}
%\end{equation*}
%being $F_\alpha$ arbitrary functions on $TQ$. 

%Note that if the SODE $\Gamma$ is variational for the unconstrained system on $TN$, then ${\mathcal L}_\Gamma(f^*\Theta_N)={\rm d}l$ for a regular Lagrangian $l\colon TN\rightarrow \mathbb{R}$. From this Lagrangian $l$ we can construct the Lagrangian submanifold $S_{l, M}$  of $(T^*TQ,\omega_{TQ})$ whose local expression is:
%\begin{equation*}
%\left(q^a,0,v^a,0;\dfrac{\partial l}{\partial q^a},p_\alpha, \dfrac{\partial l}{\partial v^a}, \tilde{p}_\alpha \right)\, .
%\end{equation*}

Recall that in Section \ref{constraints-section} for the case of a submanifold projecting over the entire $Q$, we saw that a constrained variational SODE could be seen as the restriction of a variational SODE on $TQ$, Theorem \ref{extension-thm}. In order to do this we just need to find a Lagrangian submanifold projecting over the entire $TQ$ and containing $\hbox{Im}(\mu_{\Gamma,F})$ which in this case has the expression
$$
\left(q^a,0,\dot{q}^a,0; \dfrac{\partial^2 l}{\partial \dot{q}^a\partial q^b}\dot{q}^b+ \dfrac{\partial^2 l}{\partial \dot{q}^a\partial \dot{q}^b}\Gamma^b,
\dfrac{\partial F_\alpha}{\partial q^b}\dot{q}^b+ \dfrac{\partial F_\alpha}{\partial \dot{q}^b}\Gamma^b, \frac{\partial l}{\partial \dot{q}^a},F_\alpha\right).
$$

If we take a Lagrangian $L:TQ\rightarrow {\mathbb R}$ such that $L_{|TN}=l$ and verifying 
\[
\frac{\partial L}{\partial q^{\alpha}}=\frac{\partial^2 L}{\partial q^a\partial \dot{q}^{\alpha}}\dot{q}^a+\frac{\partial^2 L}{\partial \dot{q}^a\partial \dot{q}^{\alpha}}\Gamma^a 
\]
on $TN$, then we can define $F=\mbox{Leg}_{L}\big|_{TN} :TN\rightarrow T^*Q$ and get $\hbox{Im}(\mu_{\Gamma,F})\subset dL$.

For instance, in adapted coordinates to $TN$, we can take any Lagrangian $L\colon TQ\rightarrow\mathbb{R}$   of the form 
\begin{equation*}
L(q,\dot{q})=l(q^a,\dot{q}^a)+\frac{1}{2} (\dot{q}^{\alpha})^2 A_{\alpha}(q, \dot{q})+\frac{1}{2} (q^{\alpha})^2 B_{\alpha}(q, \dot{q}),
\end{equation*}
where $A_{\alpha}, B_{\alpha}\in C^{\infty}(TQ)$.
Obvioulsy
\[
F(q^a, \dot{q}^a)=(q^a, 0, \frac{\partial l}{\partial \dot{q}^a}, 0).
\]
Therefore, we conclude that the solutions of the holonomic problem given by $l$ are included in the solutions of $L$ with initial conditions given on $TN$.

\

\begin{ex}{\rm  Planar pendulum of length $h$ with a particle of mass $m$. In this case $TN=T\mathbb{S}^1$ and $TQ=T\mathbb{R}^2$. The local adapted coordinates are $(q^1,q^2)=(\theta,r-h)$. We consider the SODE $\Gamma$ on $T\mathbb{S}^1$ coming from the Lagrangian $l\colon T\mathbb{S}^1 \rightarrow \mathbb{R}$,
\begin{equation*}
l(\theta,\dot{\theta})=\dfrac{1}{2} mh^2 \dot{\theta}^2-mgh \cos\theta\, .
\end{equation*}
In this case $f(\theta,\dot{\theta})=(\theta,mh^2\dot{\theta})$ and we could take $F(\theta,\dot{\theta})=\left(\theta, 0,mh^2\dot{\theta},F_2(\theta, 0, \dot{\theta}, 0)\right)$. Proposition~\ref{prop-holo} guarantees that $\hbox{Im}(\mu_{\Gamma, F})$ is isotropic in $(T^* TQ,\omega_{TQ})$. 
A choice of Lagrangian $L\colon TQ \rightarrow \mathbb{R}$ associated with that $F$ is 
$$L=\dfrac{1}{2} mh^2 \dot{\theta}^2-mgh \cos\theta+\frac{1}{2}\dot{r}^2A(q,\dot{q})+\frac{1}{2}(r-h)^2 B(q, \dot{q})$$ and a regular one is, for instance,
$$L=l+\frac{1}{2}\dot{r}^2+B(q,\dot{q})(r-h)^2.$$}
\end{ex}

\section{Time-dependent inverse problem for constrained systems}\label{Sec:TimeConstrained}

Now let us extend time-dependent Helmholtz conditions reviewed in Section~\ref{Sec:NewTime} to constrained systems. Let $M\subset TQ$ be a submanifold projecting over the whole configuration manifold $Q$, and $\Gamma$ a SODE on $\mathbb{R}\times M$. If $(t,q^{i},\dot{q}^{a})$ denote coordinates on $\mathbb{R}\times M$, $i=1,\ldots,n=\mbox{dim}\, Q$, $a=1,\ldots,m\leq n$, then the solutions of $\Gamma$ are given by
$$\ddot{q}^a=\Gamma^{a}(t,q^{j},\dot{q}^{b}), \quad \dot{q}^{\alpha}=\psi^{\alpha}(t,q^{j},\dot{q}^{b}),
$$ 
where $\alpha=1,\ldots,n-m$. 

%$$
%\xymatrix{
%T(\mathbb{R}\times M) \ar[r]^/-20pt/{TF} & T(\mathbb{R}\times T^{*}Q)\cong T\mathbb{R}\times TT^{*}Q \ar[r]^/25pt/{pr_{2}} & TT^{*}Q \\
%&\\
%\mathbb{R}\times M \ar[uu]^{\Gamma} \ar[r]^{F} \ar[uur]_{\mu_{\Gamma,F}:=TF\circ\Gamma} & \mathbb{R}\times T^{*}Q \\
%}
%$$

%\begin{defn}
%Let $(P,\left\{,\right\})$ be a Poisson manifold and denote by $\Lambda$ the Poisson bivector. Let $N\subset P$ be a submanifold. We say that it is isotropic if 
%$$
%\Lambda(\alpha,\beta)=0\ \ \forall \alpha,\beta\in \sharp^{-1}(TN),
%$$
%equivalently if $\sharp(TN^{\circ})\supset TN\cap\mathcal{C}$.
%\end{defn}

As in Section \ref{constraints-section} we need to introduce the notion of isotropic submanifolds but now in the Poisson context (see Section
\ref{Sec:NewTime}).  

\begin{defn}[\cite{Uchino}]
Let $(P,\left\{,\right\})$ be a Poisson manifold and denote by $\sharp:T^{*}P\longrightarrow TP$ the morphism of vector bundles induced by the Poisson bivector. Let $N\subset P$ be a submanifold. We say that it is isotropic if 
$$
\sharp(TN^{\circ})\supseteq TN\cap\mathcal{C}.
$$
Recall that $\mathcal{C}=\hbox{Im}(\sharp)$ denotes the characteristic distribution.
\end{defn}

\begin{defn}\label{Defn:VariationalTimeC}
We say that a SODE $\Gamma$ on $\mathbb{R}\times M$ is variational if there is an immersion $F:\mathbb{R}\times M\longrightarrow \mathbb{R}\times T^{*}Q$ over $\mathbb{R}\times Q$ such that $\hbox{Im}(TF \circ \Gamma)$ is an isotropic submanifold of $(T(\mathbb{R}\times T^{*}Q),\left\{ ,\right\}^T)$.
\end{defn}

$$
\xymatrix{
T(\mathbb{R}\times M) \ar[r]^/-20pt/{TF} & T(\mathbb{R}\times T^{*}Q)\cong T\mathbb{R}\times TT^{*}Q \ar[r]^/25pt/{pr_{2}} & TT^{*}Q \\
&\\
\mathbb{R}\times M \ar[uu]^{\Gamma} \ar[r]^{F} \ar[uur]_{\quad \gamma_{\Gamma,F}:=TF\circ\Gamma} & \mathbb{R}\times T^{*}Q \\
}
$$

We will now impose the isotropy condition on $\hbox{Im}(\gamma_{\Gamma,F})$ to obtain the time-dependent Helmholtz conditions for constrained systems. In local coordinates $\gamma_{\Gamma,F}$ is given by
$$
\gamma_{\Gamma,F}(t,q^{i},\dot{q}^{a})=\left( t,q^{i},F_{i},1,\dot{q}^{a},\psi^{\alpha}, \Gamma(F_{i})=\frac{\partial F_{i}}{\partial t}+\dot{q}^{a}\frac{\partial F_{i}}{\partial q^{a}}+\psi^{\alpha}\frac{\partial F_{i}}{\partial q^{\alpha}}+\Gamma^{a}\frac{\partial F_{i}}{\partial\dot{q}^{a}} \right).
$$

We also have
\begin{eqnarray*}
T(\hbox{Im}(\gamma_{\Gamma,F}))\cap\mathcal{C}&=&\hbox{span}\left\{ V_{i}:=\frac{\partial}{\partial q^{i}}+\frac{\partial F_{j}}{\partial q^{i}}\frac{\partial}{\partial p_{j}}+\frac{\partial\psi^{\alpha}}{\partial q^{i}}\frac{\partial}{\partial\dot{q}^{\alpha}}+\frac{\partial \Gamma(F_{j})}{\partial q^{i}}\frac{\partial}{\partial\dot{p}_{j}}, \right. \\
&& \left. W_{a}:=\frac{\partial}{\partial\dot{q}^{a}}+\frac{\partial F_{i}}{\partial \dot{q}^{a}}\frac{\partial}{\partial p_{i}}+\frac{\partial\psi^{\alpha}}{\partial\dot{q}^{a}}\frac{\partial}{\partial\dot{q}^{\alpha}}+\frac{\partial \Gamma(F_{i})}{\partial\dot{q}^{a}}\frac{\partial}{\partial\dot{p}_{i}}
\right\},
\end{eqnarray*}
and
\begin{eqnarray*}
\sharp(T\hbox{Im}(\gamma_{\Gamma,F})^{\circ})&=&\hbox{span}\left\{ A_{i}:= \frac{\partial}{\partial\dot{q}^{i}}+\frac{\partial F_{i}}{\partial q^{j}}\frac{\partial}{\partial\dot{p}_{j}}+\frac{\partial F_{i}}{\partial\dot{q}^{a}}\frac{\partial}{\partial p_{a}}, B_{i}:= \frac{\partial}{\partial q^{i}}+\frac{\partial \Gamma(F_{i})}{\partial q^{j}}\frac{\partial}{\partial\dot{p}_{j}}+\frac{\partial \Gamma(F_{i})}{\partial\dot{q}^{a}}\frac{\partial}{\partial p_{a}}, \right. \\
&& \left. C^{\alpha}:= -\frac{\partial}{\partial p_{\alpha}}+\frac{\partial\psi^{\alpha}}{\partial q^{j}}\frac{\partial}{\partial\dot{p}_{j}}+\frac{\partial\psi^{\alpha}}{\partial\dot{q}^{a}}\frac{\partial}{\partial p_{a}}\right\}.
\end{eqnarray*}

Then the equations we obtain by imposing $T(\hbox{Im}(\gamma_{\Gamma,F}))\cap\mathcal{C}\subset \sharp(TN^{\circ})$ are:

\begin{eqnarray}
\frac{\partial F_{a}}{\partial\dot{q}^{b}}+\frac{\partial\psi^{\alpha}}{\partial\dot{q}^{a}}\frac{\partial F_{\alpha}}{\partial\dot{q}^{b}}&=&\frac{\partial F_{b}}{\partial\dot{q}^{a}}+\frac{\partial\psi^{\alpha}}{\partial\dot{q}^{b}}\frac{\partial F_{\alpha}}{\partial\dot{q}^{a}}, \label{TC1} \\
\frac{\partial \Gamma(F_{i})}{\partial q^{k}}+\frac{\partial F_{\alpha}}{\partial q^{k}}\frac{\partial\psi^{\alpha}}{\partial q^{i}}&=&\frac{\partial \Gamma(F_{k})}{\partial q^{i}}+\frac{\partial F_{\alpha}}{\partial q^{i}}\frac{\partial\psi^{\alpha}}{\partial q^{k}}, \label{TC2}\\
\frac{\partial F_{a}}{\partial q^{i}}+\frac{\partial\psi^{\alpha}}{\partial\dot{q}^{a}}\frac{\partial F_{\alpha}}{\partial q^{i}}&=&\frac{\partial \Gamma(F_{i})}{\partial\dot{q}^{a}}+\frac{\partial F_{\alpha}}{\partial\dot{q}^{a}}\frac{\partial\psi^{\alpha}}{\partial q^{i}}. \label{TC3}
\end{eqnarray}

Equations (\ref{TC1}) and (\ref{TC3}) are obtained by imposing that $W_{a}$ be in $\sharp(T\hbox{Im}(\gamma_{\Gamma,F})^{\circ})$, while (\ref{TC2}) and (\ref{TC3}) are the conditions that arise when imposing that $V_{i}$ be in $\sharp(T\hbox{Im}(\gamma_{\Gamma,F})^{\circ})$.

%\begin{eqnarray*}
%\sharp^{-1}(\hbox{Im}(\mu_{\Gamma,F}))&=&\hbox{span}\left\{ V_{i}=-d\dot{p}_{i}+\frac{\partial F_{j}}{\partial q^{i}}d\dot{q}^{j}-\frac{\partial\psi^{\alpha}}{\partial q^{i}}dp_{\alpha}+\frac{\partial G_{j}}{\partial q^{i}}dq^{j}+A(dt,dv_{t}),\right.\\
%&& \left. W_{a}=-dp_{a}+\frac{\partial F_{i}}{\partial\dot{q}^{a}}d\dot{q}^{i}-\frac{\partial\psi^{\alpha}}{d\dot{q}^{a}}dp_{\alpha}+\frac{\partial G_{i}}{\partial \dot{q}^{a}}dq^{i}+B(dt,dv_{t}) \right\}
%\end{eqnarray*}
%
%Then $\Lambda(V_{i},V_{k})=0, \Lambda(W_{a},W_{b})=0,\Lambda(W_{a},V_{i})=0$ give
%\begin{eqnarray*}
%\frac{\partial G_{i}}{\partial q^{k}}+\frac{\partial F_{\alpha}}{\partial q^{k}}\frac{\partial\psi^{\alpha}}{\partial q^{i}}&=&\frac{\partial G_{k}}{\partial q^{i}}+\frac{\partial F_{\alpha}}{\partial q^{i}}\frac{\partial\psi^{\alpha}}{\partial q^{k}},\\
%\frac{\partial F_{a}}{\partial\dot{q}^{b}}+\frac{\partial\psi^{\alpha}}{\partial\dot{q}^{a}}\frac{\partial F_{\alpha}}{\partial\dot{q}^{b}}&=&\frac{\partial F_{b}}{\partial\dot{q}^{a}}+\frac{\partial\psi^{\alpha}}{\partial\dot{q}^{b}}\frac{\partial F_{\alpha}}{\partial\dot{q}^{a}},\\
%\frac{\partial F_{a}}{\partial q^{i}}+\frac{\partial\psi^{\alpha}}{\partial\dot{q}^{a}}\frac{\partial F_{\alpha}}{\partial q^{i}}&=&\frac{\partial G_{i}}{\partial\dot{q}^{a}}+\frac{\partial F_{\alpha}}{\partial\dot{q}^{a}}\frac{\partial\psi^{\alpha}}{\partial q^{i}}.\\
%\end{eqnarray*}

\begin{thm}
A SODE $\Gamma$ on $\mathbb{R}\times M$ is variational if and only if there is a $2$-form $\Omega$ on $\mathbb{R}\times M$ such that
\begin{enumerate}
\item $d\Omega=0$, \label{TCi}
\item $\Omega(v_{1},v_{2})=0$, for all vertical vectors $v_{1}, v_{2}\in V(\mathbb{R}\times M)$, \label{TCii}
\item $i_{\Gamma}\Omega=0$, \label{TCiii}
\item $\left.\flat_{\Omega}\right|_{V(\mathbb{R}\times M)}$ is injective. \label{TCiv}
\end{enumerate}
\end{thm}

\begin{proof}
We prove this result using Theorem \ref{time-crampin}.

$\Rightarrow$ If $\Gamma$ is variational in the sense given in Definition~\ref{Defn:VariationalTimeC}, then we define a $2$-form on $\mathbb{R}\times M$ by
$$
\Omega=-dF^{*}\theta_{Q}+di_{\Gamma}F^{*}\theta_{Q}\wedge dt-\mathcal{L}_{\Gamma}F^{*}\theta_{Q}\wedge dt.
$$
Condition \ref{TCii} is readily satisfied and condition \ref{TCiii} can also be checked without making use of the conditions on $F$: 
\begin{eqnarray*}
i_{\Gamma}\Omega&=&-i_{\Gamma}dF^{*}\theta_{Q}+\overbrace{i_{\Gamma}(di_{\Gamma}F^{*}\theta_{Q}\wedge dt)}^{-di_{\Gamma}F^{*}\theta_{Q}+(i_{\Gamma}di_{\Gamma}F^{*}\theta_{Q})dt}-\overbrace{i_{\Gamma}(\mathcal{L}_{\Gamma}F^{*}\theta_{Q}\wedge dt)}^{-\mathcal{L}_{\Gamma}F^{*}\theta_{Q}+i_{\Gamma}\mathcal{L}_{\Gamma}F^{*}\theta_{Q}\, dt}\\
&=&-i_{\Gamma}dF^{*}\theta_{Q}-di_{\Gamma}F^{*}\theta_{Q}+(i_{\Gamma}di_{\Gamma}F^{*}\theta_{Q})dt+i_{\Gamma}dF^{*}\theta_{Q}\\
&&+di_{\Gamma}F^{*}\theta_{Q}-i_{\Gamma}(i_{\Gamma}dF^{*}\theta_{Q}+di_{\Gamma}F^{*}\theta_{Q})  dt\\
&=&-(i_{\Gamma}i_{\Gamma}dF^{*}\theta_{Q}) dt=0.
\end{eqnarray*}

Condition \ref{TCi} is equivalent to $d(\mathcal{L}_{\Gamma}F^{*}\theta_{Q}\wedge dt)=0$, and this is guaranteed by equations (\ref{TC1}), (\ref{TC2}) and (\ref{TC3}).
%$\hbox{Im}(\alpha_{Q}\circ pr_{2}\circ \mu_{\Gamma,F})$ being isotropic in $T^{*}TQ$.

Finally condition \ref{TCiv} is a consequence of $F$ being an immersion. This can be checked using local coordinates as in Theorem \ref{omega-constraints}. Now 
\begin{eqnarray*}
\Omega&=&-\frac{\partial F_{i}}{\partial q^{j}}dq^{j}\wedge dq^{i}-\frac{\partial F_{i}}{\partial \dot{q}^{a}}d\dot{q}^{a}\wedge dq^{i}\\
&& +\left( \frac{\partial F_{a}}{\partial q^{i}}\dot{q}^{a}+\frac{\partial F_{\alpha}}{\partial q^{i}}\psi^{\alpha}-\frac{\partial F_{i}}{\partial q^{a}}\dot{q}^{a}-\frac{\partial F_{i}}{\partial q^{\alpha}}\psi^{\alpha}-\frac{\partial F_{i}}{\partial \dot{q}^{a}}\Gamma^{a} \right)dq^{i}\wedge dt+\left( \frac{\partial F_{a}}{\partial\dot{q}^{b}}\dot{q}^{a}+\frac{\partial F_{\alpha}}{\partial\dot{q}^{b}}\psi^{\alpha} \right)d\dot{q}^{b}\wedge dt
\end{eqnarray*}
and therefore $i_{v_{1}}\Omega-i_{v_{2}}\Omega=-\frac{\partial F_{i}}{\partial\dot{q}^{a}}(v_{1}^{a}-v_{2}^{a})dq^{i}+\left( \frac{\partial F_{a}}{\partial\dot{q}^{b}}\dot{q}^{a}+\frac{\partial F_{\alpha}}{\partial\dot{q}^{b}}\psi^{\alpha} \right)(v_{1}^{b}-v_{2}^{b})dt$ for any $v_{1},v_{2}$ in $V(\mathbb{R}\times M)$. Since $\left( \frac{\partial F_{i}}{\partial\dot{q}^{a}} \right)$ is assumed to have maximal rank, $i_{v_{1}}\Omega=i_{v_{2}}\Omega$ implies $v_{1}=v_{2}$.

$\Leftarrow$
We proceed as in the proofs of Theorems \ref{poi} and \ref{omega-constraints} to get a local $1$-form $\tilde{\Theta}$ on $\mathbb{R}\times M$ such that $d\tilde{\Theta}=\Omega$ and $\tilde{\Theta}(v)=0$ for all vertical vector fields $v$. We define
%Now define $\bar{\Theta}=\tilde{\Theta}-\left(i_{\frac{\partial}{\partial t}}\tilde{\Theta}\right)dt$, which also satisfies $\bar{\Theta}\left(\frac{\partial}{\partial t}\right)=0$, and 
$$
\begin{array}{lccc}
F:&\mathbb{R}\times M & \longrightarrow & \mathbb{R}\times T^{*}Q \\
& (t,v_{q}) & \longmapsto & (t,\tilde{F}(t,v_{q})) 
\end{array}
$$ 
by
$$
\langle \tilde{F}(t,v_{q}), w_{q} \rangle = \langle pr_{2}\circ\tilde{\Theta}(t,v_{q}),W_{v_q} \rangle,
$$
where $v_{q}\in M, w_{q}\in TQ, W_{v_q}\in TM$ and $\left. T\tau_{Q}\right|_{M}(W_{v_q}) =w_{q}$.
$$
\xymatrix{
T^{*}(\mathbb{R}\times M) \ar[r]^{pr_{2}} & T^{*}M \\
\mathbb{R}\times M \ar[r]^{\tilde{F}} \ar[u]^{\tilde{\Theta}} & T^{*}Q
}
$$

We check that $\hbox{Im}(\gamma_{\Gamma,F})$ is isotropic using local coordinates. As $\tilde{\Theta}$ vanishes on vertical vectors, we can write 
$$
\tilde{\Theta}=F_{i}dq^{i}+\mu_{t}dt\, .
$$
Then
\begin{eqnarray*}
\Omega=-d\tilde{\Theta}&=&-dF_{i}\wedge dq^{i}-d\mu_{t}\wedge dt \\
&=& -\frac{\partial F_{i}}{\partial q^{j}}dq^{j}\wedge dq^{i}-\frac{\partial F_{i}}{\partial \dot{q}^{a}}d\dot{q}^{a}\wedge dq^{i}-\frac{\partial F_{i}}{\partial t}dt\wedge dq^{i}-\frac{\partial \mu_{t}}{\partial q^{j}}dq^{j}\wedge dt-\frac{\partial \mu_{t}}{\partial \dot{q}^{a}}d\dot{q}^{a}\wedge dt\, .
\end{eqnarray*}
By imposing the condition $i_{\Gamma}\Omega=0$ we get
\begin{eqnarray*}
\frac{\partial \mu_{t}}{\partial \dot{q}^{a}}&=&-\frac{\partial F_{b}}{\partial \dot{q}^{a}}\dot{q}^{b}-\frac{\partial F_{\alpha}}{\partial\dot{q}^{a}}\psi^{\alpha},\\
\frac{\partial \mu_{t}}{\partial q^{i}}&=&\Gamma(F_{i})-\frac{\partial F_{a}}{\partial q^{i}}\dot{q}^{a}-\frac{\partial F_{\alpha}}{\partial q^{i}}\psi^{\alpha},
\end{eqnarray*}
so we can write
$$
\Omega=-dF_{i}\wedge dq^{i}-\left[\left( \Gamma(F_{j})-\frac{\partial F_{a}}{\partial q^{j}}\dot{q}^{a}-\frac{\partial F_{\alpha}}{\partial q^{j}}\psi^{\alpha} \right)dq^{j} + \left( -\frac{\partial F_{b}}{\partial \dot{q}^{a}}\dot{q}^{b}-\frac{\partial F_{\alpha}}{\partial\dot{q}^{a}}\psi^{\alpha} \right)d\dot{q}^{a}\right]\wedge dt,
$$
and now the closedness of the second factor gives equations (\ref{TC1}), (\ref{TC2}) and (\ref{TC3}) for $F$. 

Finally we see that F is an immersion. Condition \ref{TCiv} states that
$$
0=i_{v_{1}}\Omega-i_{v_{2}}\Omega=-\frac{\partial F_{i}}{\partial \dot{q}^{a}}(v_{1}^{a}-v_{2}^{a})dq^{i}-\left( -\frac{\partial F_{\alpha}}{\partial\dot{q}^{a}}\psi^{\alpha}-\frac{\partial F_{b}}{\partial \dot{q}^{a}}\dot{q}^{b} \right)(v_{1}^{a}-v_{2}^{a})dt
$$
is satisfied
if and only if $v_{1}=v_{2}$. Since $\frac{\partial F_{i}}{\partial \dot{q}^{a}}(v_{1}^{a}-v_{2}^{a})=0$ implies $\left( -\frac{\partial F_{\alpha}}{\partial\dot{q}^{a}}\psi^{\alpha}-\frac{\partial F_{b}}{\partial \dot{q}^{a}}\dot{q}^{b} \right)(v_{1}^{a}-v_{2}^{a})=0$, we have that $\frac{\partial F_{i}}{\partial \dot{q}^{a}}(v_{1}^{a}-v_{2}^{a})=0$ implies $i_{v_{1}}\Omega-i_{v_{2}}\Omega=0$ and $v_{1}=v_{2}$, that is, $\left( \frac{\partial F_{i}}{\partial \dot{q}^{a}} \right)$ has maximal rank and $F$ is an immersion.

\end{proof}

\section{Conclusions and future developments}

The contributions of this paper include a characterization of the inverse problem of the calculus of variations in terms of special submanifolds in symplectic geometry; precisely, Lagrangian and isotropic submanifolds. 
Our approximation is flexible enough to take into account systems of second order differential equations with constraints, in particular, nonholonomic systems and their hamiltonization. Moreover, using symplectic techniques, we can prove that if a constrained explicit second order differential equation admits a solution of the inverse problem then it can always be represented by a Lagrangian system without constraints. This last system agrees with the constrained SODE along the submanifold of $TQ$ which gives the constraints (see  Theorem \ref{extension-thm}). 
We adapt our techniques to the case of explicit time-dependent SODE's now using Poisson techniques instead of the symplectic ones. 

As we said before, one of the advantages of our approach is the easy adaptability to different cases. In particular, in future work we will study the following extensions: 
\begin{itemize}
\item The inverse problem for reduced systems; in particular, Euler-Poincar\'{e} equations and Lagrange-Poincar\'e equations. In this case, we need to work with a notion of SODE over more general spaces than tangent bundles (for instance, $TQ/G$ where $G$ is a Lie group acting free and properly on the configuration manifold). To study this problem, we will use the Lie algebroid formalism developed in \cite{2005LeMaMa}.

\item We will carefully study the relationship between our techniques and hamiltonization of nonholonomic systems. This is useful to study invariance properties of the nonholonomic flow (preservation of a volume form, symmetries...).

\item Another interesting possibility is to extend our technique, always using Lagrangian and isotropic submanifolds, now for the symplectic cotangent bundle $(T^*Q\times T^*Q, \Omega)$ where $\Omega=pr_2^*\omega_Q-pr_1^*\omega_Q$ . This case will be useful to study the inverse problem for discrete systems, that is, when a second-order difference equation can be derived as the  flow associated to the discrete Euler-Lagrange equations for a discrete Lagrangian $L_d: Q\times Q\rightarrow \R$ (see \cite{2001MaWe}). Of course, we will have a version for reduced systems using similar techniques to the ones in the previous paragraph \cite{2006MaMaMa}. 
 \end{itemize}

\appendix 

\section{Appendix} \label{Appendix}

We will establish the equivalence between the equations for $\Sigma_{\Gamma,
    F}\subset T^{*}TQ$ to be Lagrangian and the Helmholtz
    conditions for $g_{ij}=\frac{\partial F_{i}}{\partial
    \dot{q}^{j}}$.

    The equations we obtain by imposing that the submanifold
    $\Sigma_{F,\Gamma}$ be Lagrangian, that is,
    \begin{eqnarray*}
    d\left(\left(\frac{\partial F_{i}}{\partial q^{j}}\dot{q}^{j}
    +\frac{\partial F_{i}}{\partial \dot{q}^{j}}\Gamma^{j}\right)dq^{i}+F_{i}d\dot{q}^{i}\right)&=&0,\\
    \end{eqnarray*}
    are the following:
    \begin{eqnarray}
    \frac{\partial F_{i}}{\partial \dot{q}^{j}}&=&\frac{\partial F_{j}}{\partial
    \dot{q}^{i}} \, ,\label{L1} \\
    \frac{\partial^{2}F_{i}}{\partial q^{j}\partial q^{k}}\dot{q}^{k}
    +\frac{\partial^{2}F_{i}}{\partial q^{j}\partial \dot{q}^{k}}\Gamma^{k}
    +\frac{\partial F_{i}}{\partial \dot{q}^{k}}\frac{\partial \Gamma^{k}}{\partial q^{j}}
    &=&\frac{\partial^{2} F_{j}}{\partial q^{i}\partial q^{k}}\dot{q}^{k}
    +\frac{\partial^{2} F_{j}}{\partial q^{i}\partial \dot{q}^{k}}\Gamma^{k}
    +\frac{\partial F_{j}}{\partial \dot{q}^{k}}\frac{\partial \Gamma^{k}}{\partial
    q^{i}}\, , \label{L2} \\
      \frac{\partial^{2} F_{i}}{\partial \dot{q}^{j}\partial q^{k}}\dot{q}^{k}
    +\frac{\partial F_{i}}{\partial q^{j}}
    +\frac{\partial^{2}F_{i}}{\partial \dot{q}^{j}\partial \dot{q}^{k}}\Gamma^{k}
    +\frac{\partial F_{i}}{\partial \dot{q}^{k}}\frac{\partial \Gamma^{k}}{\partial \dot{q}^{j}}-  \frac{\partial F_{j}}{\partial q^{i}}&=&0. \label{L3}
    \end{eqnarray}

		Assume $F$ is a local diffeomorphism that satisfies (\ref{L1}), (\ref{L2}) and (\ref{L3}). The first three sets of Helmholtz conditions~\eqref{eq:Helmholtz1}, that is, $det(g_{ij})\not=0$, $g_{ij}=g_{ji}$ and $\frac{\partial g_{ij}}{\partial \dot{q}^{k}}=\frac{\partial g_{ik}}{\partial \dot{q}^{j}}$, are readily satisfied by $g_{ij}=\left( \frac{\partial F_{i}}{\partial \dot{q}^{j}} \right)$.
		
		Taking the difference $(\ref{L3})_{ij}-(\ref{L3})_{ji}=0$ we get $a_{ij}=a_{ji}$, where $a_{ij}=\frac{\partial F_{i}}{\partial q^{j}}+\frac{1}{2}\frac{\partial F_{i}}{\partial \dot{q}^{k}}\frac{\partial \Gamma^{k}}{\partial \dot{q}^{j}}$. Then,

\begin{eqnarray*}
    \Gamma(g_{ij})-\nabla^k_jg_{ik}-\nabla^k_i g_{kj}-(\ref{L3})&=&\frac{\partial^{2}F_{i}}{\partial q^{k}\partial
    \dot{q}^{j}}\dot{q}^{k}+\frac{\partial^{2}F_{i}}{\partial \dot{q}^{k}\partial \dot{q}^{j}}\Gamma^{k}
    +\frac{1}{2}\frac{\partial F_{i}}{\partial \dot{q}^{k}}\frac{\partial\Gamma^{k}}{\partial
    \dot{q}^{j}}+\frac{1}{2}\frac{\partial F_{k}}{\partial \dot{q}^{j}}\frac{\partial \Gamma^{k}}{\partial \dot{q}^{i}}\\
    &&-\frac{\partial^{2} F_{i}}{\partial \dot{q}^{j}\partial q^{k}}\dot{q}^{k}
    -\frac{\partial F_{i}}{\partial q^{j}}
    -\frac{\partial^{2}F_{i}}{\partial \dot{q}^{j}\partial \dot{q}^{k}}\Gamma^{k}
    -\frac{\partial F_{i}}{\partial \dot{q}^{k}}\frac{\partial \Gamma^{k}}{\partial \dot{q}^{j}}
    +\frac{\partial F_{j}}{\partial q^{i}}\\
		&=&-\frac{\partial F_{i}}{\partial q^{j}}-\frac{1}{2}\frac{\partial F_{i}}{\partial \dot{q}^{k}}\frac{\partial \Gamma^{k}}{\partial \dot{q}^{j}}+
		\frac{\partial F_{j}}{\partial q^{i}}+\frac{1}{2}\frac{\partial F_{j}}{\partial \dot{q}^{k}}\frac{\partial \Gamma^{k}}{\partial \dot{q}^{i}}=a_{ji}-a_{ij}=0,		
		\end{eqnarray*}
		Thus the $\nabla$ condition~\eqref{eq:HelmholtzNabla} is satisfied.

Now we check that the $\Phi$ condition~\eqref{eq:HelmholtzPhi} is satisfied using $(\ref{L2})$ and the condition $a_{ij}=a_{ji}$. From the latter, we have
$$
\frac{\partial F_{i}}{\partial q^{j}}=\frac{\partial F_{j}}{\partial q^{i}}+\frac{1}{2}\frac{\partial F_{j}}{\partial \dot{q}^{k}}\frac{\partial \Gamma^{k}}{\partial \dot{q}^{i}}-\frac{1}{2}\frac{\partial F_{i}}{\partial \dot{q}^{k}}\frac{\partial \Gamma^{k}}{\partial \dot{q}^{j}}\, .
$$
Substituting this on the left hand side of $(\ref{L2})$ we get
\begin{equation}\label{eq:intermediateStep}
\begin{array}{l}
\ds{g_{jl}\left[ \frac{1}{2}\frac{\partial^{2}\Gamma^{l}}{\partial q^{k}\partial \dot{q}^{i}}\dot{q}^{k}+\frac{1}{2}\frac{\partial^{2}\Gamma^{l}}{\partial\dot{q}^{k}\partial\dot{q}^{i}}\Gamma^{k}-\frac{\partial \Gamma^{l}}{\partial q^{i}} \right]+ \frac{1}{2}\frac{\partial \Gamma^{l}}{\partial \dot{q}^{i}}\left( \frac{\partial^{2}F_{j}}{\partial q^{k}\partial\dot{q}^{l}}\dot{q}^{k}+\frac{\partial^{2}F_{j}}{\partial\dot{q}^{k}\partial\dot{q}^{l}}\Gamma^{k} \right)}
\\
\ds{=g_{il}\left[ \frac{1}{2}\frac{\partial^{2}\Gamma^{l}}{\partial q^{k}\partial \dot{q}^{j}}\dot{q}^{k}+\frac{1}{2}\frac{\partial^{2}\Gamma^{l}}{\partial\dot{q}^{k}\partial \dot{q}^{j}}\Gamma^{k}-\frac{\partial \Gamma^{l}}{\partial q^{j}} \right]+\frac{1}{2}\frac{\partial \Gamma^{l}}{\partial \dot{q}^{j}}\left( \frac{\partial^{2}F_{i}}{\partial q^{k}\partial\dot{q}^{l}}\dot{q}^{k}+\frac{\partial^{2}F_{i}}{\partial\dot{q}^{k}\partial\dot{q}^{l}}\Gamma^{k} \right)\, .}
\end{array}
\end{equation}
Using $(\ref{L3})$ and $a_{ij}=a_{ji}$ again we get
\begin{equation*}
\begin{array}{l}
\ds{\frac{1}{2}\frac{\partial \Gamma^{l}}{\partial \dot{q}^{i}}\left( \frac{\partial^{2}F_{j}}{\partial q^{k}\partial\dot{q}^{l}}\dot{q}^{k}+\frac{\partial^{2}F_{j}}{\partial\dot{q}^{k}\partial\dot{q}^{l}}\Gamma^{k} \right)=\frac{1}{2}\frac{\partial \Gamma^{l}}{\partial \dot{q}^{i}}\left( \frac{\partial F_{l}}{\partial q^{j}}-\frac{\partial F_{j}}{\partial q^{l}}-\frac{\partial F_{j}}{\partial \dot{q}^{k}}\frac{\partial \Gamma^{k}}{\partial\dot{q}^{l}} \right)}
\\
\ds{=\frac{1}{2}\frac{\partial \Gamma^{l}}{\partial \dot{q}^{i}}\left(-\frac{1}{2}\frac{\partial F_{l}}{\partial\dot{q}^{k}}\frac{\partial \Gamma^{k}}{\partial\dot{q}^{j}}-\frac{1}{2}\frac{\partial F_{j}}{\partial \dot{q}^{k}}\frac{\partial \Gamma^{k}}{\partial \dot{q}^{l}} \right)=-\frac{1}{4}g_{jk}\frac{\partial \Gamma^{k}}{\partial\dot{q}^{l}}\frac{\partial \Gamma^{l}}{\partial\dot{q}^{i}}-\frac{1}{4}\frac{\partial \Gamma^{l}}{\partial\dot{q}^{i}}\frac{\partial F_{l}}{\partial\dot{q}^{k}}\frac{\partial \Gamma^{k}}{\partial\dot{q}^{j}}\, .}
\end{array}
\end{equation*}
Since the last term is equal on both sides of~\eqref{eq:intermediateStep}, that is, 
$$\frac{\partial \Gamma^{l}}{\partial\dot{q}^{i}}\frac{\partial F_{l}}{\partial\dot{q}^{k}}\frac{\partial \Gamma^{k}}{\partial\dot{q}^{j}}=\frac{\partial \Gamma^{l}}{\partial\dot{q}^{j}}\frac{\partial F_{l}}{\partial\dot{q}^{k}}\frac{\partial \Gamma^{k}}{\partial\dot{q}^{i}},$$
we obtain the $\Phi$ condition
$$
g_{jl}\Phi^{l}_{i}=g_{il}\Phi^{l}_{j}, \hbox{ where } \Phi^{l}_{j}=\frac{\partial^{2}\Gamma^{l}}{\partial q^{k}\partial
    \dot{q}^{j}}\dot{q}^{k}
    +\frac{\partial^{2}\Gamma^{l}}{\partial \dot{q}^{k}\partial
    \dot{q}^{j}}\Gamma^{k}
    -2\frac{\partial\Gamma^{l}}{\partial q^{j}}
    -\frac{1}{2}\frac{\partial \Gamma^{l}}{\partial \dot{q}^{r}}\frac{\partial \Gamma^{r}}{\partial \dot{q}^{j}}.
$$

On the other hand, if we assume that the Helmholtz conditions are satisfied by $g_{ij}$ then there exists a local diffeomorphism $F(q, \dot{q})=(q^i, F_i(q, \dot{q}))$ such that $g_{ij}=\frac{\partial F_{i}}{\partial\dot{q}^{j}}$. Then $\Omega=-d(F^{*}\Theta_{Q})$ satisfies the conditions in Theorem \ref{theorem-crampin}. According to \cite{81Crampin}, $\Omega$ is given by
$$
\Omega=g_{ij}dq^{i}\wedge \nu^{j},
$$
where $\left\{dq^{j},\nu^{j}=d\dot{q}^{j}-\frac{1}{2}\frac{\partial \Gamma^{j}}{\partial
\dot{q}^{k}}dq^{k}\right\}$ is the dual basis to $\left\{\frac{\partial}{\partial q^{i}}, H_{i}=\frac{\partial}{\partial
q^{i}}+\frac{1}{2}\frac{\partial \Gamma^{k}}{\partial
\dot{q}^{i}}\frac{\partial}{\partial \dot{q}^{k}} \right\}$ 
and $H_{i}$ is the horizontal lift of
$\frac{\partial}{\partial q^{i}}$ with respect to the connection
defined by $\Gamma$. Since
$$
\Omega=-d(F_{i}dq^{i})=-\frac{\partial F_{i}}{\partial q^{j}}dq^{j}\wedge
dq^{i}-\frac{\partial F_{i}}{\partial \dot{q}^{k}}\left(\nu^{k}
+\frac{1}{2}\frac{\partial\Gamma^{k}}{\partial \dot{q}^{j}}dq^{j} \right)\wedge dq^{i}
$$
$$
=\underbrace{\left(-\frac{\partial F_{i}}{\partial q^{j}}-\frac{1}{2}\frac{\partial F_{i}}{\partial \dot{q}^{k}}
\frac{\partial\Gamma^{k}}{\partial
\dot{q}^{j}}\right)}_{-a_{ij}}
dq^{j}\wedge dq^{i}
+\underbrace{\frac{\partial F_{i}}{\partial \dot{q}^{k}}}_{g_{ik}}dq^{i}\wedge \nu^{k},
$$
we obtain $a_{ij}=a_{ji}$ and we can reverse the calculations in the above implication.

\begin{remark}
Analogous computations can be carried out for the equations in the time-dependent case, now using the local expression for $\Omega$ in \cite{84CPT}.
\end{remark}

\section*{Acknowledgements}

This work has been partially supported by MICINN (Spain)
 MTM2008-00689, MTM2010-21186-C02-01 and MTM2010-21186-C02-02; 2009SGR1338 from the Catalan government; the European project IRSES-project “GeoMech-246981 and the ICMAT Severo Ochoa project SEV-2011-0087.
MFP has been financially supported by a FPU scholarship from MECD.

\bibliographystyle{plain}
\bibliography{References}

\begin{thebibliography}{10}

\bibitem{AbMa}
R.~Abraham and J.~E. Marsden.
\newblock {\em Foundations of mechanics}.
\newblock Benjamin/Cummings Publishing Co. Inc. Advanced Book Program, Reading,
  Mass., 1978.
\newblock Second edition, revised and enlarged, With the assistance of Tudor
  Ra{\c{t}}iu and Richard Cushman.

\bibitem{AT1992}
I.~Anderson and G.~Thompson.
\newblock The inverse problem of the calculus of variations for ordinary
  differential equations.
\newblock {\em Mem. Amer. Math. Soc.}, 98(473):vi+110, 1992.

\bibitem{Anderson}
I.~M. Anderson and T.~Duchamp.
\newblock On the existence of global variational principles.
\newblock {\em Amer. J. Math.}, 102(5):781--868, 1980.

\bibitem{2010BalNar}
P.~Balseiro and L.~C. Garc{\'{\i}}a-Naranjo.
\newblock Gauge transformations, twisted {P}oisson brackets and
  {H}amiltonization of nonholonomic systems.
\newblock {\em Arch. Ration. Mech. Anal.}, 205(1):267--310, 2012.

\bibitem{93BlochCrouch}
A.~M. Bloch and P.~E. Crouch.
\newblock Nonholonomic and vakonomic control systems on {R}iemannian manifolds.
\newblock In {\em Dynamics and control of mechanical systems ({W}aterloo, {ON},
  1992)}, volume~1 of {\em Fields Inst. Commun.}, pages 25--52. Amer. Math.
  Soc., Providence, RI, 1993.

\bibitem{2009BlochFerMestdag}
A.~M. Bloch, O.~E. Fernandez, and T.~Mestdag.
\newblock Hamiltonization of nonholonomic systems and the inverse problem of
  the calculus of variations.
\newblock {\em Rep. Math. Phys.}, 63(2):225--249, 2009.

\bibitem{1992Cantrijn}
F.~Cantrijn, M.~de~Le{\'o}n, and E.~A. Lacomba.
\newblock Gradient vector fields on cosymplectic manifolds.
\newblock {\em J. Phys. A}, 25(1):175--188, 1992.

\bibitem{2002Cortes}
J.~Cort{\'e}s, M.~de~Le{\'o}n, D.~{Mart{\'i}n de Diego}, and S.~Mart{\'i}nez.
\newblock Geometric description of vakonomic and nonholonomic dynamics.
  {C}omparison of solutions.
\newblock {\em SIAM J. Control Optim.}, 41(5):1389--1412, 2002.

\bibitem{1990Courant}
T.~Courant.
\newblock Tangent {D}irac structures.
\newblock {\em J. Phys. A}, 23(22):5153--5168, 1990.

\bibitem{81Crampin}
M.~Crampin.
\newblock On the differential geometry of the {E}uler-{L}agrange equations, and
  the inverse problem of {L}agrangian dynamics.
\newblock {\em J. Phys. A}, 14(10):2567--2575, 1981.

\bibitem{CPST1999}
M.~Crampin, G.~E. Prince, W.~Sarlet, and G.~Thompson.
\newblock The inverse problem of the calculus of variations: separable systems.
\newblock {\em Acta Appl. Math.}, 57(3):239--254, 1999.

\bibitem{84CPT}
M.~Crampin, G.~E. Prince, and G.~Thompson.
\newblock A geometrical version of the {H}elmholtz conditions in time-dependent
  {L}agrangian dynamics.
\newblock {\em J. Phys. A}, 17(7):1437--1447, 1984.

\bibitem{94CSMBP}
M.~Crampin, W.~Sarlet, E.~Mart{\'i}nez, G.~B. Byrnes, and G.~E. Prince.
\newblock Towards a geometrical understanding of {D}ouglas' solution of the
  inverse problem of the calculus of variations.
\newblock {\em Inverse Problems}, 10(2):245--260, 1994.

\bibitem{2005LeMaMa}
M.~de~Le{\'o}n, J.~C. Marrero, and E.~Mart{\'{\i}}nez.
\newblock Lagrangian submanifolds and dynamics on {L}ie algebroids.
\newblock {\em J. Phys. A}, 38(24):R241--R308, 2005.

\bibitem{1989LeRo}
M.~de~Le{\'o}n and P.~R. Rodrigues.
\newblock {\em Methods of differential geometry in analytical mechanics},
  volume 158 of {\em North-Holland Mathematics Studies}.
\newblock North-Holland Publishing Co., Amsterdam, 1989.

\bibitem{Douglas}
J.~Douglas.
\newblock Solution of the inverse problem of the calculus of variations.
\newblock {\em Trans. Amer. Math. Soc.}, 50:71--128, 1941.

\bibitem{2008BlochFer}
O.~E. Fernandez and A.~M. Bloch.
\newblock Equivalence of the dynamics of nonholonomic and variational
  nonholonomic systems for certain initial data.
\newblock {\em J. Phys. A}, 41(34):344005, 20, 2008.

\bibitem{GrKaGra2004}
K.~Grabowska, J.~Grabowski, and P.~Urba{\'n}ski.
\newblock A{V}-differential geometry: {P}oisson and {J}acobi structures.
\newblock {\em J. Geom. Phys.}, 52(4):398--446, 2004.

\bibitem{Grabowski}
J.~Grabowski and P.~Urba{\'n}ski.
\newblock Tangent lifts of {P}oisson and related structures.
\newblock {\em J. Phys. A}, 28(23):6743--6777, 1995.

\bibitem{GM1999}
J.~Grifone and Z.~Muzsnay.
\newblock Sur le probl{\`e}me inverse du calcul des variations: existence de
  lagrangiens associ{\'e}s {\`a} un spray dans le cas isotrope.
\newblock {\em Ann. Inst. Fourier (Grenoble)}, 49(4):1387--1421, 1999.

\bibitem{GuiStern}
V.~Guillemin and S.~Sternberg.
\newblock {\em Geometric asymptotics}.
\newblock American Mathematical Society, Providence, R.I., 1977.
\newblock Mathematical Surveys, No. 14.

\bibitem{82Henneaux}
M.~Henneaux.
\newblock On the inverse problem of the calculus of variations.
\newblock {\em J. Phys. A}, 15(3):L93--L96, 1982.

\bibitem{84Henneaux}
M.~Henneaux.
\newblock On the inverse problem of the calculus of variations in field theory.
\newblock {\em J. Phys. A}, 17(1):75--85, 1984.

\bibitem{91IbortMarin}
L.~A. Ibort and J.~Mar{\'i}n-Solano.
\newblock On the inverse problem of the calculus of variations for a class of
  coupled dynamical systems.
\newblock {\em Inverse Problems}, 7(5):713--725, 1991.

\bibitem{2008KP}
O.~Krupkov{\'a} and G.~E. Prince.
\newblock Second order ordinary differential equations in jet bundles and the
  inverse problem of the calculus of variations.
\newblock In {\em Handbook of global analysis}, pages 837--904, 1215. Elsevier
  Sci. B. V., Amsterdam, 2008.

\bibitem{LiMarle}
P.~Libermann and C.-M. Marle.
\newblock {\em Symplectic geometry and analytical mechanics}, volume~35 of {\em
  Mathematics and its Applications}.
\newblock D. Reidel Publishing Co., Dordrecht, 1987.
\newblock Translated from the French by Bertram Eugene Schwarzbach.

\bibitem{2006MaMaMa}
J.~C. Marrero, D.~Mart{\'{\i}}n~de Diego, and E.~Mart{\'{\i}}nez.
\newblock Discrete {L}agrangian and {H}amiltonian mechanics on {L}ie groupoids.
\newblock {\em Nonlinearity}, 19(6):1313--1348, 2006.

\bibitem{2001MaWe}
J.~E. Marsden and M.~West.
\newblock Discrete mechanics and variational integrators.
\newblock {\em Acta Numer.}, 10:357--514, 2001.

\bibitem{2010Crampin}
T.~Mestdag and M.~Crampin.
\newblock Nonholonomic systems as restricted {E}uler-{L}agrange systems.
\newblock {\em Balkan J. Geom. Appl.}, 15(2):78--89, 2010.

\bibitem{1990MFLMR}
G.~Morandi, C.~Ferrario, G.~{Lo Vecchio}, G.~Marmo, and C.~Rubano.
\newblock The inverse problem in the calculus of variations and the geometry of
  the tangent bundle.
\newblock {\em Phys. Rep.}, 188(3-4):147--284, 1990.

\bibitem{ranada}
M.~F. Ra{\~n}ada.
\newblock Time-dependent {L}agrangian systems: a geometric approach using
  semibasic forms.
\newblock {\em J. Phys. A}, 23(15):3475--3482, 1990.

\bibitem{2012Rossi}
O.~Rossi and J.~Musilov{\'a}.
\newblock On the inverse variational problem in nonholonomic mechanics.
\newblock {\em Commun. Math.}, 20(1):41--62, 2012.

\bibitem{82Sarlet}
W.~Sarlet.
\newblock The {H}elmholtz conditions revisited. {A} new approach to the inverse
  problem of {L}agrangian dynamics.
\newblock {\em J. Phys. A}, 15(5):1503--1517, 1982.

\bibitem{SCM1998}
W.~Sarlet, M.~Crampin, and E.~Mart{\'i}nez.
\newblock The integrability conditions in the inverse problem of the calculus
  of variations for second-order ordinary differential equations.
\newblock {\em Acta Appl. Math.}, 54(3):233--273, 1998.

\bibitem{Sonin1886}
N.J. Sonin.
\newblock About determining maximal and minimal properties of plane curves.
\newblock {\em Warsawskye Universitetskye Izvestiya}, 1-2:1--68, 1886.

\bibitem{79Takens}
F.~Takens.
\newblock A global version of the inverse problem of the calculus of
  variations.
\newblock {\em J. Differential Geom.}, 14(4):543--562 (1981), 1979.

\bibitem{TuH}
W.~M. Tulczyjew.
\newblock Les sous-vari{\'e}t{\'e}s lagrangiennes et la dynamique
  hamiltonienne.
\newblock {\em C. R. Acad. Sci. Paris S{\'e}r. A-B}, 283(1):Ai, A15--A18, 1976.

\bibitem{Tu}
W.~M. Tulczyjew.
\newblock Les sous-vari{\'e}t{\'e}s lagrangiennes et la dynamique lagrangienne.
\newblock {\em C. R. Acad. Sci. Paris S{\'e}r. A-B}, 283(8):Av, A675--A678,
  1976.

\bibitem{Uchino}
K.~Uchino.
\newblock Lagrangian calculus on {D}irac manifolds.
\newblock {\em J. Math. Soc. Japan}, 57(3):803--825, 2005.

\bibitem{Vaisman}
I.~Vaisman.
\newblock {\em Symplectic geometry and secondary characteristic classes},
  volume~72 of {\em Progress in Mathematics}.
\newblock Birkh{\"a}user Boston Inc., Boston, MA, 1987.

\bibitem{Vaisman2}
I.~Vaisman.
\newblock {\em Lectures on symplectic and {P}oisson geometry}.
\newblock Textos de Matem{\'a}tica. S{\'e}rie B [Texts in Mathematics. Series
  B], 23. Universidade de Coimbra Departamento de Matem{\'a}tica, Coimbra,
  2000.
\newblock Notes edited by Fani Petalidou.

\bibitem{Helmholtz1887}
H.~von Helmholtz.
\newblock Ueber die physikalische bedeutung des prinicips der kleinsten
  wirkung.
\newblock {\em Journal f{\"u}r die reine und angewandte Mathematik},
  100:137--166, 1887.

\bibitem{Weinstein}
A.~Weinstein.
\newblock Symplectic manifolds and their {L}agrangian submanifolds.
\newblock {\em Advances in Math.}, 6:329--346, 1971.

\bibitem{Yano}
K.~Yano and S.~Ishihara.
\newblock {\em Tangent and cotangent bundles: differential geometry}.
\newblock Marcel Dekker, Inc., New York, 1973.
\newblock Pure and Applied Mathematics, No. 16.

\end{thebibliography}
\end{document}